\documentclass [oneside,11pt] {amsart} 

\usepackage{amsfonts}
\usepackage{amssymb}
\usepackage{mathtools}
\usepackage{hyperref}
\usepackage{comment}
\usepackage{color,xcolor}
\usepackage{a4}
\usepackage{geometry}
\usepackage{todonotes}
\usepackage{comment}
\usepackage{bigints}
\usepackage{tikz}
\usetikzlibrary{calc}
\usetikzlibrary{arrows.meta}

\usepackage[normalem]{ulem}
\usepackage{soul,xcolor}
\newcommand{\redsout}{\bgroup\markoverwith{\textcolor{red}{\rule[0.5ex]{2pt}{.4pt}}}\ULon}

\numberwithin{equation}{section}
\newtheorem{theorem}{Theorem}[section]

\newtheorem{prop}{Proposition}[section]
\newtheorem{lemma}{Lemma}[section]

\newtheorem{remark}{Remark}[section]

\newcommand{\R}{\mathbb{R}}
\newcommand{\p}{\partial}

\newcommand{\rd}{\,d}
\renewcommand{\Re}{\mathrm{Re}}
\renewcommand{\Im}{\mathrm{Im}}
\newcommand{\supp}{\mathrm{supp}}
\newcommand{\dist}{\mathrm{dist}}

\newcommand{\LV}{\left|}
\newcommand{\RV}{\right|}
\newcommand{\LC}{\left(}
\newcommand{\RC}{\right)}

\newcommand{\LA}{\left<}
\newcommand{\RA}{\right>}

\newcommand{\RNum}[1]{\uppercase\expandafter{\romannumeral #1\relax}}

\usepackage[framemethod=tikz]{mdframed}
\usepackage{url}
\definecolor{mycolor}{rgb}{0.122, 0.435, 0.698}
\definecolor{aliceblue}{rgb}{0.94, 0.97, 1.0}

\newmdenv[innerlinewidth=0.5pt, roundcorner=4pt,linecolor=mycolor,innerleftmargin=6pt,
innerrightmargin=6pt,innertopmargin=6pt,innerbottommargin=6pt]{mybox}
\newmdenv[backgroundcolor=aliceblue,innerlinewidth=0.5pt, roundcorner=4pt,linecolor=mycolor,innerleftmargin=6pt,
innerrightmargin=6pt,innertopmargin=6pt,innerbottommargin=6pt]{mybox1}

\begin{document}

\title[Partial data inverse problems for the magnetic Schr\"odinger equation]{Partial data inverse problems for the nonlinear magnetic Schr\"odinger equation}

\author[Lai]{Ru-Yu Lai}
\address{School of Mathematics, University of Minnesota, Minneapolis, MN 55455, USA}
\email{rylai@umn.edu}

\author[Uhlmann]{Gunther Uhlmann}
\address{Department of Mathematics, University of Washington, Seattle, WA 98195, USA}
\email{gunther@math.washington.edu}

\author[Yan]{Lili Yan}
\address
{School of Mathematics\\
University of Minnesota\\ 
Minneapolis, MN 55455, USA }

\email{lyan@umn.edu}

\maketitle
 \setcounter{tocdepth}{1}
\tableofcontents
 \setcounter{tocdepth}{2}

\begin{abstract}
In this paper, we study the partial data inverse problem for nonlinear magnetic Schr\"odinger equations.
We show that the knowledge of the Dirichlet-to-Neumann map, measured on an arbitrary part of the boundary, determines the time-dependent linear coefficients, electric and magnetic potentials, and nonlinear coefficients, provided that the divergence of the magnetic potential is given. Additionally, we also investigate both the forward and inverse problems for the linear magnetic Schr\"odinger equation with a time-dependent leading term. In particular, all coefficients are uniquely recovered from boundary data. 
\end{abstract}

\section{Introduction}
Let $T>0$ be a real number and $\Omega$ be a bounded domain in $\R^n$, $n\geq 2$, with smooth boundary $\p\Omega$. 
We investigate the partial data inverse problem for the nonlinear time-dependent magnetic Schr\"odinger equation: 
\begin{align}\label{eq_IBVP}
\left\{\begin{array}{rrll}
(i\p_t + \Delta_A  + q) u&=& N(t,x,u,\overline{u})  &\hbox{ in }Q,\\
u&=& f &\hbox{ on } \Sigma,\\
u&=& 0 &\hbox{ on } \{t=0\}\times \Omega,
\end{array}  \right.
\end{align}
where the set $Q$ with its lateral boundary $\Sigma$ are defined by 
$$
Q:=(0,T)\times\Omega \quad \hbox{ and }\quad \Sigma: = (0,T)\times \p \Omega,
$$
and the operator $\Delta_A$ is given by $\Delta_{A} : =(\nabla+ iA)^2 =\Delta + 2iA\cdot \nabla + i\nabla\cdot A - |A|^2.$ Throughout the paper, we denote by $\p_t$ the derivative with respect to $t$, and denote by $\nabla$ and $\nabla\cdot$ the gradient and divergent with respect to $x\in \R^n$, respectively.
Here $A=A (t,x)$ is the real-valued magnetic potential in $\R^n$ and $q=q(t,x)$ is the real-valued electric potential. 
The nonlinearity we consider is of the following form:  
\begin{equation}\label{eq_1_nonlinear}
N(t,x,u,\overline{u})=\sum\limits_{\substack{\sigma\geq 1\\
2\le \sigma+\beta\le M}}
B_{\sigma\beta}(t,x)u^\sigma\overline{u}^{\beta},
\end{equation}
where $\sigma$ and $\beta$ are non-negative integers, for a fixed integer $M\geq 2$.
For instance, when $\beta=0$, we have polynomial-type nonlinear term
$ 
N(t,x,u,\overline{u}) =\sum_{\sigma=2}^M B_{\sigma0} u^\sigma.
$ 
In particular, the equation in \eqref{eq_IBVP} with $A=0$, $N(t,x,u,\overline{u}) = B_{21}|u|^2u$ is known as the time-dependent Gross-Pitaevskii equation characterizing the ground-state single-particle wavefunction in a Bose–Einstein condensate.

\vskip.1cm
The main objective of this paper is to recover both the linear and nonlinear coefficients 
from partial boundary data.
Let $\Gamma$ be an open proper subset of the boundary $\p \Omega$. We denote 
$$
\Sigma^\sharp := (0,T)\times \Gamma.
$$
The data we utilize is the the partial Dirichlet-to-Neumann (DN) map defined by
\begin{align*}
\Lambda^\sharp_{A,q,N}: \mathcal{D}_\delta(\Sigma)&\to {{H}}^{2\kappa-\frac{3}{2}}(\Sigma^\sharp)\\
f&\to \nu\cdot(\nabla+ iA)u|_{\Sigma^\sharp},
\end{align*}
where $u$ is the unique solution to \eqref{eq_IBVP}, and $\nu$ is the unit outer normal to $\p\Omega$, see \eqref{eq_defn_D} for the  definition of $\mathcal{D}_\delta(\Sigma)$ with $\delta>0$. As shown in Proposition \ref{prop_nonlinear_wellposed}, the initial boundary value problem (IBVP) \eqref{eq_IBVP} is well-posed for small boundary data $f$ and thus the DN map here is well-defined for small $\delta$.

\vskip.1cm

The inverse problem we are interested in is whether the partial DN map $\Lambda^\sharp_{A,q,N}$ uniquely determines the linear coefficients $A$, $q$ and the nonlinear coefficients $B_{\alpha\beta}$ in $Q$, respectively? However, it is known that due to a gauge invariant, one can only expect to recover $(A,q)$ uniquely provided that $\nabla\cdot A$ is given \cite{KianSoccorsi}. Therefore, under certain assumptions on the coefficients, we will answer this question by showing that $\Lambda^\sharp_{A_1,q_1,N_1}=\Lambda^\sharp_{A_2,q_2,N_2}$ implies $A_1=A_2$, $q_1=q_2$, and $B_{1,\sigma\beta}=B_{2,\sigma\beta}$ in $Q$ for all $2\le \sigma+\beta\le M$ with $\sigma\geq 1$.

For the stationary linear Schr\"odinger equation, $-\Delta u+ Vu=0$, the inverse problem is related to the well-known Calder\'on problem, which has intensive results for the unique reconstruction of $V$ with full data and partial data, see review paper \cite{uhlmann2009calderon} and the reference therein.
Moreover, the inverse boundary value problem for the stationary magnetic Schr\"odinger equation, $-\Delta_A u+V u=0$, has been considered in, for instance, \cite{haberman2016unique, krupchyk2014uniqueness, NSU1995, panchenko2002inverse,salo2004inverse, sun_inverse_1993}.
Due to a gauge invariance, only $\text{curl} A$ can be uniquely recovered.

\vskip.1cm

Inverse problems for nonlinear partial differential equations (PDEs) consider the recovery of nonlinear coefficients from measurements. This direction of research has received great attention since it find applications in various physical phenomena, such as Bose-Einstein condensates, the propagation of light in nonlinear optical fibers and plasmas
\cite{Carlo, malomed_nonlinear_2005, pitaevskii_bose-einstein_2003}. 
To tackle the nonlinear term in inverse problems, a standard approach of first-order linearization, first introduced by Isakov \cite{Isakov93}, is to linearize the DN map such that one can apply the existing result for the inverse problem for the linearized equation. Later, a second-order linearization was utilized in \cite{SunU97} and the higher-order linearization method was considered in \cite{KLU18} by adequately utilizing the nonlinear feature in the equation. This approach has been applied to study inverse boundary value problems for different types of nonlinear PDEs, see for example, hyperbolic equations \cite{CLOP,KLU18,LUW2018}, parabolic equations \cite{feizmohammadi_inverse_2022}, elliptic equations \cite{FL2019, Kang2002, KU201909, KU2019, LaiTingZhou20, LLLS201903, LLLS201905}, and transport equations \cite{LaiUhlmannYang, LaiUhlmannZhou22, LY2023, LiOuyang2022}.

\vskip.1cm

In this work, we focus on the dynamic nonlinear magnetic Schr\"odinger equation with time-dependent coefficients.
For the linear equation, the stability estimate for time-independent potentials were derived in \cite{AM, Bella, BC, BD, BKS, CS} with dynamic DN map. In particular, the first result for the unique determination of the time-dependent electric and magnetic potentials from the DN map was showed by Eskin in \cite{eskin_inverse_2008}. Moreover, the stability estimate for the time-dependent ones were given in \cite{CKS,KianSoccorsi,KianTetlow} when the data is measured on the full boundary and  \cite{Bella-Fraj2020} when partial measurement is available.
For the Schr\"odinger equation with a power-type nonlinearity, stable recovery of the time-dependent nonlinearity was considered in \cite{lai_partial_2023} with partial boundary data. From the knowledge of source-to-solution map, the work \cite{lassas_inverse_2024} investigated unique determination of the time-dependent linear and nonlinear potentials.

\subsection{Main results}
The main results of this paper consist of the following two different settings. In the first setting, the aim is to uniquely recover electromagnetic potentials and nonlinear coefficients in the magnetic Schr\"odinger equation \eqref{eq_IBVP}. The second one is considering the uniqueness result for the leading coefficient $c(t)$ and the linear coefficients $A,\,q$ appearing in $(i\p_t+c(t)\Delta_A+q) v =0$. 

\vskip.1cm

Let $\mathcal{O}\subset \overline{\Omega}$ be a nonempty open neighborhood of $\p\Omega$. 
We now state our first main result.
\begin{theorem}\label{thm_main}
Let $\Omega$ be a bounded and simply connected domain in $\R^n$ with smooth boundary $\p\Omega$, $n\geq 2$. Suppose that $A_j \in C^\infty(\overline{Q};\R^n)$ and $q_j\in C^\infty(\overline{Q};\R)$ satisfy $A_j=0$ and $q_j=0$ in $(0,T)\times\mathcal{O}$
for $j=1,\,2$, and $\nabla \cdot A_1 = \nabla\cdot A_2$ in $Q$.
Moreover, assume that $B_{j,\sigma\beta}\in C^{\infty}(\overline{Q};\R)$ in $N_j$, defined in \eqref{eq_1_nonlinear}, satisfy $B_{1,\sigma\beta}-B_{2,\sigma\beta}=0$ on $(0,T)\times \mathcal{O}$ for $j=1,\,2$.
If $\Lambda^\sharp_{A_1,q_1,N_1} =\Lambda^\sharp_{A_2,q_2,N_2}$, then 
$$A_1 = A_2,\quad q_1 = q_2,\quad \hbox{and }\quad   B_{1,\sigma\beta} = B_{2,\sigma\beta}\quad \hbox{ in }\quad Q,$$ 
for all integers $\sigma$ and $\beta$ such that $1\le\sigma$, $0\le\beta$ and $2\le\sigma+\beta\le M$.
\end{theorem}

We note that the support of these coefficients is essential for the purpose of constructing geometric optics (GO) solutions with higher-order regularity for the linear magnetic Schr\"odinger equation and recovering the coefficients with partial data due to a unique continuation principle, see Section~\ref{section_geo_sol}-\ref{sec:nonlinear} for more illustrations. Also, the smooth condition is not essential since it can be relaxed by requiring $A,q,B\in C^K(\overline{Q})$ for sufficiently large $K$ depending on $n$.
Additionally, the condition $\sigma \geq 1$ is required to cancel out the phase in the product of GO solutions, see Section~\ref{sec:nonlinear} for details.

In the second setting,
we study the inverse problem for the linear magnetic Schr\"odinger equation with a time-dependent leading term $c(t)$:
\begin{align}\label{eq_app_IBVP}
\left\{\begin{array}{rrll}
(i\p_t+c(t)\Delta_A+q) v  &=& 0  &\hbox{ in } Q,\\
v  &=& f  &\hbox{ on } \Sigma,\\
v &=& 0  &\hbox{ on } \{t =0\}\times\Omega.
\end{array}  \right.
\end{align}
{Throughout the paper, we assume $c(t) >c_0>0$.} As mentioned in Proposition \ref{prop_app_wellpose}, the initial boundary value problem \eqref{eq_app_IBVP} is well-posed for $f\in \mathcal{H}^{\frac{5}{2}}_0(\Sigma)$.  
We define the DN map associated to \eqref{eq_app_IBVP}:
\begin{align*}
\Lambda_{c,A,q}: \mathcal{H}^{\frac{5}{2}}_0(\Sigma)&\to L^2(0,T;H^{\frac{1}{2}}(\Sigma))\\
f&\to c (t)\nu\cdot(\nabla+ iA)v|_{\Sigma},
\end{align*}
where $v\in H^{1,2}(Q)$ is the unique solution to \eqref{eq_app_IBVP}.

\vskip.1cm

The second main result, shown in the Appendix \ref{sec:appendix}, states that the time-dependent leading term and the electromagnetic potentials can be recovered from the DN map measured on the whole boundary, respectively.

\begin{theorem}\label{thm_app_c}
Let $\Omega$ be a bounded and simply connected domain in $\R^n$ with smooth boundary $\p\Omega$, $n\geq 2$. 
For $j = 1,\,2$, suppose that $c_j\in C^6([0,T];\R)$, $A_j\in C^6(\overline{Q};\R^n)$, and $q_j\in C^5(\overline{Q};\R)$ satisfy $c(t) >c_0>0$.
If $\Lambda_{c_1,A_1,q_1} =\Lambda_{c_2,A_2,q_2}$, then 
$$
c_1(t) = c_2(t) \quad \hbox{ for }t\in(0,T).
$$

Furthermore, if we assume $\nabla\cdot A_1 = \nabla\cdot A_2$ in $Q$ and $\p_x^\alpha A_1 =\p_x^\alpha A_2$ on $\Sigma$ for $\alpha\in \mathbb{N}^n$, $|\alpha|\le 5$, then we have 
$$
A_1 = A_2 \quad \hbox{ and }\quad q_1 = q_2\quad \hbox{ in }\quad Q.
$$
\end{theorem}

The main contributions of this paper are as follows. 
Our first theorem generalizes \cite{lai_partial_2023}, where a power type nonlinear term was studied, to a mixed higher order nonlinearity. The crucial idea is constructing GO solutions to the linear magnetic Schr\"odinger equations with carefully chosen leading amplitudes concentrating near designed straight lines. Through their interactions, we are able to decouple different degrees and types of nonlinear terms in an integral identity. 
{To make the idea clearer, let us assume that $A$ and $q$ have been recovered and  take $\sigma+\beta=2$ for example. Let $f_1,\, f_2\in \widetilde{\mathcal{H}}^{2\kappa + \frac{3}{2}}_0(\Sigma)$. For small enough $\varepsilon=(\varepsilon_1,\, \varepsilon_2)$, we have $f=\varepsilon_1 f_1+\varepsilon_2 f_2\in \mathcal{D}_\delta(\Sigma)$. We denote $u_{j,\varepsilon}$ be the unique solution to \eqref{eq_IBVP} with $B_{j,\sigma\beta}$ subject to the boundary condition $f_\varepsilon$ for $j=1,\,2$. Thus, 
$$
    w_j:=\p_{\varepsilon_1}\p_{\varepsilon_2} u_{j,\varepsilon}|_{\varepsilon=0} 
$$
solves a linearized equation of \eqref{eq_IBVP}:
$$
(i\p_t+\Delta_A+q) w_j =
2B_{j,20}v_1v_2+ B_{j,11}(v_1\overline{v_2}+\overline{v_1}v_2) $$ 
with trivial boundary and initial data,
where $v_\ell$ satisfies $(i\p_t+\Delta_A+q) v_\ell=0$ with $v_\ell|_{\Sigma} = f_\ell$ for $\ell=1,\,2$.
From mixed derivative of the DN maps yields $$\p_\nu w_1=\p_\nu w_2\quad \hbox{ on }\Sigma^\sharp.$$ 
By subtracting the equations for $w_j$ above and 
integrating the resulting equation with $v_0$, which is the solution to the adjoint problem, we obtain
\begin{align*} 
0 = \bigintsss_Q  \LC 2(B_{1,20}-B_{2,20}) v_1v_2+ (B_{1,11}-B_{2,11})(v_1\overline{v_2}+\overline{v_1}v_2) \RC\overline{v_0} \rd x \rd t.
\end{align*}
With suitably chosen GO solutions, only the phase of the product of GO solutions $v_1v_2\overline{v_0}$ vanishes such that one can uniquely recover $B_{j,20}$ at each intersection point of the concentrating amplitudes. The other coefficient $B_{j,11}$ can be determined similarly by choosing different set of GO solutions with vanishing phase on $v_1\overline{v_2v_0}$. 
We refer to Section~\ref{sec:nonlinear} for more comprehensive discussions and Proposition~\ref{prop_int_id_B_m} for higher order cases $\sigma+\beta\geq 3$.}

Second, we prove the well-posed result for the linear magnetic Schr\"odinger equation with leading time-dependent coefficient. We also show that both leading coefficient and linear electromagnetic coefficients can be uniquely recovered. This extends earlier results in \cite{Bella-Fraj2020, KianSoccorsi} for $c=1$.

\subsection{Outline}
The rest of the paper is organized as follows. Section~\ref{sec:preliminary} introduces fundamental spaces and notations, followed by the discussion of well-posed results for IBVP for both linear and nonlinear magnetic Schr\"odiner equations. 
In Section~\ref{section_geo_sol}, we construct two types of geometric optics (GO) solutions of higher order regularity to the linear magnetic Schr\"odinger equation with \textit{non-localized amplitudes} and \textit{localized amplitudes}. 
Section~\ref{sec:linear} discusses unique determination of the electric and magnetic potentials by applying the first type of GO solutions. 
The nonlinear coefficient is recovered in Section~\ref{sec:nonlinear} with the localized GO solutions.
Finally, an inverse problem for the linear magnetic Schr\"odinger equation \eqref{eq_app_IBVP} with time-dependent coefficients is investigated in Appendix~\ref{sec:appendix}.

\vskip1cm

\section{Preliminary results and the forward problem}\label{sec:preliminary}
We will start by introducing several function spaces and notations that will be used throughout this paper. As preparation for the study of the well-posedness problem for the nonlinear magnetic Schr\"odinger equation in \eqref{eq_IBVP}, we show there exists a unique solution for a linear magnetic Schr\"odinger equation in a suitable space beforehand.

\subsection{Function spaces and notations}\label{section_notation}
Let $Z$ be a Banach space with norm $\|\cdot\|$. The space $L^2(0,T;Z)$ consists of all measurable functions $u:[0,T]\rightarrow Z$ with norm
$$
\|u\|_{L^2(0,T;Z)} :=\LC\int^T_0 \|u(t,\cdot)\|_{Z}^2dt\RC^{1/2}<\infty.
$$
Let $r$ and $s$ be non-negative real numbers and let the notation $X$ be either $\Omega$, $\p \Omega$ or $\Gamma$. We define the following space:
\[
H^{r,s}((0,T)\times X): = H^r(0,T;L^2(X))\cap L^2(0,T;H^s(X)),
\]
which is a Hilbert space equipped with the norm 
\[
\|u\|_{H^{r,s}((0,T)\times X)}:=\|u\|_{H^r(0,T;L^2(X))} +\|u\|_{L^2(0,T;H^s(X))}<\infty.
\]
We write $H^{r,s}(Q)$, $H^{r,s}(\Sigma)$, $H^{r,s}(\Sigma^\sharp)$ instead of $H^{r,s}((0,T)\times \Omega)$, $H^{r,s}((0,T)\times\p \Omega)$, $H^{r,s}((0,T)\times\p \Gamma)$, respectively. 
Based on \cite[Proposition 2.3, Chapter 4]{lions_non-homogeneous_1972}, 
we may identify $H^{r,r}(Q)$, $H^{r,r}(\Sigma)$, $H^{r,r}(\Sigma^\sharp)$ with the standard Sobolev space $H^r(Q)$, $H^r(\Sigma)$, $H^{r}(\Sigma^\sharp)$, respectively, and their equipped norms are equivalent, i.e.,
\[
\|\cdot\|_{H^{r,r}(Q)} \sim \|\cdot\|_{H^{r}(Q)}, \quad \|\cdot\|_{H^{r,r}(\Sigma)} \sim \|\cdot\|_{H^{r}(\Sigma)}, \quad \|\cdot\|_{H^{r,r}(\Sigma^\sharp)} \sim \|\cdot\|_{H^{r}(\Sigma^\sharp)}.
\]

Let $k\in \mathbb{N}:=\{1,2,\cdots\}$ 
and for $\tau =0,\, T$, we introduce the spaces 
\[
\mathcal{H}_\tau^{k}(Q): = \{g\in H^{k}(Q): \p_t^ \ell  g(\tau,\cdot)= 0\quad \hbox{in }\Omega \hbox{ for } \ell \le k-1,\,  \ell \in \mathbb{N}\},
\]
and 
\[
\widetilde{\mathcal{H}}_\tau^{k+\frac{1}{2}}(\Sigma)
: = \{g\in H^{k+\frac{1}{2}}(\Sigma):  \p_t^ \ell  g(\tau,\cdot)= 0\quad\hbox{on }\p\Omega \hbox{ for }  \ell  \le k,\,  \ell \in \mathbb{N}\}.
\]

In addition, following the notations in \cite[Remark 9.5, Section 9.4, Chapter 1]{lions_non-homogeneous_1972-1}, we also define the space $H^{1}(0,T;H^{-1}(\Omega))$ equipped with the norm 
\[
\|u\|_{H^{1}(0,T;H^{-1}(\Omega))}:=\|u\|_{L^2(0,T;H^{-1}(\Omega))} + \|\p_tu\|_{L^2(0,T;H^{-1}(\Omega))}<\infty.
\]
Finally, for the purpose of establishing the well-posed results later, for $\delta>0$ and $\kappa>\frac{n+1}{2}$, we denote $\mathcal{D}_\delta(\Sigma)$ to be the $\delta$-neighborhood of the origin in the space $\widetilde{\mathcal{H}}_0^{2\kappa+\frac{3}{2}}(\Sigma)$:
\begin{align}\label{eq_defn_D}
\mathcal{D}_\delta(\Sigma): = \{f\in \widetilde{\mathcal{H}}^{2\kappa + \frac{3}{2}}_0(\Sigma): \|f\|_{H^{2\kappa+\frac{3}{2}}(\Sigma)}<\delta\}.
\end{align}

\subsection{Well-posedness for the linear magnetic Schr\"odinger equation (MSE)} 
We will prove that there exists a unique solution to the IBVP for the linear  MSE with higher order energy estimate. This will be the basic fact needed later in the proof of well-posedness problem for the nonlinear magnetic Schr\"odinger equation.
In addition, the higher-order energy estimate also plays an essential role in constructing geometric optics solutions with continuous reminder terms.

Throughout this section, we assume that $c(t)\in C^1([0,T];\R)$ satisfies $c(t)>c_0>0$ for some positive constant $c_0$ and suppose that $A\in C^2(\overline{Q};\R^n)$, and $q\in C^1(\overline{Q};\R)$, unless otherwise indicated.
In what follows, the notations $p\lesssim q$ or $q \gtrsim p$ will be used to indicate $p\le C q$, where $C$ is a positive constant depending only on $n$, $c$, $A$, $q$, $B_{\sigma\beta}$, $\Omega$, or $T$.

Let the notation $Y$ be either $\Omega$, $\Sigma$ or $Q$. We denote the natural $L^2$ product for scalar/vector-valued functions by 
\[
\LA U,V \RA_{L^2(Y)} =\sum_{i=1}^n\LA u_i,v_i \RA_{L^2(Y)}, \quad \hbox{where}\quad \LA u_i, v_i\RA_{L^2(Y)}= \int_Y u_i \overline{v_i}\,,\quad 
\]
for $U=(u_1,\cdots,u_n),\, V=(v_1,\cdots,v_n)\in (L^2(Y))^n$,
with the induced $L^2$-norm  $\|\cdot\|_{L^2(Y)}$ on $Y$.

We prove the energy estimate for a more general linear magnetic Schr\"odinger of the following form
\[
L_{c,A,q}v: = i\p_tv + c(t)\Delta_{A(t,x)}v+ q(t,x)v.
\]
Here $\Delta_A = \nabla_A\cdot\nabla_A$,  with the magnetic gradient $\nabla_A$ given by
$$
\nabla_A u(t,x) := (\nabla+iA(t,x))u(x)\quad\hbox{ for }u \in H^1(\Omega),\quad (t,x)\in Q.
$$
When $c(t)= 1$, we denote $L_{A,q}: = L_{1,A,q}$ for simplicity.

Let us define the sesquilinear form
\begin{align}\label{eq_defn_a}
\mathfrak a(t;u,v) := c(t)\int_{\Omega}\nabla_{A}u(t,x)\cdot\overline{\nabla_Av(t,x)}\,dx -\int_\Omega q(t,x)u(x)\overline{v(x)}\,dx, \quad \hbox{for }u,\,v\in H^1_0(\Omega),
\end{align} 
which satisfies
\[
\mathfrak{a}(t;u,v)=\overline{\mathfrak{a}(t;v,u)}.
\]
For $u\in H^1_0(\Omega)$, triangle inequality yields 
\[
\|\nabla_Au(t,\cdot)\|^2_{L^2(\Omega)} \ge \frac{\|\nabla u \|_{L^2(\Omega)}^2}{2} - \|A\|_{L^\infty(Q)}^2 \|u \|^2_{L^{2}(\Omega)},
\]
and, therefore, for each $t\in (0,T)$, 
we get 
\begin{equation}\label{eq_a_coercive}
\mathfrak a(t;u,u) + C_1\|u \|^2_{L^2(\Omega)} \geq C_2\|u \|^2_{H^1(\Omega)},
\end{equation}
where $C_1 =  c_0\|A\|^2_{L^\infty(Q)} + \|q\|_{L^\infty(Q)} + \frac{c_0}{2}$, and $C_2= \frac{c_0}{2}$.
Let us also define
\begin{equation}\label{eq_a_deri_def}
\begin{aligned}
\mathfrak a'(t;u,v) :=\,& c'(t)\LA \nabla_{A(t,\cdot)} u  ,{\nabla_{A(t,\cdot)}v }\RA_{L^2(\Omega)}
-\LA\p_tq(t,\cdot)u ,{v } \RA_{L^2(\Omega)}
\\
&+ ic(t) \LA  \p_t{A(t,\cdot)}   u , \nabla_{A(t,\cdot)}v \RA_{L^2(\Omega)}  -  ic(t)\LA \p_t{A(t,\cdot)}{\cdot }\nabla_{A(t,\cdot)}u  , v \RA_{L^2(\Omega)},
\end{aligned}  
\end{equation}
which, by H\"older's inequality, satisfies
\begin{align}\label{eq_a_deri_bdd}
|\mathfrak a'(t;u,u)|\lesssim \|u\|_{H^1(\Omega)}^2 \quad \hbox{for }t\in (0,T),\quad u\in H^1_0(\Omega).
\end{align}

In addition, through applying the integration by parts to \eqref{eq_defn_a} and \eqref{eq_a_deri_def}, for $u,\,v\in H^1_0(\Omega)$, we obtain the following identities: 
\begin{equation}\label{eq_a_ibp}
\begin{aligned}
\mathfrak a(t,u,v) &= c(t)\LA \nabla u,\nabla v\RA_{L^2(\Omega)} -2ic(t)\LA A(t,\cdot)\cdot\nabla u,v\RA_{L^2(\Omega)} 
\\
&\quad-ic(t)\LA (\nabla\cdot A(t,\cdot))u,v\RA_{L^2(\Omega)} +c(t)\LA |A(t,\cdot)|^2  u,v\RA_{L^2(\Omega)}-\LA q(t,\cdot) u,v\RA_{L^2(\Omega)},  
\end{aligned}
\end{equation}
and
\begin{equation}\label{eq_a_deri_ibp}
\begin{aligned}
\mathfrak a'(t,u,v) &=c'(t)\LA \nabla  u  ,{\nabla v }\RA_{L^2(\Omega)}-2i  \LA \p_t(cA)(t,\cdot)\cdot\nabla u,v\RA_{L^2(\Omega)} -i \LA (\nabla\cdot \p_t) (cA)(t,\cdot) u,v\RA_{L^2(\Omega)} 
\\
&\quad + c'(t)\LA |A(t,\cdot)|^2 u,v\RA_{L^2(\Omega)} + \LA \LC c \p_t |A(t,\cdot)|^2 - \p_t q(t,\cdot) \RC u,v\RA_{L^2(\Omega)},
\end{aligned}
\end{equation}
which will be used in the proof for unique existence of the solution for the linear MSE right below.

\begin{theorem}\label{thm_linear_wellposedness}
Suppose that $0<c_0<c(t)\in C^1([0,T];\R)$, $A\in C^2(\overline{Q};\R^n)$, and $q\in C^1(\overline{Q};\R)$. We consider the following IBVP for the linear magnetic Schr\"odinger equation 
\begin{align}\label{eq_linear_IBVP}
\left\{\begin{array}{rrll}
(i\p_t + c(t)\Delta_A+ q)v  &=& F  &\hbox{ in } Q,\\
v  &=& 0  &\hbox{ on } \Sigma,\\
v &=& 0  &\hbox{ in } \{t=0\}\times\Omega.
\end{array}  \right.
\end{align}
\begin{itemize}
\item[(a)] Let $F\in H^{1}(0,T;H^{-1}(\Omega))$ be such that 
$F(0,\cdot)= 0$ in $\Omega$. Then the IBVP problem \eqref{eq_linear_IBVP} has a unique solution $v\in L^2(0,T;H^1_0(\Omega))\cap H^{1}(0,T;H^{-1}(\Omega))$ satisfying 
\begin{align}\label{eq_linear_est_weak}
\|v\|_{L^2(0,T;H^1(\Omega))}\lesssim \|F\|_{H^{1}(0,T;H^{-1}(\Omega))}.
\end{align}
\item[(b)] Let $F\in H^{1,0}(Q)$ satisfy $F(0,\cdot)= 0$ in $\Omega$. Then the IBVP problem \eqref{eq_linear_IBVP} has a unique solution $v\in H^{1,2}(Q)$ satisfying 
\begin{align}\label{eq_linear_IBVP_est}
\|v\|_{H^{1,2}(Q)}\lesssim \|F\|_{H^{1,0}(Q)}.
\end{align}
\end{itemize}
\end{theorem}
\begin{proof}[Proof of Theorem \ref{thm_linear_wellposedness} (a)]
The proof is based on a modification of \cite[Theorem 10.1, Chapter 3]{lions_non-homogeneous_1972-1} by applying the Faedo-Galerkin method, see also \cite{KianSoccorsi}. We will first construct approximate solutions to \eqref{eq_linear_IBVP} in finite dimensions and then establish the associated stability estimates, which enable us to prove existence and uniqueness for \eqref{eq_linear_IBVP}.  

\textbf{Step 1: Construction of approximate solutions.}
We start by showing that there exists a unique weak solution to the IBVP \eqref{eq_linear_IBVP} through solving a finite-dimensional approximation. 
In doing so, we choose $\{e_k\}_{k\in\mathbb{N}}$ being a Hilbert basis of $H^1_0(\Omega)$ such that it is an orthonormal basis of $L^2(\Omega)$ and an orthogonal basis of $H^1_0(\Omega)$. For each $m\in \mathbb{N}$, we define the approximation solutions 
\begin{align}\label{eq_thm21_v_m_construction}
v_m(t,x):=\sum^m_{k=1} g_{k,m}(t) e_k(x), \quad(t,x)\in Q.
\end{align}
Here the functions $g_{k,m}(t)$ are selected so that  
\begin{equation}\label{eq_thm21_v_m}
\LA i\p_t v_m(t,\cdot),e_k\RA_{L^2(\Omega)}  -\mathfrak 
a(t;v_m(t,\cdot),e_k) = \LA F(t,\cdot),e_k\RA_{H^{-1}(\Omega),H^1_0(\Omega)}, \quad 
v_{m}(0,\cdot) = 0
\end{equation}
for $t\in(0,T)$ and $k=1,\cdots, m$. 
Here $\LA\cdot,\cdot\RA_{H^{-1}(\Omega),H^{1}_0(\Omega)}$ denotes the duality between $H^{-1}(\Omega)$ and $H^1_0(\Omega)$.
More precisely, $g_{k,m}(t)$ would satisfy
\begin{align}\label{eq_thm21_g_km}
i g_{k,m}'(t) - \sum^{m}_{\ell=1} g_{\ell,m}(t) \mathfrak a(t;e_\ell,e_k) = F_k(t)\quad \hbox{ with }\quad g_{k,m}(0) = \LA v_m(0,\cdot),e_k \RA_{L^2(\Omega)}=0,
\end{align}
where $F_k(t): =\LA F(t,\cdot),e_k\RA_{H^{-1}(\Omega),H^1_0(\Omega)}\in H^{1}(0,T)$. By Carath\'eodory's Theorem, there exists a unique $\{g_{k,m}\}_{k=1}^{m}\in W^{1,\infty}(0,T)$ solving the above system since $F_k\in H^1(0,T)\subset W^{1,1}(0,T)$, see \cite[Theorem A.50]{john_finite_2016}. Therefore, for $F\in W^{1,1}(0,T;H^{-1}(\Omega))$, there exists a unique solution $v_m\in W^{1,\infty}(0,T;H^1_0(\Omega))$ solving \eqref{eq_thm21_v_m}.

\textbf{Step 2: An $L^2$-estimate for $v_m$.} 
We multiply the equation in \eqref{eq_thm21_v_m} by $\overline{g_{k,m}}$. Summing over $k$ from $1$ to $m$ yields that
\[
\LA i\p_t v_m(t,\cdot),v_m(t,\cdot)\RA_{L^2(\Omega)}  -\mathfrak a(t;v_m(t,\cdot),v_m(t,\cdot)) = \LA F(t,\cdot),v_m(t,\cdot)\RA_{H^{-1}(\Omega),H^1_0(\Omega)}, \quad t\in (0,T).
\]
Then we take twice the imaginary parts on both sides and obtain
\begin{equation}\label{eq_thm21_3}
\p_t\|v_m(t,\cdot)\|_{L^2(\Omega)}^2 = 2\Im \LA F(t,\cdot),v_m(t,\cdot) \RA_{H^{-1}(\Omega),H^1_0(\Omega)}.
\end{equation}
Since $v_m(0,\cdot) = 0$ in $\Omega$, integrating the above equation over $(0,t)$ gives
\begin{align}\label{vm L2}
\|v_m(t,\cdot)\|_{L^2(\Omega)}^2  \lesssim \int_0^t\|v_m(s,\cdot)\|^2_{H^1(\Omega)}\rd s +  \int_0^t \|F(s,\cdot)\|^2_{H^{-1}(\Omega)} \rd s.
\end{align}

\textbf{Step 3: An $H^1$-estimate for $v_m$.}
Let us multiply the equation in \eqref{eq_thm21_v_m} by $\overline{g'_{k,m}}$ and sum over $k$ from $1$ to $m$. Now we have
\[
\LA i\p_t v_m(t,\cdot),\p_t v_m(t,\cdot)\RA_{L^2(\Omega)} -\mathfrak a(t;v_m(t,\cdot),\p_tv_m(t,\cdot)) = \LA F(t,\cdot),\p_t v_m(t,\cdot) \RA_{H^{-1}(\Omega),H^1_0(\Omega)}, \quad t\in (0,T),
\]
from which, 
taking twice the real parts on both sides, we get 
\begin{align}\label{ID:linear form a}
-\mathfrak a(t;v_m,\p_tv_m) - \mathfrak a(t;\p_tv_m,v_m) = 2\Re \LA F(t,\cdot),\p_t v_m(t,\cdot) \RA_{H^{-1}(\Omega),H^1_0(\Omega)}.
\end{align}
Combining with \eqref{eq_a_deri_def}, we can write
\begin{equation}\label{eq_thm21_5}
\frac{d}{ds}\mathfrak a(s;v_m(s,\cdot),v_m(s,\cdot)) = \mathfrak a'(s;v_m(s,\cdot),v_m(s,\cdot)) -  2\Re \LA F(s,\cdot),\p_t v_m(s,\cdot) \RA_{H^{-1}(\Omega),H^1_0(\Omega)}.
\end{equation}
Using $v_m(0,\cdot) = 0$ in $\Omega$ and \eqref{eq_a_deri_bdd}, we integrate \eqref{eq_thm21_5} over $(0,t)$ and get
\begin{equation}\label{eq_thm21_6}
\mathfrak a(t;v_m(t,\cdot),v_m(t,\cdot))\lesssim \int_0^t 
\mathfrak \|v_m(s,\cdot)\|^2_{H^1(\Omega)} \rd s +  2 \LV \int_0^t \LA F(s,\cdot),\p_s v_m(s,\cdot) \RA_{H^{-1}(\Omega),H^1_0(\Omega)}\rd s \RV.
\end{equation}
In addition, by applying the integration by parts, we control the second term on the right as follows:
\begin{align*}
&2 \LV \int_0^t \LA F(s,\cdot),\p_s v_m(s,\cdot) \RA_{H^{-1}(\Omega),H^1_0(\Omega)}\rd s \RV
\\
=\,& 2\LV\LA F(t,\cdot),v_m(t,\cdot) \RA_{H^{-1}(\Omega),H^1_0(\Omega)} -\int_0^t \LA \p_sF(s,\cdot),v_m(s,\cdot) \RA_{H^{-1}(\Omega),H^1_0(\Omega)}\rd s \RV
\\
\le\,& \frac{1}{\varepsilon_0}\|F(t,\cdot)\|^2_{H^{-1}(\Omega)} + {\varepsilon_0}\|v_m(t,\cdot)\|^2_{H^1(\Omega)} + \int_0^t\|\p_sF(s,\cdot)\|^2_{H^{-1}(\Omega)}\rd s + \int_0^t \|v_m(s,\cdot)\|^2_{H^1(\Omega)}\rd s ,
\end{align*}
where $\varepsilon_0>0$ is an arbitrary constant. Here we also used the fact that $v_m(0,\cdot)= 0$ in $\Omega$.
Together with \eqref{eq_a_coercive}, \eqref{vm L2},
and \eqref{eq_thm21_6}, the term $\varepsilon_0\|v_m(t,\cdot)\|^2_{H^1(\Omega)}$ will be absorbed by $C_2\|v_m(t,\cdot)\|^2_{H^1(\Omega)}$ in \eqref{eq_a_coercive} for sufficiently small $\varepsilon_0$; moreover, we can then derive
\begin{equation}\label{eq_thm21_7}
\|v_m(t,\cdot)\|^2_{H^1(\Omega)} \lesssim 
\int_0^t\|v_m(s,\cdot)\|^2_{H^1(\Omega)}\rd s  + \|F(t,\cdot)\|^2_{H^{-1}(\Omega)} + \|F\|^2_{H^1(0,T;H^{-1}(\Omega))}. 
\end{equation}
Lastly, integrating \eqref{eq_thm21_7} over $(0,\tau)$ for $\tau\in (0,T)$ yields
\[
\int_0^\tau\|v_m(t,\cdot)\|^2_{H^1(\Omega)}\rd t\lesssim \int_0^\tau\int_0^t\|v_m(s,\cdot)\|^2_{H^1(\Omega)}\rd s\rd t + (1+T)\|F\|^2_{H^1(0,T;H^{-1}(\Omega))},
\]
which implies 
\begin{align}\label{eq_prop_v_m_est}
\|v_m\|_{L^2(0,T;H^1(\Omega))} \lesssim \|F\|_{H^{1}(0,T;H^{-1}(\Omega))}.
\end{align}thanks to Gronwall's lemma.

\textbf{Step 4: Existence and uniqueness.}
Using \eqref{eq_prop_v_m_est}, Banach-Alaoglu Theorem implies there exists a subsequence of $\{v_m\}_{k\in \mathbb{N}}$, denoted again by $\{v_m\}_{k\in \mathbb{N}}$, converging weakly to some $v\in L^{2}(0, T; H^1_0(\Omega))$ in $L^2(0,T;H^1_0(\Omega))$-sense. By sending $m\to \infty$ in \eqref{eq_thm21_v_m}, we get that $v$ satisfies
\[
\left\{\begin{array}{rrll}
(i\p_t  + c(t)\Delta_A + q)v &=& F  &\hbox{ in } Q,\\
v &=& 0  &\hbox{ on } \Sigma,\\
v &=& 0  &\hbox{ on } \{t=0\}\times\Omega ,
\end{array}  \right.
\]
and \eqref{eq_linear_est_weak} holds.
By writing $i\p_t v = F-c(t)\Delta_A v-qv$, we get $\p_t v\in L^{2}(0,T; H^{-1}(\Omega))$, which shows the existence. The proof is complete by noting that uniqueness of the solution follows directly from \eqref{eq_linear_est_weak}.
\end{proof}

\begin{proof}[Proof of Theorem \ref{thm_linear_wellposedness} (b)]
We first follow the construction in the proof of Theorem \ref{thm_linear_wellposedness} (a) by replacing $\LA \cdot,\cdot \RA_{H^{-1}(\Omega),H^1_0(\Omega)}$ by $\LA \cdot,\cdot\RA_{L^2(\Omega)}$. Thus, we obtain $v_m\in W^{1,\infty}(0,T;H^1_0(\Omega))$ of the form \eqref{eq_thm21_v_m_construction} solving 
\begin{align}\label{eq_thm21_b_v_m}
\LA i\p_t v_m(t,\cdot),e_k\RA_{L^2(\Omega)}  -\mathfrak 
a(t;v_m(t,\cdot),e_k) = \LA F(t,\cdot),e_k\RA_{L^2(\Omega)}, \quad 
v_{m}(0,\cdot)= 0
\end{align}
for $t\in(0,T)$ and $k=1,\cdots, m$.
{  
This yields that $w_m := \p_t v_m$ solves the following problem:
\begin{align}\label{eq_thm21_8}
\LA i\p_t w_m(t,\cdot),e_k\RA_{L^2(\Omega)}  -\mathfrak 
a(t;w_m(t,\cdot),e_k) = \mathfrak a'(t,v_m,e_k) + \LA \p_tF(t,\cdot),e_k\RA_{L^2(\Omega)}, \quad 
w_{m}(0,\cdot)= 0.
\end{align}
}

Next, to derive the estimate of $\|v_m\|_{L^2(0,T;H^1(\Omega))}$, we will follow a similar argument as in Step 2 of (a). 
We multiple \eqref{eq_thm21_b_v_m} by $\overline{g_{k,m}}$ and sum it over $k$. Also, integrating over $(0,t)$ leads to
\[
\|v_m(t,\cdot)\|_{L^2(\Omega)}^2  \lesssim \int_0^t\|v_m(s,\cdot)\|^2_{L^2(\Omega)}\rd s +  \int_0^t \|F(s,\cdot)\|^2_{L^2(\Omega)} \rd s.
\]
By Grownwall's lemma, we obtain
\begin{align}\label{eq_thm21_vm_L2}
\|v_m\|_{L^2(Q)} \lesssim \|F\|_{L^{2}(Q)}.
\end{align}
Moreover, as in Step 3 of (a), it is straightforward to derive 
\begin{align}\label{eq_thm21_vm_H1}
\|v_m\|_{L^2(0,T;H^1(\Omega))} \lesssim \|F\|_{H^{1,0}(Q)}.
\end{align}

Moreover, we will bound $\|w_m\|_{L^2(Q)}$ as follows. Using \eqref{eq_a_deri_ibp}, we multiply \eqref{eq_thm21_8} by $\overline{g'_{k,m}}$ and sum over $k$ from $1$ to $m$. This results in
\begin{equation}\label{eq_thm21_9}
\begin{aligned}
&\LA i\p_tw_m(t,\cdot),w_m(t,\cdot) \RA_{L^2(\Omega)} - \mathfrak a (t,w_m(t,\cdot),w_m(t,\cdot)) 
\\
=& \,c'(t)\LA  \nabla v_m(t,\cdot),\nabla w_m(t,\cdot) \RA_{L^2(\Omega)}+ \LA \widetilde{F}_m(t,\cdot),w_m(t,\cdot)\RA_{L^2(\Omega)}, 
\end{aligned}  
\end{equation}
where 
\begin{align*}
\widetilde{F}_m :=&-2i\p_t(cA) \cdot\nabla v_m +\LC -i (\nabla\cdot\p_t)(cA) +c'(t)|A|^2+ 2c(t)\p_tA\cdot A -\p_tq \RC v_m +  \p_tF,
\end{align*}
satisfies
\begin{align}\label{eq_thm21_Fm}
\|\widetilde{F}_m(t,\cdot)\|_{L^2(\Omega)} \lesssim \|v_m(t,\cdot)\|_{H^1(\Omega)}+\|\p_tF(t,\cdot)\|_{L^2(\Omega)}.
\end{align}
Taking twice the imaginary parts and integrating over $(0,t)$, together with $w_m(0,\cdot) = 0$ in $\Omega$, we have from \eqref{eq_thm21_9} that 
\begin{equation}\label{eq_thm21_10}
\|w_m(t,\cdot)\|^2_{L^2(\Omega)}
\lesssim 
\int_{0}^t \LV c'(s)\LA \nabla v_m(s,\cdot),\nabla w_m(s,\cdot)\RA_{L^2(\Omega)}\RV\,ds
+\int_{0}^t\|\widetilde{F}_m(s,\cdot)\|^2_{L^2(\Omega)}
+\|w_m(s,\cdot)\|^2_{L^2(\Omega)}\,ds. 
\end{equation}

It is sufficient to estimate the first term on the right-hand-side of \eqref{eq_thm21_10}. We multiply \eqref{eq_thm21_b_v_m} by $\overline{g'_{k,m}}$ and sum over $k$ from $1$ to $m$, which leads to 
\begin{equation*}
\LA iw_m(t,\cdot),w_m(t,\cdot)\RA_{L^2(\Omega)}  -\mathfrak 
a(t;v_m(t,\cdot),w_m(t,\cdot)) = \LA F(t,\cdot),w_m(t,\cdot)\RA_{L^2(\Omega)}.
\end{equation*}
Substituting the expression \eqref{eq_a_ibp} of $\mathfrak 
a(t;v_m(t,\cdot),w_m(t,\cdot))$ into the above identity yields
\begin{align*}
&\LV c(t)\LA \nabla v_m(t,\cdot),\nabla w_m(t,\cdot) \RA_{L^2(\Omega)} \RV
\\
\lesssim &\LV \LA F(t,\cdot),w_m(t,\cdot)\RA_{L^2(\Omega)}\RV + \| v_m(t,\cdot)\|^2_{H^1(\Omega)} +\|w_m(t,\cdot)\|^2_{L^2(\Omega)},
\end{align*}
by H\"older's inequality.
Since $c(t)>c_0>0$, we then derive
\begin{align*}
\LV\LA \nabla v_m(t,\cdot),\nabla w_m(t,\cdot) \RA_{L^2(\Omega)} \RV 
\lesssim 
\frac{1}{c_0} \LC \| v_m(t,\cdot)\|^2_{H^1(\Omega)} +\|w_m(t,\cdot)\|^2_{L^2(\Omega)} +\|F(t,\cdot)\|^2_{L^2(\Omega)} \RC.
\end{align*}
Together with  \eqref{eq_thm21_Fm} and \eqref{eq_thm21_10}, we finally arrives at
\begin{align}
\|w_m(t,\cdot)\|^2_{L^2(\Omega)} \lesssim \int_0^t \|w_m(s,\cdot)\|^2_{L^2(\Omega)}\,ds + \| v_m \|^2_{L^2(0,T;H^1(\Omega))}+\|F \|^2_{H^{1,0}(Q)}.
\end{align}
Applying Gronwall's lemma and \eqref{eq_thm21_vm_H1}, the above inequality implies 
\begin{align}\label{eq_thm21_13}
\|\p_tv_m\|_{L^2(Q)}=
\|w_m\|_{L^2(Q)}\lesssim
\|w_m\|_{L^\infty(0,T;L^2(\Omega))}\lesssim \|F\|_{H^{1,0}(Q)}.
\end{align}

Lastly, we prove the existence and uniqueness. Argue as in the proof of Theorem~\ref{thm_linear_wellposedness} (a), \eqref{eq_thm21_vm_H1} implies there exists a subsequence of $\{v_m\}_{k\in \mathbb{N}}$, denote again by $\{v_m\}_{k\in \mathbb{N}}$, converging to some $v\in L^2(0,T;H^1_0(\Omega))$ in $L^2(0,T;H^1_0(\Omega))$-sense. By sending $m\to\infty$, we get that $v$ satisfies \eqref{eq_linear_IBVP} and thus $i\p_t v = F-c(t)\Delta_A v-qv \in L^{2}(0,T; H^{-1}(\Omega))$. Also, 
\begin{align}\label{eq_thm21_v_H1}
\|v\|_{L^2(0,T;H^1(\Omega))}\lesssim \|F\|_{H^{1,0}(Q)}.
\end{align}
Applying Banach-Alaoglu Theorem to $w_m$, we get from \eqref{eq_thm21_13} that there is a subsequence of $\{w_m\}_{k\in \mathbb{N}}$ converging to $w\in L^2(Q)$ in $L^2$-sense and 
\begin{align*}
\|w\|_{L^2(Q)}\lesssim \|F\|_{H^{1,0}(Q)}.
\end{align*}
Since $w_m = \p_tv_m$, it leads to $\p_tv = w\in L^2(Q)$ and 
\begin{align}\label{eq_thm21_deri_v_L2}
\|\p_tv\|_{L^2(Q)}\lesssim \|F\|_{H^{1,0}(Q)}.
\end{align}
From \eqref{eq_linear_IBVP}, for each fixed $t\in(0,T)$, $v(t,\cdot)$ satisfies the elliptic equation $c(t)\Delta_A v+q v=F-i\p_t v$ with $v(t,\cdot)=0$ on $\p\Omega$. Since $F-i\p_t v\in L^2(\Omega)$, we have $v(t,\cdot)\in H^2(\Omega)$ by elliptic regularity. Also combining with \eqref{eq_thm21_v_H1} and \eqref{eq_thm21_deri_v_L2}, we derive \eqref{eq_linear_IBVP_est}, which ends the proof. 
\end{proof}

\begin{remark}\label{remark_L2_est}
Note that in \eqref{eq_thm21_vm_L2}, one only needs the $L^2$ boundedness of $F$ to get
the following stability estimate of the solution in the $L^2$ space:
\begin{align}\label{v L2}
\|v\|_{L^2(Q)}\lesssim \|F\|_{L^2(Q)}.
\end{align}
\end{remark}

\begin{remark}\label{remark_wellposed_c}
With additional assumption $F\in L^\infty(0,T;H^{-1}(\Omega))$ in Theorem \ref{thm_linear_wellposedness} (a), we have 
\begin{equation}\label{eq_remark_wellposed_c}
\|v\|_{L^\infty(0,T;H^1(\Omega))}\lesssim 
\|F\|_{L^\infty(0,T;H^{-1}(\Omega))} + \|F\|_{H^1(0,T;H^{-1}(\Omega))},
\end{equation}
which is resulted from applying Gronwall's lemma to \eqref{eq_thm21_7}.
\end{remark}

In the remaining of this section, we will only consider the equation 
$$ 
L_{A,q}u=(i\p_t+\Delta_A+q)u=0
$$ 
with $c(t)\equiv 1$. We will show that the regularity of the well-posed result can be enhanced for this setting.

\begin{prop}\label{prop_linear_est_higher}
Let $\kappa\in \mathbb{N}$.
Suppose that $F\in \mathcal{H}_0^{2\kappa}(Q)$,  
$A \in  C^\infty(\overline{Q};\R^n)$, and $q\in C^\infty(\overline{Q};\R)$ satisfy  
$A=q=0$ on $(0,T)\times\mathcal{O}$.
Then the problem
\begin{align}\label{forward:eq_linear_IBVP}
\left\{\begin{array}{rrll}
(i\p_t+\Delta_A+q)v  &=& F  &\hbox{ in } Q,\\
v  &=& 0  &\hbox{ on } \Sigma,\\
v &=& 0  &\hbox{ on } \{t=0\}\times\Omega,
\end{array}  \right.
\end{align}
has a unique solution $v\in \mathcal{H}_0^{2\kappa}(\Omega)$ satisfying 
\begin{equation}\label{eq_linear_est_higher}
\|v\|_{H^{2\kappa}(Q)} \lesssim \|F\|_{H^{2\kappa}(Q)}.
\end{equation}
\end{prop}

Later we use the notation $\mathcal{S}^{L}: \mathcal{H}^{2\kappa}_0(Q) \to \mathcal{H}^{2\kappa}_0(Q)$ to denote the solution operator of the problem \eqref{forward:eq_linear_IBVP} and write $v = \mathcal{S}^L(F)$.

\begin{proof}
We will use the induction argument to prove \eqref{eq_linear_est_higher}.
Let's start by showing the cases of $\kappa= 1$.

\textbf{Step 1.}
We first note that $\p_t^\ell v= 0$ on $\Sigma$ for $0\leq \ell \le 2\kappa$ due to $v = 0$ on $\Sigma$.
Using the equation $i\p_t v = F - \Delta_Av-qv$ in \eqref{forward:eq_linear_IBVP}, initial data $v(0,\cdot) = 0$ on $\Omega$, and the hypotheses $\p_t^\ell F(0,\cdot)= 0$ on $\Omega$, $0\leq \ell\le 2\kappa-1$, it follows that
\[
\p_t^\ell v(0,\cdot) = 0 \hbox{ on } \Omega, \quad 0\leq \ell \le 2\kappa.
\]

For $\kappa=1$, since $F\in \mathcal{H}_0^{2}(Q)\subset H^{1,0}(Q)$, Theorem~\ref{thm_linear_wellposedness} (b) ensures the unique existence of the solution $v\in H^{1,2}(Q)$ of \eqref{forward:eq_linear_IBVP} and 
\begin{equation}\label{eq_induction2_4}
\|v\|_{H^{1,2}(Q)}\lesssim \|F\|_{H^{1,0}(Q)}.
\end{equation}

To show $\|v\|_{H^{2}(Q)} \lesssim \|F\|_{H^{2}(Q)}$, we apply $\p_t$ to $L_{A,q}v =F$ and derive
\begin{align}\label{eq_induction2_2_new}
L_{A,q}(\p_t v)=i\p_t^{2} v +\Delta_A \p_t v + q\p_t v  = F_1(\p_tF,\p_t A, \p_t\p_x A,v, \p_x v, \p_tq).
\end{align}
We also have $\p_t v = 0$ on $\Sigma$ and $\p_tv(0,\cdot)=0$ on $\Omega$.
Here $F_1$ contains the linear combinations of the terms of the form $\p_tF$, $\p_tA\cdot\nabla v$, $(\p_t\nabla\cdot A)v$, and $(\p_tq)v$. In fact, $F_1\in H^{1}(0,T;H^{-1}(\Omega))$ and $F_1(0,\cdot)=0$ on $\Omega$. Using Theorem~\ref{thm_linear_wellposedness} (a), we derive  $\p_t v\in L^2(0,T;H^1(\Omega))$ is a solution of \eqref{eq_induction2_2_new} and it satisfies
\begin{align}\label{eq_induction2_3}
\|\p_t v\|_{L^2(0,T;H^1(\Omega))}\lesssim \|F_1\|_{H^{1}(0,T;H^{-1}(\Omega))}
\lesssim \|F\|_{H^{2}(Q)},
\end{align}
which implies $\p_t\p_x v\in L^2(Q)$. From the definition of $F_1$, we now get $F_1\in H^{1,0}(Q)$. With this, we are able to raise the regularity of the solution. Applying Theorem~\ref{thm_linear_wellposedness} (b) to \eqref{eq_induction2_2_new}, we deduce $\p_t v\in H^{1,2}(Q)$ so that 
\begin{align}\label{eq_induction2_1}
\|\p_t v\|_{H^{1,2}(Q)}\lesssim  \|F\|_{H^{2}(Q)},
\end{align}
which results in the desired estimate $\|v\|_{H^{2}(Q)} \lesssim \|F\|_{H^{2}(Q)}$.

Now by the induction argument, we assume that  
\begin{equation}\label{induction k-1}
\|v\|_{H^{2(\kappa-1)}(Q)} \lesssim \|F\|_{H^{2(\kappa-1)}(Q)}.
\end{equation}
Showing \eqref{eq_linear_est_higher} is equivalent to justifying
\[
\|v\|_{H^{2\kappa}(0,T;L^2(\Omega))} +\|v\|_{L^2(0,T;H^{2\kappa}(\Omega))} \lesssim \|F\|_{H^{2\kappa}(Q)},
\]
see \cite[Proposition 2.3, Chapter 4]{lions_non-homogeneous_1972} for identifying the Sobolev space $H^{2\kappa}(Q)$ with $H^{2\kappa,2\kappa}(Q)$.
To establish $\|v\|_{H^{2\kappa}(0,T;L^2(\Omega))}\lesssim \|F\|_{H^{2\kappa}(Q)}$, by our assumption \eqref{induction k-1}, it suffices to show
\begin{align}\label{EST:induction goal_1}
\quad \|\p_t^{2\kappa-1} v\|_{L^2(Q)}+\|\p_t^{2\kappa} v\|_{L^2(Q)}\lesssim  \|F\|_{H^{2\kappa}(Q)}.   
\end{align}
Note that, for $1\le\ell\le{\kappa-1}$, $\Delta^{\ell+1}v|_\Sigma=\Delta^{\ell}F|_{\Sigma}\in H^{2(\kappa-\ell)-\frac{1}{2}}(\Sigma)$, due to 
$A=q=0$ on $(0,T)\times\mathcal{O}$ 
and $v|_{\Sigma}=0$. By elliptic regularity
and the assumption \eqref{induction k-1}, it remains to show
\begin{align}\label{EST:induction goal_2}
\|\Delta^\kappa u\|_{L^2(Q)}\lesssim  \|F\|_{H^{2\kappa}(Q)},
\end{align}
which would imply $\|v\|_{L^2(0,T;H^{2\kappa}(\Omega))} \lesssim \|F\|_{H^{2\kappa}(Q)}$.

\textbf{Step 2.}
To this end, we will claim that 
\begin{align}\label{induction increase t}
\|\p_t^\ell v\|_{H^{1,2}(Q)}\lesssim \|F\|_{H^{\ell+1}(Q)}, \quad 0\leq \ell \leq 2\kappa-1,
\end{align}
which yields \eqref{EST:induction goal_1}.
For this, we apply the induction argument again.
The $\ell=0$ case has been settled by \eqref{eq_induction2_4}, and $\ell = 1$ case has been established by \eqref{eq_induction2_1}.
Assuming that \eqref{induction increase t} holds for $0,1,\cdots,\ell -1$, we now prove it holds for $\ell$.
Applying $\p_t^{\ell}$ to $L_{A,q}v =F$, we get 
\begin{align}\label{eq_induction2_2_new2}
(i\p_t^{\ell+1} +\Delta_A \p_t^{\ell} + q\p_t^{\ell})v  = F_{\ell},
\end{align}
where $F_{\ell}$ contains the linear combinations of the terms of the form $\p_t^{\ell} F$, $(\p_t^{\ell-j}A)\cdot(\nabla\p_t^jv)$, $(\p_t^{\ell-j}(\nabla\cdot A))(\p_t^jv)$, $(\p_t^{\ell-j}q)(\p_t^jv)$
for $j = 0,1,\cdots, \ell-1$. From \eqref{induction increase t}, $F_{\ell}\in H^{1}(0,T;H^{-1}(\Omega))$ and, in particular, it satisfies
\begin{equation*}
\begin{aligned}
&\|F_{\ell}\|_{H^{1}(0,T;H^{-1}(\Omega))} 
\\
\lesssim \,
&\|\p_t^{\ell}F\|_{H^1(0,T;H^{-1}(\Omega))} + \sum_{j=0}^{\ell}\|\nabla\p_t^j v\|_{L^2(0,T;H^{-1}(\Omega))}
+ \sum_{j=0}^{\ell}\|\p_t^j v\|_{L^2(0,T;H^{-1}(\Omega))}
\\
\lesssim \,&\|F\|_{H^{\ell+1}(Q)} + \sum_{j=0}^{\ell}\|\p_t^jv\|_{L^2(Q)}
\lesssim \|F\|_{H^{\ell+1}(Q)}.
\end{aligned}
\end{equation*}
Together with $F_{\ell}(0,\cdot)=0$ on $\Omega$, $\p_t^{\ell}v = 0$ on $\Sigma$ and $\p_t^{\ell}v(0,\cdot) = 0$ on $\Omega$, we apply Theorem~\ref{thm_linear_wellposedness} (a) to \eqref{eq_induction2_2_new2} so that the solution $\p_t^{\ell}v\in L^2(0,T;H^1(\Omega))$ satisfies
\begin{align}\label{eq_induction2_3 new}
\|\p_t^{\ell}v\|_{L^2(0,T;H^1(\Omega))}\lesssim \|F_{\ell}\|_{H^1(0,T;H^{-1}(\Omega))}
\lesssim \|F\|_{H^{\ell+1}(Q)},
\end{align}
which implies $F_{\ell}\in H^{1,0}(Q)$. 
We then apply Theorem~\ref{thm_linear_wellposedness} (b) to deduce that $\p_t^{\ell}v\in H^{1,2}(Q)$ and satisfies
\begin{align*}
\|\p_t^{\ell}v\|_{H^{1,2}(Q)}\lesssim \|F_{\ell}\|_{H^{1,0}(Q)}
&\lesssim \|\p_t^{\ell}F\|_{H^{1,0}(Q)} + \sum_{j=0}^{\ell}\|\p_t^j\nabla v\|_{L^2(Q)} + \sum_{j=0}^{\ell}\|\p_t^j v\|_{L^2(Q)}
\\
&
\lesssim \|F\|_{H^{\ell+1}(Q)} + \sum_{j=0}^{\ell}\|\p_t^jv\|_{L^2(0,T;H^1(\Omega))}
\\
&\lesssim \|F\|_{H^{\ell+1}(Q)} + \sum_{j=0}^{\ell-1}\|\p_t^jv\|_{H^{1,2}(Q)} + \|\p_t^{\ell}v\|_{L^2(0,T;H^1(\Omega))} 
\\
&\lesssim \|F\|_{H^{\ell+1}(Q)},
\end{align*}
where we utilized our induction argument and \eqref{eq_induction2_3 new}. This finishes the proof of the claim.

\textbf{Step 3.} 
It remains to show \eqref{EST:induction goal_2}.
The case $\kappa = 1$ is settled by \eqref{eq_linear_IBVP_est}. For $\kappa\ge 2$,
using $L_{A,q}v= F $, we write 
\[
\Delta^{\kappa} v = \Delta^{\kappa-1}(F-i\p_tv - 2iA\cdot\nabla v - i(\nabla \cdot A)v + |A|^2v-qv)
=: \RNum{1}+\RNum{2}+\RNum{3},
\]
where
\begin{align*}
\RNum{1}: &= \Delta^{\kappa-1}\LC F -i(\nabla\cdot A) v+|A|^2v -qv\RC, 
\\
\RNum{2}: &= -\Delta^{\kappa-2}(\Delta i\p_t v)  =  -i\p_t\Delta^{\kappa-2}\LC F - i\p_tv-2iA\cdot\nabla v -i(\nabla\cdot A) v +|A|^2v -qv\RC,
\\
\RNum{3}: &= - \Delta^{\kappa-2}\Delta(2iA\cdot\nabla v) \\
& = - 2i \Delta^{\kappa-2}\Biggl( \Delta A\cdot \nabla v+ 2\sum_{j,k = 1}^n(\p_{x_j}A_k)(\p_{x_j}\p_{x_k}v)\\
&\quad + A\cdot\nabla\LC F- i\p_tv-2iA\cdot\nabla v -i(\nabla\cdot A) v +|A|^2v -qv\RC \Biggl).
\end{align*}
Noting that the highest derivative term in $\RNum{1}$, $\RNum{2}$, and $\RNum{3}$ is of order $2\kappa-2$, by our induction assumption \eqref{induction k-1}, we have that 
\[
\|\Delta^\kappa v\|_{L^2(Q)}\lesssim \|\RNum{1}\|_{L^2(Q)}+\|\RNum{2}\|_{L^2(Q)}+\|\RNum{3}\|_{L^2(Q)}
\lesssim \|F\|_{H^{2\kappa-2}(Q)}+ \|v\|_{H^{2\kappa-2}(Q)}\lesssim \|F\|_{H^{2\kappa}(Q)},
\]
which completes the proof.
\end{proof}
\begin{remark}\label{remark_adjoint_source}
Suppose that the assumptions of Proposition~\ref{prop_linear_est_higher} hold by considering $F\in \mathcal{H}_T^{2\kappa}(Q)$ instead.
By replacing $t$ by $T-t$, we see that the adjoint problem 
\begin{align}\label{eq_adjoint_IBVP_source}
\left\{\begin{array}{rrll}
(i\p_t+\Delta_A+q) v  &=& F  &\hbox{ in } Q,\\
v &=& 0  &\hbox{ on } \Sigma,\\
v &=& 0  &\hbox{ on } \{t=T\}\times\Omega ,
\end{array}  \right.   
\end{align}
has a unique solution $v\in \mathcal{H}_T^{2\kappa}(Q)$ satisfying \eqref{eq_linear_est_higher}.
\end{remark}

Now we give the existence result for the linear MSE with nontrivial boundary value. 
\begin{prop}[Well-posedness for the linear MSE]\label{prop_linear_boundary}
Let $\kappa\in\mathbb{N}$ 
and 
$f\in \widetilde {\mathcal{H}}_0^{2\kappa+\frac{3}{2}}(\Sigma)$.
Suppose that 
$A \in  C^\infty(\overline{Q};\R^n)$, and $q\in C^\infty(\overline{Q};\R)$ satisfy 
$A=q=0$ on $(0,T)\times\mathcal{O}$.
Then the following IBVP 
\begin{align}\label{eq_linear_IBVP_boundary}
\left\{\begin{array}{rrll}
(i\p_t+\Delta_A+q)v  &=& 0  &\hbox{ in } Q,\\
v &=& f  &\hbox{ on } \Sigma,\\
v &=& 0  &\hbox{ on } \{t=0\}\times\Omega ,
\end{array}  \right.   
\end{align}
has a unique solution $v\in \mathcal{H}_0^{2\kappa}(Q)$ satisfying 
\begin{align}\label{eq_linear_IBVP_bounary_est}
\|v\|_{H^{2\kappa}(Q)}\lesssim \|f\|_{ H^{2\kappa+\frac{3}{2}}(\Sigma)}.
\end{align}
\end{prop} 
\begin{proof}
In light of \cite[Theorem 2.3, Chapter 4]{lions_non-homogeneous_1972}, there exists $U\in \mathcal{H}_0^{2\kappa+ 2}(Q)$  and $U|_{\Sigma} = f$ with 
\begin{equation}\label{eq_proof_linear_well_1}
\|U\|_{H^{2\kappa+ 2}(Q)}\lesssim \|f\|_{H^{2\kappa+\frac{3}{2}}(\Sigma)}.
\end{equation}

To show the existence of $v$ to \eqref{eq_linear_IBVP_boundary}, it suffices to find $\widetilde v:= v-U\in \mathcal{H}_0^{2\kappa}(Q)$ such that 
\[
L_{A,q} \widetilde v  =\widetilde F\]
where 
\[
\widetilde F: = 
-L_{A,q} U  \in \mathcal{H}_0^{2\kappa}(Q).
\]
By Proposition \ref{prop_linear_est_higher},
there exists $\widetilde v\in \mathcal{H}^{2\kappa}_0(Q)$ and the following estimate holds:
\begin{align}\label{eq_proof_linear_well_2}
\|\widetilde v\|_{H^{2\kappa}(Q)}\le \|\widetilde F\|_{H^{2\kappa}(Q)}.
\end{align}
Inequality \eqref{eq_linear_IBVP_bounary_est} is implied by \eqref{eq_proof_linear_well_1} and \eqref{eq_proof_linear_well_2}.
\end{proof}

\begin{remark}\label{remark_adjoint_boundary} We also have the corresponding result for the adjoint problem.
Let $\kappa\in \mathbb{N}$ and $f\in \widetilde {\mathcal{H}}_T^{2\kappa+\frac{3}{2}}(\Sigma)$.
Suppose that 
$A \in  C^\infty(\overline{Q};\R^n)$, and $q\in C^\infty(\overline{Q};\R)$ satisfy 
$A=q=0$ on $(0,T)\times\mathcal{O}$.
Then the adjoint problem 
\begin{align}\label{eq_adjoint_IBVP_boundary}
\left\{\begin{array}{rrll}
(i\p_t+\Delta_A+q)v  &=& 0  &\hbox{ in } Q,\\
v &=& f  &\hbox{ on } \Sigma,\\
v &=& 0  &\hbox{ on } \{t=T\}\times\Omega ,
\end{array}  \right.   
\end{align}
has a unique solution $v\in \mathcal{H}_T^{2\kappa}(Q)$ satisfying \eqref{eq_linear_IBVP_bounary_est}.
\end{remark}

\subsection{Well-posedness for the nonlinear MSE}

The goal of this subsection is to show the unique existence of solutions for the initial boundary value problem for the nonlinear MSE.

\begin{prop}[Well-posedness for the nonlinear MSE]\label{prop_nonlinear_wellposed}
Let $\kappa > {n+1\over 2}$ be an integer. 
Suppose that $B_{\sigma\beta}\in C^\infty(\overline{Q};\R)$,
$A \in  C^\infty(\overline{Q};\R^n)$, $q\in C^\infty(\overline{Q};\R)$ such that 
$A=q=0$ on $(0,T)\times\mathcal{O}$.
There exists $\delta>0$ small enough such that for
$f\in \mathcal{D}_\delta(\Sigma)$, the problem \eqref{eq_IBVP} has a unique solution $u\in {\mathcal{H}}^{2\kappa}_0(Q)$ satisfying 
\begin{align}\label{eq_nonlinear_est}
\|u\|_{H^{2\kappa}(Q)}\lesssim \|f\|_{H^{2\kappa+\frac{3}{2}}(\Sigma)}.
\end{align}
\end{prop}

\begin{proof} 
For $f\in \widetilde {\mathcal{H}}_0^{2\kappa+\frac{3}{2}}(\Sigma)$, by Proposition~\ref{prop_linear_boundary}, there exists a unique solution $v_f\in {\mathcal{H}}_0^{2\kappa}(Q)$ of \eqref{eq_linear_IBVP_boundary} satisfying 
\begin{align}\label{eq_proof_v_f}
\|v_f\|_{H^{2\kappa}(Q)}\le C_1 \|f\|_{H^{2\kappa+\frac{3}{2}}(\Sigma)},
\end{align}
for some constant $C_1$. It is sufficient to show the existence of the solution {$w\in {\mathcal{H}}_0^{2\kappa}(Q)$} to the following problem:	
\begin{align}\label{eq_adjoint_IBVP_boundary w}
\left\{\begin{array}{rrll}
(i\p_t+\Delta_A+q)w  &=& N(t,x,v_f+w,\overline{v_f+w})  &\hbox{ in } Q,\\
w &=& 0 &\hbox{ on } \Sigma,\\
w &=& 0  &\hbox{ on } \{t=0\}\times \Omega.
\end{array}  \right.   
\end{align}

To this end, we will apply the implicit function theorem. Recall that $\mathcal{S}^L:{\mathcal{H}}_0^{2\kappa}(Q)\rightarrow {\mathcal{H}}_0^{2\kappa}(Q)$ is the solution operator of \eqref{forward:eq_linear_IBVP}. 
We define two maps $F,\, \mathcal{G}: {\mathcal{H}}_0^{2\kappa}(Q)\times \widetilde {\mathcal{H}}_0^{2\kappa+\frac{3}{2}}(\Sigma) \rightarrow {\mathcal{H}}_0^{2\kappa}(Q)$ by
$$F(w ,f) := N(t,x,v_f+w,\overline{v_f+w}) \quad\hbox{ and }\quad \mathcal{G} (w, f) := w - \mathcal{S}^LF(w, f).$$
Here we used the fact that for $\kappa>{n+1\over 2}$, the space ${\mathcal{H}}_0^{2\kappa}(Q)$ is a Banach algebra so that $F(w, f)\in {\mathcal{H}}_0^{2\kappa}(Q)$ for each $(w, f)\in {\mathcal{H}}_0^{2\kappa}(Q)\times \widetilde {\mathcal{H}}_0^{2\kappa+\frac{3}{2}}(\Sigma)$.
Applying Proposition~\ref{prop_linear_est_higher} then yields that  there exists a unique solution $\mathcal{S}^L F(w,f) \in {\mathcal{H}}_0^{2\kappa}(Q)$ to the problem \eqref{forward:eq_linear_IBVP}.
Also, due to the well-posedness theorem for the linear MSE, we have $v_f\equiv 0$ if $f=0$ and thus 
$$
\mathcal{G}(0,0)= 0 - \mathcal{S}^LF(0,0)=0.
$$

Note that the map $\mathcal{S}^L$ is linear and $F(w,f)$ is a polynomial with respect to $w$, by using the chain of rule, it follows that $\mathcal{G}$ is smooth in $w$. 
Analogously, since the solution map, which maps from $f$ to $v_f$ is also linear, $\mathcal{G}$ is also smooth in $f$. In particular, 
we derive that 
$\p_w \mathcal{G}(0,0): {\mathcal{H}}_0^{2\kappa}(Q)\rightarrow {\mathcal{H}}_0^{2\kappa}(Q)$ and $\p_w \mathcal{G}(0,0) = Id$, the identity map.
Therefore, by the implicit function theorem
{\cite[Theorem 10.6]{renardy_introduction_2004}} (\cite[Appendix B]{poschel_inverse_1987}), there exists   
$\delta>0$ and a smooth map $\psi$ from $\mathcal{D}_\delta(\Sigma)\rightarrow {\mathcal{H}}_0^{2\kappa}(Q)$ such that for all $f \in \mathcal{D}_\delta(\Sigma)$, we have 
$$
\mathcal{G}(\psi(f),f) = 0,
$$ 
which implies $w:=\psi(f)$ is the solution of \eqref{eq_adjoint_IBVP_boundary w}.
Since $\mathcal{\psi}(0)=0$ and $\mathcal{\psi}$ is Lipschitz continuous, the solution $w=\psi(f)$ satisfies
\begin{align}\label{eq_proof_w_f}
\|w\|_{H^{2\kappa}(Q)}\le C_1 \|f\|_{H^{2\kappa+\frac{3}{2}}(\Sigma)}.
\end{align}
With this, we can deduce that $u:=v_f+w$ is the solution of \eqref{eq_IBVP} and satisfies
\begin{align}\label{eq_proof_u}
\|u\|_{H^{2\kappa}(Q)}\leq \|v_f\|_{H^{2\kappa}(Q)}+\|w\|_{H^{2\kappa}(Q)}\le C_1 \|f\|_{H^{2\kappa+\frac{3}{2}}(\Sigma)}\le C_1\delta,
\end{align}
for some constant $C_1$. 
\end{proof}	

From Proposition~\ref{prop_nonlinear_wellposed}, we have showed the unique existence of solution for \eqref{eq_IBVP} and derived that the solution map $\psi: \mathcal{D}_\delta(\Sigma)\rightarrow \mathcal{H}_0^{2\kappa}(Q)$ is smooth. It follows that $\Lambda^\sharp_{A,q,N}:\mathcal{D}_\delta(\Sigma)\rightarrow \mathcal{H}^{2\kappa-{3\over 2}}(\Sigma^\sharp)$ 
is well-defined and smooth.

\section{Geometric optics solutions}\label{section_geo_sol}
In this section, we will construct two types of geometric optics (GO) solutions to the {linear magnetic Schr\"odinger equation $L_{c,A,q}v=0$
} with \textit{non-localized amplitudes} and \textit{localized amplitudes}. In particular, the former GO solutions will be utilized to recover the linear coefficients in Section~\ref{sec:linear}, while the later one with amplitudes localized near a straight line will be used to recover the nonlinear term in Section~\ref{sec:nonlinear}.

We make the ansatz that the GO solutions to the equation $L_{c,A,q}v=0$ are of the form 
$$
v(t,x) = e^{i\Phi(t,x)} \underbrace{\LC\sum_{k=0}^N \rho^{-k} V_{k}(t,x) \RC }_{V(t,x;\rho)} + R_\rho(t,x),
$$
where we take the linear phase
$$
\Phi(t,x): = \rho \LC \frac{x\cdot\omega}{\sqrt{c(t)}}{-}\rho|\omega|^2 t\RC
$$
with parameter $\rho>0$ and vector $0\neq  \omega\in \R^n$. Let the operator $L_{c,A,q}$ act on $v$ and then we get
$$
0= L_{c,A,q} v = e^{i\Phi(t,x)}\LC L_{c,A,q} V+ {2i\sqrt{c(t)}}\rho \omega\cdot \nabla_A V  {+ \rho {c'(t)\over 2c(t)^{3/2}} (x\cdot\omega)}V\RC +L_{c,A,q}R_{\rho}.
$$
By comparing the order of $\rho$ on both sides, this implies that we are looking for the terms $V_k\equiv V_k(t,x)$ satisfying 
\begin{equation}\label{CON:a_k}\begin{split}
&{2i\sqrt{c(t)}}\omega\cdot\nabla_A V_{0}{+ {c'(t)(x\cdot\omega)\over 2c(t)^{3/2}} V_0} = 0,\\
&2i\sqrt{c(t)}\,\omega\cdot\nabla_A V_{1}{+ {c'(t)(x\cdot\omega)\over 2c(t)^{3/2}} V_1} =- L_{c,A,q} V_{0},\\
&\qquad\qquad\vdots\\
&2i\sqrt{c(t)}\,\omega\cdot\nabla_A V_{N}{+ {c'(t)(x\cdot\omega)\over 2c(t)^{3/2}} V_N} =- L_{c,A,q} V_{N-1},
\end{split}
\end{equation}
and the remainder term $R_\rho\equiv R_\rho(t,x)$ satisfying
\begin{align}\label{CON:r}
\left\{\begin{array}{rcll}
L_{c,A,q}R_{\rho} &=&- \rho^{-N}e^{i\Phi(t,x)}L_{c,A,q}V_{N}&\quad \hbox{in }Q,\\
R_{\rho}&=&0 &\quad \hbox{on }\Sigma,\\
R_{\rho}&=&0 &\quad \hbox{on }\{t=0\}\times \Omega.\\
\end{array} \right. 
\end{align}

\subsection{GO solutions with non-concentrated amplitudes}\label{sec:go nonconcentrated}
Here we take $\omega\in \mathbb{S}^{n-1}$, similar to the GO solutions constructed in \cite{KianSoccorsi} and \cite{Bella-Fraj2020}. Given {$c(t)\in C^\infty([0,T];\R)$,} $A \in  C^\infty(\overline{Q};\R^n)$ and $q\in C^\infty(\overline{Q};\R)$, 
we extend them to $(0,T)\times\R^n$ smoothly, still denoted by $A$ and $q$, respectively, such that $A(t,\cdot)$ and $q(t,\cdot)$ {vanish outside an open neighborhood of $\overline{\Omega}$. }

Let $h>0$ be a fixed small constant and $\zeta\in C^\infty_0(\R)$ be a smooth function supported in $(h, T-h)$ such that $0\le\zeta\le 1$ and $\zeta =1$ in $[2h,T-2h]$. Also, for each $k\in\mathbb{N}$, there exists a constant $C>0$ such that 
$$
\|\zeta\|_{C^k(\R)}\leq C h^{-k}.
$$

We first construct the leading term $V_0$ for the linear MSE with smooth coefficients $A$ and $q$. Let $\tau\in \R$, $\xi\in \omega^{\perp}: = \{x\in \R^n: x\cdot\omega = 0\}$, and $\eta\in \mathbb{S}^{n-1}$. For $\theta\in C^\infty([0,T]\times \R^n;\R^n)$ with $\supp (\theta(t,\cdot))\subset \Omega$, we define the function $\Theta$ as follows: 
\begin{align}\label{eq_amplitude_eta}
\Theta(t,x) = \eta\cdot \nabla(e^{-i(t\tau+x\cdot\xi)}e^{-i\int_\R\omega\cdot \theta(t,x+s\omega)\rd s}),\quad (t,x)\in [0,T]\times \R^n,
\end{align}
which satisfies $\omega\cdot\nabla\Theta = 0$ in $(0,T)\times \R^n$ by a direct computation. Note that one can also take $\Theta$ as a constant since it also satisfies the transport equation $\omega\cdot\nabla\Theta=0$. Both cases will be utilized in the subsequent section.

Next, for $\omega\in \mathbb{S}^{n-1}$, we define a function $W$ by
\begin{align}\label{eq_amplitude_e_non}
W(t,x) : = {e^{i {c'(t)(x\cdot \omega)^2 \over 8c(t)^2}}} e^{i\int_0^\infty A(t,x+s\omega)\cdot\omega\rd s},\quad (t,x)\in (0,T)\times \R^n,
\end{align}
which is a solution to $\omega\cdot \nabla_{A} W { - {ic'(t)(x\cdot\omega)\over 4c(t)^2} W}= 0$.
We then define the leading amplitude:
\begin{align}\label{eq_amplitude_v_0_non}
V_0 (t,x) := \zeta(t)\Theta (t,x) W(t,x).
\end{align}
It is clear that $V_0\in C^\infty([0,T]\times\R^n)$ is a {smooth} function satisfying the first equation in \eqref{CON:a_k} and, based on the definition of $\zeta$, we have 
$$
V_0(t,x)=0\quad \hbox{ for  $(t,x)\in ((0,h)\cup(T-h,T))\times \Omega$}.
$$

{Now we construct the terms $V_k$, $k= 1,2,\cdots, N$. Denote $H(t,x):=A(t,x) -\frac{c'(t)}{4c(t)^2}x$. 
From \eqref{CON:a_k}, we know that the function $V_k$ is the solution to the transport equation: 
 \[
 \omega\cdot\nabla_H V_{k} = \frac{iL_{c,A,q} V_{k-1}}{2\sqrt{c(t)}}.
 \]
For a fixed $x\in \Omega$, by imposing vanishing initial conditions $V_k(y)=0$ for $y\in \{y\in \R^n:\, (y-x)\cdot \omega =0\}$, we have \begin{align}\label{eq_amplitude_v_l_non}
V_k (t,y+ s\omega) = \frac{i }{2\sqrt{c(t)} } \int_0^s e^{-i\int_{s_1}^s H (t,y+ s_2\omega)\cdot\omega\rd s_2}
L_{c, A , q }V_{k-1} (t,y+s_1\omega)\rd s_1.
\end{align}}
Since the regularity and the support of $V_k$ inherit from $V_0$, it follows that $V_k\in C^\infty([0,T]\times\R^n)$ and 
$$
V_k(t,x)=0\quad \hbox{ for  $(t,x)\in ((0,h)\cup(T-h,T))\times \Omega$}.
$$

When $c(t)\equiv 1$, we construct geometric optics solutions with higher order regularity in the following two propositions. For non-constant $c(t)$, the GO solutions are constructed in the Appendix, in which we concern the inverse problem for linear MSE only.

\begin{prop}\label{prop:GO_Nonconcentrated}
Suppose that $A \in  C^\infty(\overline{Q};\R^n)$,  $q\in C^\infty(\overline{Q};\R)$ satisfy 
$A=q=0$ on $(0,T)\times \mathcal{O}$.
Let $\omega\in \mathbb{S}^{n-1}$ and $m\in\mathbb{N}$. Then there exist GO solutions to the linear magnetic Schr\"odinger equation  $L_{A,q}v=0$ in $Q$ of the form
$$
v(t,x) = e^{i\Phi(t,x)}  \LC\sum^{N}_{k=0} \rho^{-k}V_k(t,x) \RC  + R_\rho(t,x), 
$$
satisfying the initial condition 
$v(0,\cdot)=0$ in $\Omega$. Here $V_k\in H^{m}(Q)$ are given by \eqref{eq_amplitude_v_0_non} - \eqref{eq_amplitude_v_l_non} and satisfy
\begin{align}\label{EST:ak}
\|V_k\|_{H^{m}(Q)}\lesssim \langle \tau,\xi\rangle^{2k+m+1} h^{-k-m},~~0\leq k\leq N 
\end{align}
for any $\tau\in\R$, $h >0$ small enough and $\xi \in \omega^\perp$.
The remainder term $R_\rho$ satisfies
\begin{align}\label{EST:r}
\|R_\rho\|_{H^m(Q)}\lesssim\rho^{-N + 2m}\langle \tau,\xi\rangle^{2N+m+3}h^{-(N+m+1)}. 
\end{align}

In particular, as $\rho\rightarrow\infty$, we have
\begin{align}\label{EST:ak rho}
\|V_k\|_{H^{m}(Q)}= O(1),~~0\leq k\leq N 
\end{align}
and 
\begin{align}\label{EST:r rho}
\|R_\rho\|_{H^m(Q)}= O(\rho^{-N + 2m}).
\end{align}
\end{prop}
Here we denote $\langle \tau,\xi\rangle = (1+\tau^2+|\xi|^2)^{1/2}$ for $\tau\in\R$ and $\xi\in\R^n$.
\begin{proof}
We first note that, from the definition of $\zeta$ and $\Theta$, the function $V_k$ satisfy the estimate 
$$
\p^m_t V_0 (t,x) \lesssim |\tau|^{m}h^{-m},\quad \p^m_x V_0 (t,x) \lesssim |\xi|^{m+1},
$$
and thus
\begin{align}\label{eq_amplitude_est}
\|V_0(t,\cdot)\|_{H^m(Q)} \lesssim \langle \tau,\xi\rangle^{m+1}h^{-m}.
\end{align}
From \eqref{eq_amplitude_v_l_non}, $V_k$ gains $k$ derivatives on $t$ and $2k$ extra derivatives on $x$:
\begin{align}\label{eq_ammlitude_est_Vk}
\p^m_t V_k (t,x) \lesssim |\tau|^{m+k}h^{-m-k},\quad \p^m_x V_k (t,x) \lesssim |\xi|^{m+2k+1},
\end{align}
which immediately give \eqref{EST:ak}. 

Now for the existence of the remainder term $R_\rho$, we apply Proposition~\ref{prop_linear_est_higher} to deduce that there exists a unique solution of $R_\rho$ to \eqref{CON:r} so that
$$
\|R_\rho\|_{H^m(Q)}\leq C\rho^{-N} \|e^{i\Phi}( L_{A,q}V_N)\|_{H^m(Q)} .
$$
With this, from \eqref{EST:ak} and the definition of $\Phi$, \eqref{EST:r} holds.
\end{proof}

\begin{remark}\label{remark adjoint}
A similar argument can be used to construct the GO solutions for the backward equation $L_{A,q}v=0$ with $v(T,\cdot)=0$ in $\Omega$. In particular, the remainder term $R_\rho$ satisfies \eqref{EST:r rho} and $R_\rho|_\Sigma=R_\rho(T,\cdot)=0$.
\end{remark}

\subsection{GO solutions with concentrated amplitudes}\label{sec:go concentrated}
Different from the construction of the non-concentrated GO solutions above, the leading amplitude here will be concentrated along a given line. 

Fixed $x_0\in \Omega$ and a nonzero vector $\omega\in \R^n\setminus\{0\}$. We denote by $\gamma_{x_0,\omega}$ the straight line through $x_0$ in direction $\omega$ with the parametrization $\gamma_{x_0,\omega}(s)=x_0+s\omega$. Choosing $\alpha_1,\cdots,\alpha_{n-1}$ in $\R^n$ so that 
$
\{ {\omega\over |\omega|}, \alpha_1,\cdots, \alpha_{n-1}\}
$
is an orthonormal basis of $\R^n$.
For a small constant $\delta>0$, the leading amplitude is defined by 
\begin{align}\label{eq_amplitude_v_0_con}
V_0 (t,x) :=\iota (t) \prod_{a=1}^{n-1}\chi\LC \frac{(x-x_0)\cdot\alpha_a}{\delta} \RC W(t,x),
\end{align}
where $\iota \in C^\infty_0(0,T)$ is a smooth cutoff function of the time variable $t$, and $\chi\in C^\infty_0(\R)$ is a smooth function with $\chi(z)=1$ for $|z|\leq 1/2$ and $\chi(z)=0$ for $|z|\geq 1$. It is clear that $\omega\cdot \nabla \prod_{a=1}^{n-1}\chi\LC \frac{(x-x_0)\cdot\alpha_a}{\delta} \RC =0$. Together with $\omega\cdot\nabla_A W=0$, it follows that $V_0$ satisfies the first equation in \eqref{CON:a_k}. Notice that $V_0$
is supported in a $\delta$-neighborhood of the line $\gamma_{x_0,\omega}$. For $k= 1,\cdots, N$, we also take $V_k$ as in \eqref{eq_amplitude_v_l_non}. Following a similar argument as in Proposition~\ref{prop:GO_Nonconcentrated}, we have the following results.

For the following proposition, we again consider the case $c(t)\equiv 1$.
\begin{prop}\label{prop:GO_Concentrated}
Suppose that $A \in  C^\infty(\overline{Q};\R^n)$,  $q\in C^\infty(\overline{Q};\R)$ satisfy 
$A=q=0$ on $(0,T)\times \mathcal{O}$.
Let $\omega\in \R^n\setminus\{0\}$ and $m\in\mathbb{N}$. Then there exist GO solutions to the linear magnetic Schr\"odinger equation  $L_{A,q}v=0$ in $Q$ of the form
$$
v(t,x) = e^{i\Phi(t,x)}  \LC\sum^{N}_{k=0}\rho^{-k}V_k(t,x) \RC + R_\rho(t,x), 
$$
satisfying the initial condition $v(0,\cdot)=0$ in $\Omega$. Here $V_k\in H^{m}(Q)$ are given by \eqref{eq_amplitude_v_0_con} and \eqref{eq_amplitude_v_l_non} and satisfy
\begin{align}\label{EST:ak rho con} 
\|V_k\|_{H^{m}(Q)}= O(1),~~0\leq k\leq N 
\end{align}
and 
\begin{align}\label{EST:r rho con}
\|R_\rho\|_{H^m(Q)}= O(\rho^{-N + 2m}),\quad \hbox{as $\rho\rightarrow\infty$}.
\end{align}
\end{prop}

\begin{remark}
A similar argument can be used to construct the GO solutions for the backward equation $L_{A,q}v=0$ with $v(T,\cdot)=0$ in $\Omega$.
\end{remark}

\section{Recovery of the magnetic and electric potentials}\label{sec:linear}
This section is devoted to studying the inverse problem for the linearized MSE. After employing a first-order linearization and establishing an integral identity connecting the unknown linear coefficients with the DN map, the GO solutions constructed in Section~\ref{sec:go nonconcentrated} will be utilized to uniquely recover both $A$ and $q$.

Throughout this section, 
let $\kappa>\frac{n+1}{2}$ be an integer, and {$B_{j,\sigma\beta}\in C^\infty(\overline{Q};\R)$ satisfy $B_{1,\sigma\beta}-B_{2,\sigma\beta}=0$ on $(0,T)\times\mathcal{O}$. Moreover, suppose that $A_j \in  C^\infty(\overline{Q};\R^n)$ and $q_j\in C^\infty(\overline{Q};\R)$ satisfying $A_j=q_j=0$ on $(0,T)\times\mathcal{O}$ for $j=1,\,2$.}
We denote 
\begin{align}\label{DEF:tilde coefficient}
\widetilde B_{\sigma\beta} = B_{1,\sigma\beta}-B_{2,\sigma\beta},\quad \widetilde A = A_1-A_2,\quad \widetilde q = q_1-q_2
\end{align}
and extend them by zero outside $\Omega$ so that they are defined in $[0,T]\times \R^n$.

For a fixed $f\in \widetilde{\mathcal{H}}^{2\kappa + \frac{3}{2}}_0(\Sigma)$, Proposition \ref{prop_nonlinear_wellposed} yields that for sufficiently small $|\varepsilon|<1$, there exists a unique solution $u_{j,\varepsilon} \in \mathcal{H}_0^{2\kappa}(Q)$, $j=1,\,2$, to the problem
\begin{align}\label{eq_5_IBVP}
\left\{\begin{array}{rrll}
(i\p_t + \Delta_{A_j}  + q_j) u_{j,\varepsilon}&=& N_j(t,x, u_{j,\varepsilon},\overline{u_{j,\varepsilon}})  &\hbox{ in }Q,\\
u_{j,\varepsilon}&=& \varepsilon f &\hbox{ on } \Sigma,\\
u_{j,\varepsilon}&=& 0 &\hbox{ on } \{t=0\}\times \Omega,
\end{array}  \right.
\end{align}
satisfying $\|u_{j,\varepsilon}\|_{H^{2\kappa}(Q)}\lesssim \|\varepsilon f\|_{ H^{2\kappa+\frac{3}{2}}(\Sigma)}$.

Let us consider the solution $v_j$ to the following problem for the linear MSE:
\begin{align}\label{eq_first_linearization_v}
\left\{\begin{array}{rrll}
(i\p_t+\Delta_{A_j}+q_j)v_j&=& 0 &\hbox{ in }Q,\\
v_j&=&  f &\hbox{ on } \Sigma,\\
v_j&=& 0 &\hbox{ on } \{t=0\}\times \Omega,
\end{array}  \right.
\end{align}
which satisfies $\|v_j \|_{H^{2\kappa}(Q)}\lesssim \|f\|_{ H^{2\kappa+\frac{3}{2}}(\Sigma)}$ based on Proposition \ref{prop_linear_boundary}.

From Proposition~\ref{prop_nonlinear_wellposed}, the map $\varepsilon \rightarrow u_{j,\varepsilon}$ is smooth, which implies the differentiability of $u_{j,\varepsilon}$ with respect to $\varepsilon$. In particular, the first derivative of the solution $u_{j,\varepsilon}$ at $\varepsilon=0$, i.e., $\p_\varepsilon u_{j,\varepsilon} |_{\varepsilon = 0}$, satisfies the linear magnetic Schr\"odinger equation \eqref{eq_first_linearization_v}.
\begin{lemma}\label{lemma_first_linerization}
Let $u_{j,\varepsilon}$ and $v_j$ be defined as above. Then
\begin{align}\label{eq_first_linearization}
\p_\varepsilon u_{j,\varepsilon} |_{\varepsilon = 0} = v_j,\quad \hbox{in} \quad H^{2\kappa}(Q). 
\end{align}
\end{lemma} 
\begin{proof}
Denote $U_j  := \frac{u_{j,\varepsilon}-u_{j,0}}{\varepsilon} =\frac{u_{j,\varepsilon}-0}{\varepsilon} $. 
Here we used the fact that with trivial data, the well-posedness theorem for the nonlinear MSE implies $u_{j,0}=0$.
Then $U_j\in H^{2\kappa}(Q)$ is the solution to 
\[
\left\{\begin{array}{rrll}
(i\p_t+\Delta_{A_j}+q_j) U_j&=& \frac{1}{\varepsilon}N_j(t,x,u_{j,\varepsilon},\overline{u_{j,\varepsilon}})  &\hbox{ in }Q,\\
U_j&=&  f &\hbox{ on } \Sigma,\\
U_j&=& 0 &\hbox{ on } \{t=0\}\times \Omega.
\end{array}  \right.
\]
Since $v_j$ is the solution to \eqref{eq_first_linearization_v}, it follows that $U_j-v_j\in H^{2\kappa}(Q)$ solves the following IBVP:
\[
\left\{\begin{array}{rrll}
(i\p_t+\Delta_{A_j}+q_j)(U_j -v_j)&=& \frac{1}{\varepsilon}N_j(t,x,u_{j,\varepsilon},\overline{u_{j,\varepsilon}}) &\hbox{ in }Q,\\
U_j -v_j &=& 0 &\hbox{ on } \Sigma,\\
U_j -v_j &=& 0 &\hbox{ on } \{t=0\}\times \Omega.
\end{array}  \right.
\]
Since $H^{2\kappa}(Q), \kappa>{n+1\over 2}$ is a Banach algebra, we apply Proposition~\ref{prop_linear_est_higher} to the above equation and use the estimate $\|u_{j,\varepsilon} \|_{H^{2\kappa}(Q)}\lesssim \|\varepsilon f\|_{ H^{2\kappa+\frac{3}{2}}(\Sigma)}$ to derive
\begin{align*}
\|U_j-v_j\|_{H^{2\kappa}(Q)}
&\lesssim
\frac{1}{\varepsilon}\|N_j(t,x,u_{j,\varepsilon},\overline{u_{j,\varepsilon}})\|_{H^{2\kappa}(Q)}
\\
&\lesssim
\sum\limits_{\substack{\sigma\geq 1\\
2\le \sigma+\beta\le M}}
\varepsilon^{\sigma+\beta-1}\|B_{j,\sigma\beta}\|_{H^{2\kappa}(Q)}\|f\|^{\sigma+\beta}_{ H^{2\kappa+\frac{3}{2}}}.
\end{align*}  
This ends the proof of \eqref{eq_first_linearization} by letting $\varepsilon$ go to zero. 
\end{proof}

\subsection{An integral identity}
We are ready to show the following identity.
\begin{prop}\label{prop_int_id_par_A_q} 
Let $f_0\in \widetilde{\mathcal{H}}^{2\kappa + \frac{3}{2}}_0(\Sigma)$. If  $\Lambda^\sharp_{A_1,q_1,N_1} =\Lambda^\sharp_{A_2,q_2,N_2}$ on $\mathcal{D}_\delta (\Sigma)$,
then
\begin{align}\label{eq_int_id_par_A_q}
0 = \int_Q \LC 2i \widetilde A\cdot\nabla + i\nabla \cdot \widetilde{A}-(|A_1|^2-|A_2|^2)+ \widetilde q\RC v_1 \overline{v_0}\rd x\rd t,
\end{align}
for any $v_1\in H^{2\kappa}(Q)$ satisfying \eqref{eq_first_linearization_v} with $j=1$ and $v_0\in H^{2\kappa}(Q)$ satisfying 
\begin{align}\label{eq_prop51_IBVP}
\left\{\begin{array}{rrll}
(i\p_t + \Delta_{A_2}  + q_2) v_0&=& 0 &\hbox{ in }Q,\\
v_0&=& f_0&\hbox{ on } \Sigma,\\
v_0&=& 0 &\hbox{ on } \{t=T\}\times \Omega.
\end{array}  \right.
\end{align}
\end{prop}
\begin{proof}
Due to the hypothesis $\Lambda^\sharp_{A_1,q_1,N_1} =\Lambda^\sharp_{A_2,q_2,N_2}$ on $\mathcal{D}_\delta (\Sigma)$ and Lemma~\ref{lemma_first_linerization}, one can deduce that the DN maps for the linearized MSE with $(A_j,q_j)$ are the same, that is,
\[
\Lambda^{\sharp, L}_{A_1,q_1}(f)=\Lambda^{\sharp, L}_{A_2,q_2}(f)
\]
for any $f\in \widetilde{\mathcal{H}}^{2\kappa + \frac{3}{2}}_0(\Sigma)$. Together with the hypothesis
$A_j=0$ on $(0,T)\times \mathcal{O}$,
it implies
\begin{align*}
\p_\nu v_1|_{\Sigma^\sharp} = \p_\nu v_2|_{\Sigma^\sharp}.
\end{align*}

Recall $\mathcal{O}\subset \overline{\Omega}$ is an open neighborhood of $\p\Omega$. Let $\mathcal{O}_j$, $j =1,\,2,\,3$, be three open subsets of $\mathcal{O}$ such that they are also neighborhood of $\p\Omega$ and $\overline{\mathcal{O}}_{j+1}\subset\mathcal{O}_j$, $\overline{\mathcal{O}}_j\subset \mathcal{O}$. Denote
\[
\Omega_j = \Omega\setminus \overline{\mathcal{O}}_j,\quad Q_j = (0,T)\times \Omega_j.
\]
We introduce a cut-off function $\chi\in C^\infty(\overline{\Omega})$ with $0\le \chi\le 1$ such that
\begin{align}\label{eq_def_cutoff_chi}
\chi= \begin{cases}
0,& \hbox{ in } \mathcal{O}_3\\
1,& \hbox{ in } \Omega_2.
\end{cases}
\end{align}

We denote $\widetilde{v} :=v_1-v_2$, which then solves the following IBVP:
\begin{align*}
\left\{\begin{array}{rrll}
L_{A_2,q_2} \widetilde{v} &=&  \LC -2i \widetilde A\cdot\nabla - i\nabla \cdot \widetilde{A} + (|A_1|^2-|A_2|^2) - \widetilde q\RC v_1 =: F   & \hbox{ in }Q,\\
\widetilde{v} &=&  0 &\hbox{ on } \Sigma,\\
\widetilde{v} &=& 0 &\hbox{ on } \{t=0\}\times \Omega,
\end{array} \right.  
\end{align*}
and satisfies $\p_\nu \widetilde{v}|_{\Sigma^\sharp} = 0$.
Here $F\in H^{2\kappa-1}(Q)$ 
and 
$F=0$ on $(0,T)\times\mathcal{O}$ due to $A = q = 0$ on $(0,T)\times\mathcal{O}$.
By letting $\gamma\rightarrow \infty$ in the unique continuation estimate in \cite[Corollary 1]{Bella-Fraj2020}, we derive
\begin{align}\label{prop_71_unique_continuation}
\|\widetilde{v}\|_{L^2((0,T)\times (\Omega_3\setminus\Omega_2))} = 0.
\end{align}

Next we set ${v}^\sharp = \chi \widetilde v$, then $v^\sharp= 0$ in $(0,T)\times\mathcal{O}_3$ and therefore, $v^\sharp|_{\Sigma}=\p_\nu v^\sharp|_{\Sigma}=0$. We have $v^\sharp$ solves the following IBVP:
\begin{align}\label{eq_ajoint_tilde_v 0}
\left\{\begin{array}{rrll}
L_{A_2,q_2}v^\sharp &=& F + [\Delta,\chi]\widetilde{v} & \hbox{ in }Q,\\
v^\sharp &=&  0 &\hbox{ on } \Sigma,\\
v^\sharp &=& 0 &\hbox{ on } \{t=0\}\times \Omega,
\end{array} \right.  
\end{align}
where we have used the fact that $\chi F=F$ 
since $\widetilde A =0$ and $\widetilde q = 0$ in $(0,T)\times\mathcal{O}$.

Let $v_0\in H^{2\kappa}(Q)$ be the solution to the adjoint IBVP $L_{A_2,q_2}v_0= 0$ with $v_0(T,\cdot) = 0$ in $\Omega$.
Multiplying \eqref{eq_ajoint_tilde_v 0} by $\overline{v_0}$, integrating over $Q$ and performing the integration by parts, we get
\begin{align}\label{eq_prop71_int_id_par}
0 = \int_Q \LC F + [\Delta,\chi]\widetilde{v} \RC \overline{v_0} \rd x \rd t.
\end{align}
Here we have used the conditions that $v^\sharp(0,\cdot)=0$ and $v_0(T,\cdot)=0$ in $\Omega$ and also $v^\sharp|_{\Sigma}=\p_\nu v^\sharp|_{\Sigma}=0$. 
From \eqref{prop_71_unique_continuation}, we get $\widetilde{v}=0$ almost everywhere (a.e.) in $(0,T)\times(\Omega_3\setminus \Omega_2)$ and thus $[\Delta,\chi]\widetilde{v}=0$ in $Q$.
Applying it in \eqref{eq_prop71_int_id_par}, the identity \eqref{eq_int_id_par_A_q} follows.
\end{proof}

\subsection{Recovery of $A$ and $q$}
Equipped with the integral identity, we are now ready to recover $A$ and $q$ by applying the GO solutions in Section~\ref{sec:go nonconcentrated} with {$m\geq 2\kappa +2$ and $N>2m$} to ensure sufficient regularity of the constructed GO solutions and the decay of their remainder terms $R_\rho$ as $\rho\rightarrow \infty$. Let $\omega\in \mathbb{S}^{n-1}$  and $\Phi(t,x) =\rho(x\cdot\omega-\rho t)$. We take the solutions $v_1\in H^m(Q)$ to \eqref{eq_first_linearization_v} with $j=1$ and $v_0\in H^m(Q)$ to the adjoint equation \eqref{eq_prop51_IBVP}: 
\begin{align}\label{recover A GO}
v_1 = e^{i\Phi(t,x)}\LC V_0^{(1)} + \sum_{k=1}^N \rho^{-k} V_k^{(1)}\RC+ R_{\rho}^{(1)},\quad 
v_0 = e^{i\Phi(t,x)}\LC V_0^{(0)} + \sum_{k=1}^N \rho^{-k} V_{k}^{(0)}\RC+ R_{\rho}^{(0)},
\end{align}
with $v_1(0,\cdot)=v_0(T,\cdot)=0$ in $\Omega$. Their amplitudes $V_k^{(1)}$ and $V_k^{(0)}$, $k\geq 0$, are as in \eqref{eq_amplitude_v_0_non}, \eqref{eq_amplitude_v_l_non} corresponding to $A_1$ and $A_2$, respectively, with $c(t)\equiv 1$. In particular, the leading terms are
\begin{align*}
V_0^{(1)} (t,x) := \zeta(t) \Theta (t,x) e^{i\int_0^\infty A_1(t,x+s\omega)\cdot\omega\rd s} \quad \hbox{and}\quad V_0^{(0)} (t,x) :=  \zeta(t) e^{i\int_0^\infty A_2(t,x+s\omega)\cdot\omega\rd s},
\end{align*}
where $\xi\in \omega^\perp$ and we choose $\theta=\widetilde{A}$ here so that
$$
\Theta(t,x) = \eta\cdot \nabla(e^{-i(t\tau+x\cdot\xi)}e^{-i\int_\R\omega\cdot \widetilde{A}(t,x+s\omega)\rd s}).
$$
\begin{theorem}\label{thm:recover A}  
Suppose all the hypotheses in Theorem \ref{thm_main} are satisfied.
Then $A_1=A_2$ and $q_1=q_2$. 
\end{theorem}
\begin{proof}
Let $f:=v_1|_{\Sigma}$ and $f_0:=v_0|_{\Sigma}$ be in $\widetilde{\mathcal{H}}_0^{2\kappa+\frac{3}{2}}(\Sigma)$. From  $\Lambda^\sharp_{A_1,q_1,N_1} =\Lambda^\sharp_{A_2,q_2,N_2} $ on $\mathcal{D}_\delta (\Sigma)$ and Proposition~\ref{prop_int_id_par_A_q}, we obtain the integral identity \eqref{eq_int_id_par_A_q}. 
By substituting $v_1$ and $v_0$ above into \eqref{eq_int_id_par_A_q}, we obtain 
\begin{align}\label{eq_int_id_par_A_q 1}
0 &=2i \int_Q (\widetilde A\cdot\nabla v_1)\overline{v}_0\rd x \rd t + \int_Q \LC i\nabla \cdot \widetilde{A} -(|A_1|^2-|A_2|^2)+ \widetilde q\RC  v_1\overline{v}_0\rd x \rd t =:I_1+I_2.  
\end{align}
From Proposition~\ref{prop:GO_Nonconcentrated}, we know that $V_j=O(1)$ and $R_\rho = O(\rho^{-N+2m})$ in $H^{m}$ norm with $N>2m$. Since $H^m(Q)$ is a Banach algebra, using Sobolev embedding $H^m\subset C^1(\overline{Q})$ when $m>\frac{n+1}{2}+1$, we can derive that 
$$
\lim_{\rho\rightarrow \infty} \rho^{-1}I_2 =0.
$$
Since the gradient in $I_1$ will gain extra $\rho$ from the phase $\Phi$, the term $I_1$ can be splitted into
\begin{align*}
I_1=2i \int_Q (\widetilde A\cdot\nabla v_1)\overline{v}_0\rd x \rd t = -2\rho \int_Q (\omega\cdot\widetilde{A}) V^{(1)}_0\overline{V^{(0)}_0}\rd x \rd t + J,
\end{align*}
where $J$ has the same order of $\rho$ as $I_2$ and thus
$$
\lim_{\rho\rightarrow \infty} \rho^{-1}J = 0. 
$$
Next, we multiply \eqref{eq_int_id_par_A_q 1} by ${\rho}^{-1}$ and let $\rho\to \infty$ to deduce that
\begin{align*}
0 &= \int_Q (\omega\cdot\widetilde{A}) V^{(1)}_0\overline{V^{(0)}_0}\rd x \rd t \\
&= \int_Q (\omega\cdot\widetilde{A})\zeta^2(t)\Theta(t,x) e^{i\int_0^\infty \widetilde{A}(t,x+s\omega)\cdot\omega\rd s}  \rd x \rd t \\
&=\int_\R \int_{\R^n}  (\omega\cdot\widetilde{A})\zeta^2(t)\Theta(t,x) e^{i\int_0^\infty \widetilde{A}(t,x+s\omega)\cdot\omega\rd s}  \rd x \rd t,
\end{align*}
where we used the fact that $\widetilde{A}$ vanishes outside $\Omega$ and $\zeta\in C^\infty_0(\R)$ in the last identity.

Now we decompose $x\in \R^n$ as $x = x_\perp + r \omega$, where $r = x\cdot\omega$ and $x_\perp : = x-r \omega\in\omega^\perp$, and use the fact $\Theta(t,x) =\Theta(t,x_\perp)$ since $\omega\cdot \nabla \Theta= 0$. By change of variable $s\rightarrow r+s$, this results in
\begin{align}\label{ID tilde A}
0 & = \int_{\R}\int_{\omega^\perp}\int_{\R} (\omega\cdot \widetilde A)(t,x_\perp+ r  \omega)\zeta^2(t) \Theta(t,x_\perp) 
e^{i\int_{r}^\infty \omega\cdot\widetilde{A}(t,x_\perp+s\omega)\rd s}  \,dr \rd x_\perp  \rd t \notag\\
&= i\int_{\R}\int_{\omega^\perp}\zeta^2(t)  \Theta(t,x_\perp) 
\LC 1-e^{i\int_\R \omega\cdot\widetilde{A}(t,x_\perp+s\omega)\rd s}\RC \rd x_\perp  \rd t \notag \\
&=  i\int_{\R}\zeta^2(t)e^{-it\tau}\int_{\omega^\perp}  \LC\eta\cdot \nabla(e^{-i x_\perp\cdot\xi}e^{-i\int_\R\omega\cdot \widetilde{A}(t,x_\perp+s\omega)\rd s})\RC \LC 1-e^{i\int_\R \omega\cdot\widetilde{A}(t,x_\perp+s\omega)\rd s}\RC \rd x_\perp  \rd t,
\end{align}
where in the second identity we used the following fact:
\begin{align*}
    \int_\R (\omega\cdot \widetilde A)(t,x_\perp+r\omega)e^{i\int_{r}^\infty\omega\cdot\widetilde{A}(t,x_\perp+s\omega)\rd s}\,dr
&= i \int_\R \p_{r} \LC e^{i\int_{r}^\infty\omega\cdot\widetilde{A}(t,x_\perp+s\omega)\rd s}\RC\,dr\\
&= i \LC 1-e^{i\int_\R \omega\cdot\widetilde{A}(t,x_\perp+s\omega)\rd s}\RC.
\end{align*}
We first note that for a $H^1$ function $f$, one can decompose $\nabla f$ as
$$
\nabla f= \nabla_\perp f + (\omega\cdot\nabla f)\omega
$$
such that
$$
\eta\cdot\nabla f = \eta\cdot (\nabla_\perp f + (\omega\cdot\nabla f)\omega) = \eta\cdot\nabla_\perp f,
$$
when we take $\eta\in \mathbb{S}^{n-1}$ and $\eta\cdot\omega =0$.
Moreover, since $\widetilde{A}(t,\cdot)$ vanishes outside $\Omega\subset B(0,R)$ for some $R>0$,
$$1-e^{i\int_\R \omega\cdot\widetilde{A}(t,x_\perp+s\omega)\rd s} =0 \quad \hbox{ for }x_\perp\in \omega^\perp\cap (\R^n\setminus B(0,R)).$$
With this, by the integration by parts, we obtain
\begin{align*}
& \int_{\omega^\perp}  \LC\eta\cdot \nabla_{\perp}(e^{-i x_\perp\cdot\xi}e^{-i\int_\R\omega\cdot \widetilde{A}(t,x_\perp+s\omega)\rd s})\RC \LC 1-e^{i\int_\R \omega\cdot\widetilde{A}(t,x_\perp+s\omega)\rd s}\RC \rd x_\perp  \\
=\, &i\int_{\omega^\perp}   e^{-i x_\perp\cdot\xi} \eta\cdot\nabla_\perp \LC \int_\R \omega\cdot\widetilde{A}(t,x_\perp+s\omega)\,ds \RC \rd x_\perp.
\end{align*}
Thus, \eqref{ID tilde A} becomes
\begin{align*}
0&= \int_\R \zeta^2(t) e^{-it\tau}\int_{\omega^\perp}   e^{-i x_\perp\cdot\xi} \eta\cdot\nabla_\perp \LC \int_\R \omega\cdot\widetilde{A}(t,x_\perp+s\omega)\,ds \RC \rd x_\perp \rd t\\
&= \int_\R \zeta^2(t) e^{-it\tau} \int_{\R^n} e^{-i x\cdot\xi} \eta\cdot \nabla (\omega\cdot \widetilde{A})(t,x)\rd x \rd t\\
&= i|\xi| \omega\cdot \widehat{\zeta^2 \widetilde{A}}(\tau,\xi),
\end{align*}
by taking $\eta={\xi\over |\xi|}$.
Here $\widehat f(\tau,\xi) : = \int_{\R^{n+1}}e^{-i(t\tau+x\cdot\xi)}f(t,x) \rd x \rd t$ is the Fourier transform of $f$. Let $e_\ell\in \mathbb{S}^{n-1}$, $\ell=1,\cdots, n-1$, such that $\{{\xi\over |\xi|}, e_1, \cdots, e_{n-1}\}$ form an orthonormal basis for $\R^n$. By using $\widehat{\zeta^2\widetilde A}(\tau,\xi)\cdot \xi=0$ due to the hypothesis $\nabla\cdot\widetilde A = 0$ in Theorem~\ref{thm_main}, we may write 
\[
\widehat{\zeta^2\widetilde A}(\tau,\xi) = \LC \widehat{\zeta^2\widetilde A}(\tau,\xi)\cdot{\xi\over |\xi|}\RC {\xi\over |\xi|}+ \sum_{\ell=1}^{n-1}\LC \widehat{\zeta^2\widetilde A}(\tau,\xi)\cdot e_{\ell}\RC e_\ell=\sum_{\ell=1}^{n-1}\LC \widehat{\zeta^2\widetilde A}(\tau,\xi)\cdot e_{\ell}\RC e_\ell.
\]
Combining with the above fact $\omega\cdot \widehat{\zeta^2 \widetilde{A}}(\tau,\xi)=0$ with $\omega = e_\ell$, $\ell = 1,\cdots, n-1$, it implies $\widehat{\zeta^2\widetilde A}(\tau,\xi)= 0$.  
Therefore, we obtain $\zeta^2\widetilde A= 0$ using the injectivity of the Fourier transform. By taking $h$ in $\zeta(t)$ small enough, we have $\widetilde{A}\equiv 0$ in $Q$.

To recover $q$, first we apply $\widetilde A = 0$ to \eqref{eq_int_id_par_A_q} and get
\[
0 = \int_Q \widetilde q v_1 \overline{v_0} \rd x \rd t.
\]
We then substitute GO solutions \eqref{recover A GO} into the above equation so that only the leading order term remains as $\rho\rightarrow \infty$, namely,
\[
0 = \int_Q \widetilde q V^{(1)}_0\overline{V^{(0)}_0}\rd x \rd t.
\]
Since $\widetilde q$ plays the same role as $\omega\cdot\widetilde{A}$ in the proof of Theorem~\ref{thm:recover A}, by following a similar argument, one can derive that $\widehat{\zeta^2 \widetilde{q}}=0$. This implies $\widetilde{q}\equiv 0$, which completes the proof.
\end{proof}

\section{Recovery of the nonlinear coefficient}\label{sec:nonlinear}
The main focus of this section is to recover the nonlinear coefficient from the DN map. In the previous section, we have proved the unique determination of the linear coefficients and therefore we denote
$$
A:=A_1=A_2,\quad q:=q_1=q_2.
$$
We recover $B_{\sigma\beta}$ inductively.

\subsection{Recovery of $B_{\sigma\beta}$ with $\sigma+\beta=2$}
We first perform second-order linearization by 
taking the boundary data $\sum_{\ell=1}^{2}\varepsilon_\ell f_\ell$, where $f_\ell\in \widetilde{\mathcal{H}}^{2\kappa + \frac{3}{2}}_0(\Sigma)$ and $|\varepsilon_\ell|<1$ for $\ell = 1,2$. Denote $\varepsilon = (\varepsilon_1,\varepsilon_{2})$. With sufficiently small $|\varepsilon|$, by Proposition \ref{prop_nonlinear_wellposed}, there exists a unique solution $u_j \in \mathcal{H}_0^{2\kappa}(Q)$, $j= 1,2$
to the following IBVP:
\begin{align}\label{eq_higher_linearization_IBVP 0}
\left\{\begin{array}{rrll}
(i\p_t+\Delta_A+q) u_j&=& N_j(t,x,u_j,\overline{u_j}) &\hbox{ in }Q,\\
u_j&=& \sum_{\ell=1}^{2}\varepsilon_\ell f_\ell &\hbox{ on } \Sigma,\\
u_j&=& 0 &\hbox{ on } \{t=0\}\times \Omega,
\end{array}  \right.  
\end{align}
such that $\|u_j \|_{H^{2\kappa}(Q)}\lesssim \sum_{\ell=1}^{2}\|\varepsilon_\ell f_\ell\|_{ H^{2\kappa+\frac{3}{2}}(\Sigma)}$.

By Proposition \ref{prop_linear_boundary}, there exists a unique solution $v_{\ell}\in \mathcal{H}_0^{2\kappa}(Q)$ to the IBVP:
\begin{align}\label{eq_higher_linearization_IBVP_v}
\left\{\begin{array}{rrll}
(i\p_t+\Delta_A+q)v_{\ell} &=& 0 &\hbox{ in }Q,\\
v_{\ell} &=&  f_\ell &\hbox{ on } \Sigma,\\
v_{\ell} &=& 0 &\hbox{ on } \{t=0\}\times \Omega,
\end{array}  \right.
\end{align}
satisfying $\|v_{\ell}\|_{H^{2\kappa}(Q)}\lesssim \|{f}_\ell\|_{ H^{2\kappa+\frac{3}{2}}(\Sigma)}$. 
Recall that $\mathcal{H}_0^{2\kappa}(Q)$ is a Banach algebra since $\kappa>{n+1\over 2}$ 
and, therefore, 
\[ 
\mathcal{N}_j:= 
2B_{j,20}v_1v_2+ B_{j,11}(v_1\overline{v_2}+\overline{v_1}v_2) \in \mathcal{H}^{2\kappa}_0(Q).
\]
By Proposition \ref{prop_linear_est_higher}, there exists a unique solution $w_j\in \mathcal{H}_0^{2\kappa}(Q)$ to the IBVP:
\begin{equation}\label{eq_higher_linearization_IBVP_w}
\begin{aligned}
\left\{\begin{array}{rrll}
(i\p_t+\Delta_A+q) w_j &=& \mathcal{N}_j  &\hbox{ in }Q,\\
w_j &=& 0 &\hbox{ on } \Sigma,\\
w_j &=& 0 &\hbox{ on } \{t=0\}\times \Omega,
\end{array}  \right.
\end{aligned}
\end{equation}
satisfying $\|w_j\|_{H^{2\kappa}(Q)}\lesssim \max\{\|B_{j,20}\|_{H^{2\kappa}(Q)},\|B_{j,11}\|_{H^{2\kappa}(Q)}\}\prod_{\ell=1}^{2} \|v_{\ell}\|_{H^{2\kappa}(Q)}$.

We now state the second-order linearization lemma.

\begin{lemma}\label{lemma_higher_linerization}
Let $u_j$, $v_{\ell}$, and $w_j$ be defined as above. Then
\begin{align}\label{eq_higher_linearization_v}
\p_{\varepsilon_\ell} u_j |_{\varepsilon = 0} = v_{\ell},\quad \hbox{in} \quad H^{2\kappa}(Q), 
\end{align}
for $\ell= 1,2$ and $j=1,2$,
and
\begin{align}\label{eq_higher_linearization_w}
\p_{\varepsilon_1}\p_{\varepsilon_2}u_j|_{\varepsilon=0} = w_j, \quad \hbox{in} \quad H^{2\kappa}(Q). 
\end{align}
\end{lemma} 
\begin{proof}
Following the same argument as in the proof of Lemma \ref{lemma_first_linerization} gives \eqref{eq_higher_linearization_v}. Hence, it remains to show \eqref{eq_higher_linearization_w}.

Applying $\p_{\varepsilon_1}\p_{\varepsilon_{2}}|_{\varepsilon=0}$ to \eqref{eq_higher_linearization_IBVP 0}, we get
\begin{align*}
\left\{\begin{array}{rrll}
(i\p_t+\Delta_A+q) (\p_{\varepsilon_1}\p_{\varepsilon_{2}}  u_j|_{\varepsilon=0})&=& \p_{\varepsilon_1}\p_{\varepsilon_{2}} N_j(t,x,u_j,\overline{u_j}) |_{\varepsilon=0}
&\hbox{ in }Q,
\\
\p_{\varepsilon_1}\p_{\varepsilon_{2}}  u_j|_{\varepsilon=0}&=& 0&\hbox{ on } \Sigma,
\\
\p_{\varepsilon_1}\p_{\varepsilon_{2}}  u_j|_{\varepsilon=0}&=& 0 &\hbox{ on } \{t=0\}\times \Omega.
\end{array}  \right.  
\end{align*}
Using \eqref{eq_higher_linearization_v} and $u_j=0$ for trivial boundary data, 
we have the following fact 
\begin{align*}
\p_{\varepsilon_1}\p_{\varepsilon_{2}} \LC \sum\limits_{\substack{\sigma\geq 1\\
2\le \sigma+\beta\le M}}
B_{j,\sigma\beta}(t,x)u_j^{\sigma}\overline{u_j}^{\beta} \RC|_{\varepsilon=0}=\mathcal{N}_j.
\end{align*}
Thus, $\p_{\varepsilon_1}\p_{\varepsilon_{2}}  u_j|_{\varepsilon=0}=w_j \in \mathcal{H}^{2\kappa}_0(Q)$ by the well-posedness of \eqref{eq_higher_linearization_IBVP_w}.
\end{proof}

Next, we derive the following integral identity. Recall that $\widetilde{B}_{\sigma\beta}=B_{1,\sigma\beta}-B_{2,\sigma\beta}$ in \eqref{DEF:tilde coefficient}.

\begin{prop}\label{prop_int_id_B}
Let $f_0\in \widetilde{\mathcal{H}}^{2\kappa + \frac{3}{2}}_0(\Sigma)$. If $\Lambda^\sharp_{A_1,q_1,N_1}=\Lambda^\sharp_{A_2,q_2,N_2}$, then
\begin{align}\label{eq_int_id_B}
0 = \bigintsss_Q  \LC 2\widetilde{B}_{20}v_1v_2+ \widetilde{B}_{11}(v_1\overline{v_2}+\overline{v_1}v_2) \RC\overline{v_0} \rd x \rd t,
\end{align}
for any $v_\ell\in H^{2\kappa}(Q)$, $\ell=1,2$, satisfying \eqref{eq_higher_linearization_IBVP_v} and $v_0\in H^{2\kappa}(Q)$ satisfying 
\begin{align}\label{eq_prop61_IBVP}
\left\{\begin{array}{rrll}
(i\p_t + \Delta_{A}  + q) v_0&=& 0 &\hbox{ in }Q,\\
v_0&=& f_0&\hbox{ on } \Sigma,\\
v_0&=& 0 &\hbox{ on } \{t=T\}\times \Omega.
\end{array}  \right.
\end{align}
\end{prop}
\begin{proof}
Let $w_j$ be defined as above and denote $\widetilde{w} :=w_1-w_2$, which then satisfies
\begin{align}\label{eq_higher_linearization_IBVP 1}
\left\{\begin{array}{rrll}
(i\p_t+\Delta_A+q) \widetilde{w} &=& \widetilde{\mathcal{N}}&\hbox{ in }Q,\\
\widetilde{w}  &=& 0&\hbox{ on } \Sigma,\\
\widetilde{w}   &=& 0 &\hbox{ on } \{t=0\}\times \Omega,
\end{array}  \right.
\end{align}
where
\[
\widetilde{\mathcal{N}}:=2\widetilde{B}_{20}v_1v_2+ \widetilde{B}_{11}(v_1\overline{v_2}+\overline{v_1}v_2) .
\]
From  $\Lambda^\sharp_{A_1,q_1,N_1}=\Lambda^\sharp_{A_2,q_2,N_2}$ on $\mathcal{D}_\delta (\Sigma)$, we can derive
$$
\p_\nu w_1|_{\Sigma^\sharp} = \p_\nu w_2|_{\Sigma^\sharp}.
$$
Also, since $v_\ell\in \mathcal{H}^{2\kappa}_0(Q)$ and $\widetilde{B}_{\sigma\beta}$ is smooth, the function $\widetilde{\mathcal{N}} \in \mathcal{H}^{2\kappa}_0(Q)$ 
and $\widetilde{\mathcal{N}}= 0$ on $(0,T)\times\mathcal{O}$.
By unique continuation estimate in \cite[Corollary 1]{Bella-Fraj2020}, we get 
\begin{align}\label{eq_unique_conti_w}
\|\widetilde{w}\|_{L^2((0,T)\times (\Omega_3\setminus\Omega_2))} = 0.
\end{align}

Following Proposition~\ref{prop_int_id_par_A_q}, we take $\chi$ defined by \eqref{eq_def_cutoff_chi} and set 
$w^\sharp=\chi w\in \mathcal{H}^{2\kappa}_0(Q)$, which solves the following IBVP:
\begin{align}\label{eq_ajoint_tilde_v}
\left\{\begin{array}{rrll}
(i\p_t+\Delta_A+q) w^\sharp &=& \widetilde{\mathcal{N}}+ [\Delta,\chi]\widetilde{w}  & \hbox{ in }Q,\\
w^\sharp &=&  0 &\hbox{ on } \Sigma,\\
w^\sharp &=& 0 &\hbox{ on } \{t=0\}\times \Omega.
\end{array} \right.  
\end{align}
Multiplying \eqref{eq_ajoint_tilde_v} by $\overline{v_0}$, integrating over $Q$ and performing integration by parts, we obtain 
\[
0 = \int_Q \LC \widetilde{\mathcal{N}}+ [\Delta,\chi]\widetilde{w} \RC \overline{v_0} \rd x \rd t.
\]
Here we have used that $w^\sharp(0,\cdot) = v_0(T,\cdot)=0$ in $\Omega$ and $v^\sharp|_{\Sigma}=\p_\nu v^\sharp|_{\Sigma}=0$, thanks to the definition of $\chi$. Together with \eqref{eq_unique_conti_w}, the equation above implies \eqref{eq_int_id_B} by arguing similarly as in Proposition \ref{prop_int_id_par_A_q}. 
\end{proof}

Now we are ready to recover $B_{\sigma\beta}$ by the following density argument.

\begin{prop}\label{recoverB_2}
If \eqref{eq_int_id_B} holds 
for any $v_\ell\in H^{2\kappa}(Q)$, $\ell=1,2$, satisfying \eqref{eq_higher_linearization_IBVP_v} and $v_0\in H^{2\kappa}(Q)$ satisfying \eqref{eq_prop61_IBVP}, 
then 
\[
\widetilde{B}_{20}=\widetilde{B}_{11}=0 \quad \hbox{in}\quad Q.
\]
\end{prop}
\begin{proof}
Fix $x_0\in \Omega$. Following Section \ref{sec:go concentrated}, for $m\ge2\kappa+2$ and $N>2m$, we take the GO solutions $v_\ell$ to \eqref{eq_higher_linearization_IBVP_v}, $\ell=1,2$, and $v_0$ to the adjoint \eqref{eq_prop61_IBVP}: 
\begin{align}\label{eq_6_go}
v_\ell(t,x) = e^{i\Phi_\ell(t,x)}\LC V^{(\ell)}_0+\sum_{k=1}^N \rho^{-k} V^{(\ell)}_{k}\RC+ R^{(\ell)}_{\rho}, 
\quad 
v_0(t,x) = e^{i\Phi_0(t,x)}\LC V^{(0)}_0+\sum_{k=0}^N \rho^{-k} V^{(0)}_{k}\RC+ R^{(0)}_{\rho}
\end{align}
with $v_\ell(0,\cdot)=v_0(T,\cdot)=0$ in $\Omega$. Here the phase functions
\begin{equation}\label{eq_prop52_Phi}
\Phi_\ell(t,x) = \rho({x\cdot\omega_\ell} - \rho |\omega_\ell|^2t), \quad \ell = 0,1,2,
\end{equation}
for some $0\neq\omega_\ell\in\R^n$.
Their amplitudes ${V^{(\ell)}_0},\,V^{(\ell)}_k$, are as in \eqref{eq_amplitude_v_0_con} and \eqref{eq_amplitude_v_l_non} corresponding to $\omega_\ell$. In particular, their leading terms are
\begin{equation}\label{eq_prop52_amplitude}
V_0^{(\ell)} = \iota (t) \prod_{a=1}^{n-1}\chi\LC \frac{(x-x_0)\cdot\alpha^{(\ell)}_a}{\delta} \RC e^{i\int_0^\infty A(t,x+s\omega_\ell)\cdot\omega_\ell\rd s},\quad \ell = 0,1,2,
\end{equation}
where $\alpha^{(\ell)}_a\in\R^n$ are chosen such that $
\{ {\omega_\ell\over |\omega_\ell|}, \alpha_1^{(\ell)},\cdots, \alpha_{n-1}^{(\ell)}\}
$
is an orthonormal basis of $\R^n$.

We now insert the GO solutions, given by \eqref{eq_6_go}, into the integral identity \eqref{eq_int_id_B}.
With $m\geq 2\kappa +2$ and $N>2m$, by using $V_k^{(\ell)}=O(1)$ and $R_\rho^{(\ell)} =O(\rho^{-N+2m})$ in $H^m$ norm as $\rho\rightarrow\infty$, 
we reduce the integral identity \eqref{eq_int_id_B} to 
\begin{multline}\label{ID:B2}
0 = \bigintsss_Q  2\widetilde{B}_{20}
e^{i(\Phi_1+\Phi_2-\Phi_0)}
V_0^{(1)} V_0^{(2)}\overline{V_0^{(0)}}
\\
+
\widetilde{B}_{11}\LC
e^{i(-\Phi_1+\Phi_2-\Phi_0)}
\overline{V_0^{(1)}} V_0^{(2)}\overline{V_0^{(0)}}+e^{i(\Phi_1-\Phi_2-\Phi_0)}
{V_0^{(1)}} \overline{V_0^{(2)}}\overline{V_0^{(0)}}\RC
\rd x\rd t 
+O(\rho^{-1}).
\end{multline}
Here we have used Sobolev embedding $H^m(Q)\subset C^1({\overline{Q}})$ when $m>\frac{ n+1 }{2}+1$ and the fact that all terms that contain 
$R^{(\ell)}_\rho$ are at most of order $O(\rho^{-1})$.
For instance, for all $k_1,k_2\ge 0$,
\[
\LV \bigintsss_Q  2\widetilde{B}_{20}e^{i(\Phi_1+\Phi_2-\Phi_0)}V^{(1)}_{k_1}V^{(2)}_{k_2}\overline{R^{(0)}_\rho}\rd x\rd t\RV \lesssim\|\widetilde{B}_{20}\|_{C(Q)} \|V^{(1)}_{k_1}\|_{C({\overline{Q}})} \|V^{(2)}_{k_2}\|_{C({\overline{Q}})} \|R^{(0)}_{\rho}\|_{C({\overline{Q}})} \lesssim O(\rho^{-1})
\]
since 
$\|V^{(\ell)}_k\|_{C({\overline{Q}})}\le\|V^{(\ell)}_k\|_{H^m(Q)} =O(1)$ and $\|R^{(\ell)}_\rho\|_{C({\overline{Q}})}\le \|R^{(\ell)}_\rho\|_{H^m(Q)} = O(\rho^{-N+2m})\le O(\rho^{-1})$.

To recover $\widetilde{B}_{20}$, we take $\omega_\ell$ such that 
\begin{align*}
\omega_1+\omega_2 &= \omega_0,
\\
|\omega_1|^2+|\omega_2|^2 &=|\omega_0|^2,
\end{align*}
which can be fulfilled, for instance, if we take $\omega_1$ and $\omega_2$ perpendicular.
With this choice of $\omega_\ell$, $\ell=0,\,1,\,2$, the exponent in $e^{i(\Phi_1+\Phi_2-\Phi_0)}$ vanishes, and the exponents in $e^{i(-\Phi_1+\Phi_2-\Phi_0)}$ and $e^{i(\Phi_1-\Phi_2-\Phi_0)}$ don't vanish. To see this, if $-\Phi_1+\Phi_2-\Phi_0=0$ holds, we then get $\omega_1=0$, which is a contradiction.
If $\Phi_1-\Phi_2+\Phi_0=0$ holds, we then get $\omega_1=0$, which is also a contradiction.

Moreover, we observe that those integrals in \eqref{ID:B2} with nontrivial exponent are of the order $O(\rho^{-1} )$. 
It is sufficient to show for instance,
\[
\int_Q \widetilde{B}_{11}
e^{i(-\Phi_1+\Phi_2-\Phi_0)}
\overline{V_0^{(1)}} V_0^{(2)}\overline{V_0^{(0)}}
\rd x\rd t = O(\rho^{-1}).
\]
Let us denote $f: = \widetilde{B}_{11}
\overline{V_0^{(1)}} V_0^{(2)}\overline{V_0^{(0)}}\in C^1({\overline{Q}})$ and $\Phi: = {-\Phi_1+\Phi_2-\Phi_0}$ with 
$$
\nabla'\Phi = \rho\LC\rho(|\omega_1|^2-|\omega_2|^2+|\omega _3|^2),(-\omega_1+\omega _2-\omega_3)\RC\neq (0,0),
$$
where $\nabla' := (\p_t,\nabla_x)$. 
Since it is possible that the first component of $\nabla'\Phi$ vanishes, we only have
$|\nabla' \Phi|=O( \rho )$. Note that $e^{i\Phi} = \frac{-1}{|\nabla'\Phi|^2}\LA i\nabla'\Phi, \nabla' e^{i\Phi}\RA$. Since $f$ is compactly supported in $Q$, we now use integration by parts and get  
\[
\LV
\int_Q fe^{i\Phi}\rd x \rd t \RV 
=\LV \int_Q f\frac{1}{|\nabla'\Phi|^2}\LA i\nabla'\Phi,\nabla' e^{i\Phi}\RA \rd x \rd t \RV
= \LV \int_Q e^{i\Phi} \frac{1}{|\nabla'\Phi|^2}\LA i\nabla'\Phi, \nabla' f \RA \rd x \rd t \RV
=O(\rho^{-1}).
\] 
Similar arguments can be applied to prove the remaining terms in \eqref{ID:B2} with nontrivial exponent are of order $O(\rho^{-1})$.

Therefore, by sending $\rho\rightarrow\infty$ in \eqref{ID:B2}, we arrive at
\[
0 = \bigintsss_Q  \widetilde{B}_{20}
V_0^{(1)} V_0^{(2)}\overline{V_0^{(0)}} \rd x\rd t.
\]
For each $t_0\in (0,T)$, since $\iota (t)\in C^\infty_0(0,T)$ in the definition of $V_0^{(\ell)}$ is arbitrary, one can further deduce that 
\[
0 = \bigintsss_\Omega  \widetilde{B}_{20}(t_0,\cdot)
V_0^{(1)} V_0^{(2)}\overline{V_0^{(0)}}(t_0,\cdot) \rd x,
\]
with $V^{(\ell)}_0(t_0,x_0)\neq0$. 
Note that based on $V_0^{(\ell)}$, the integrand is localized in a $\delta$-neighborhood of $x_0$, we get by multiplying the above equation by $\delta^{-n}$ and sending $\delta\to 0$ that 
\[
\widetilde{B}_{20}(t_0,x_0)=0.
\]

To recover $\widetilde{B}_{11}$, we then take $\omega_\ell$ such that 
\begin{align*}
-\omega_1+\omega_2 &= \omega_0,
\\
-|\omega_1|^2+|\omega_2|^2 &=|\omega_0|^2.
\end{align*} 
Following a similar argument as above, we can derive $\widetilde{B}_{11}(t_0,x_0)=0$.
Since $(t_0, x_0)$ is arbitrary, we get $\widetilde B_{20}=\widetilde B_{11}= 0$ in $Q$.
\end{proof}

\subsection{Recovery of $B_{\sigma\beta}$ with $3\le\sigma+\beta\le  M$}

We will recover 
$\widetilde{B}_{\sigma\beta}$, $\sigma+\beta=m$ inductively based on $m$. The case $m=2$ has been settled. Assume that the following holds:
\begin{equation}\label{eq_5_induction}
\widetilde{B}_{\sigma\beta}=0,\quad 2\le \sigma+\beta\le m-1.
\end{equation}
We now prove that $\widetilde{B}_{\sigma\beta}=0$, for $\sigma+\beta= m$.

Let $f_\ell\in \widetilde{\mathcal{H}}^{2\kappa + \frac{3}{2}}_0(\Sigma)$ and $|\varepsilon_\ell|<1$ for $\ell = 1,\cdots, m$. 
Denote $\varepsilon = (\varepsilon_1,\cdots, \varepsilon_{m})$. With sufficiently small $|\varepsilon|$, by Proposition \ref{prop_nonlinear_wellposed}, there exists a unique solution $u_j \in \mathcal{H}_0^{2\kappa}(Q)$, $j= 1,\,2$, 
to the following IBVP:
\begin{align}\label{eq_higher_linearization_IBVP_m}
\left\{\begin{array}{rrll}
(i\p_t+\Delta_A+q) u_j&=& N_j(t,x,u_j,\overline{u_j}) &\hbox{ in }Q,\\
u_j&=& \sum_{\ell=1}^{m}\varepsilon_\ell f_\ell &\hbox{ on } \Sigma,\\
u_j&=& 0 &\hbox{ on } \{t=0\}\times \Omega,
\end{array}  \right.  
\end{align}
such that $\|u_j \|_{H^{2\kappa}(Q)}\lesssim \sum_{\ell=1}^{m}\|\varepsilon_\ell f_\ell\|_{ H^{2\kappa+\frac{3}{2}}(\Sigma)}$.

{
To proceed, let us first show by induction within induction, call it subduction, that for any $1\le k\le m-1$,   
\begin{equation}\label{eq_5_subinduction_1}
\p_{\varepsilon_{\ell_1}}\cdots\p_{\varepsilon_{\ell_k}}u_1|_{\varepsilon=0} = \p_{\varepsilon_{\ell_1}}\cdots\p_{\varepsilon_{\ell_k}}u_2|_{\varepsilon=0}.
\end{equation}
Here we take $\ell_i\neq\ell_j$ and $\ell_1,\,\cdots,\, \ell_k\in\{1,\,2,\,\cdots,\,m\}$. 
The case $k=1$ is settled by following the same argument as in Lemme \ref{lemma_higher_linerization}, which uses 
Proposition \ref{prop_linear_boundary} and $u_j = 0$ for trivial boundary data.
Thus, we have  
$$
v_\ell := \p_{\varepsilon_\ell}u_1|_{\varepsilon=0}=\p_{\varepsilon_\ell}u_2|_{\varepsilon=0}.
$$
Let us assume that \eqref{eq_5_subinduction_1} holds for $k= 1,\,\cdots,\, K-1$, for some $K\le m-1$. 
By applying $\p_{\varepsilon_{\ell_1}}\cdots\p_{\varepsilon_{\ell_{K}}}$ to \eqref{eq_higher_linearization_IBVP_m} and then evaluating at $\varepsilon=0$, we have 
\begin{equation}\label{eq_5_k_diff}
\left\{\begin{array}{rrll}
(i\p_t+\Delta_A+q) \LC\p_{\varepsilon_{\ell_1}}\cdots\p_{\varepsilon_{\ell_{K}}}u_j|_{\varepsilon=0}\RC &=& \mathcal{N}_{j,K}+ \mathcal{R}_{j,K-1} &\hbox{ in }Q,\\
\p_{\varepsilon_{\ell_1}}\cdots\p_{\varepsilon_{\ell_{K}}}u_j|_{\varepsilon=0}&=& 0 &\hbox{ on } \Sigma,\\
\p_{\varepsilon_{\ell_1}}\cdots\p_{\varepsilon_{\ell_{K}}}u_j|_{\varepsilon=0}&=& 0 &\hbox{ on } \{t=0\}\times \Omega,
\end{array}  \right.     
\end{equation}
where 
\begin{align*}
\mathcal{N}_{j,K}&= 
\p_{\varepsilon_1}\cdots\p_{\varepsilon_{K}}\LC \sum\limits_{\substack{\sigma\ge 1\\
\sigma+\beta=K}}
B_{j,\sigma\beta}(t,x)u_j^{\sigma}\overline{u_j}^{\beta}\RC \Bigg|_{\varepsilon =0}
\end{align*}
and remainder terms
$\mathcal{R}_{j,K-1}$ contains only a polynomial of the terms $A$, $q$, $B_{j, \sigma\beta}$, $\sigma+\beta=k$, and $\p_{\varepsilon_{\ell_1}}\cdots\p_{\varepsilon_{\ell_k}}u_j|_{\varepsilon=0}$ for $k\le K-1$. Here we also used the fact $u_j= 0$ for trivial boundary data. 
By our induction \eqref{eq_5_induction} and \eqref{eq_5_subinduction_1}, 
we can deduce that
$$
\mathcal{R}_{1,K-1} = \mathcal{R}_{2,K-1}\quad   \hbox{and} \quad   \mathcal{N}_{1,K} = \mathcal{N}_{2,K}. 
$$

This then yields that 
\[
\left\{\begin{array}{rrll}
(i\p_t+\Delta_A+q) \LC\p_{\varepsilon_{\ell_1}}\cdots\p_{\varepsilon_{\ell_{K}}}u_1|_{\varepsilon=0}-\p_{\varepsilon_{\ell_1}}\cdots\p_{\varepsilon_{\ell_{K}}}u_2|_{\varepsilon=0}\RC &=& 0&\hbox{ in }Q,\\
\LC\p_{\varepsilon_{\ell_1}}\cdots\p_{\varepsilon_{\ell_{K}}}u_1|_{\varepsilon=0}-\p_{\varepsilon_{\ell_1}}\cdots\p_{\varepsilon_{\ell_{K}}}u_2|_{\varepsilon=0}\RC&=& 0 &\hbox{ on } \Sigma,\\
\LC\p_{\varepsilon_{\ell_1}}\cdots\p_{\varepsilon_{\ell_{K}}}u_1|_{\varepsilon=0}-\p_{\varepsilon_{\ell_1}}\cdots\p_{\varepsilon_{\ell_{K}}}u_2|_{\varepsilon=0}\RC&=& 0 &\hbox{ on } \{t=0\}\times \Omega,
\end{array}  \right.  
\]
only has trivial solution due to well-posed result in Proposition \ref{prop_linear_boundary}. Hence, \eqref{eq_5_subinduction_1} holds for $K$, which completes the subinduction step.
}

Let us now continue with the main induction argument and apply $\p_{\varepsilon_{\ell_1}}\cdots\p_{\varepsilon_{\ell_{m}}}$ to \eqref{eq_higher_linearization_IBVP_m} evaluated at $\varepsilon=0$. 
By \eqref{eq_5_subinduction_1}, we already have $\mathcal{R}_{1,m-1}=\mathcal{R}_{2,m-1}$.
Denote $w_j := \p_{\varepsilon_{\ell_1}}\cdots\p_{\varepsilon_{\ell_{m}}}u_j|_{\varepsilon=0}$ and $\widetilde{w}:=w_1-w_2$, which solves 
\begin{equation}\label{eq_higher_linearization_IBVP_w_m}
\begin{aligned}
\left\{\begin{array}{rrll}
(i\p_t+\Delta_A+q) \widetilde{w} &=& \mathcal{N}_{1,m}-\mathcal{N}_{2,m}  &\hbox{ in }Q,\\
\widetilde{w} &=& 0 &\hbox{ on } \Sigma,\\
\widetilde{w} &=& 0 &\hbox{ on } \{t=0\}\times \Omega.
\end{array}  \right.
\end{aligned}
\end{equation}

Let $\pi:\{1,2,\cdots,m\}\rightarrow \{1,2,\cdots,m\}$ be a permutation function. We denote the set of all permutations of $m$ elements by $S_{m}$.
Let us denote $\mathcal{N}_{m}: = \mathcal{N}_{1,m}-\mathcal{N}_{2,m}$ and rewrite it as follows
\[
\mathcal{N}_{m}:=\sum_{\sigma=1}^{m}\widetilde{B}_{\sigma(m-\sigma)} \sum_{\pi\in S_m}v_{\pi(1)}\cdots v_{\pi(\sigma)}\overline{v_{\pi(\sigma+1)}\cdots v_{\pi(m)}}.
\]

Next, the following general integral identity holds by applying a similar argument as the proof of Proposition \ref{prop_int_id_B}. 

\begin{prop}\label{prop_int_id_B_m}
If $\Lambda^\sharp_{A_1,q_1,N_1}=\Lambda^\sharp_{A_2,q_2,N_2}$ on $\mathcal{D}_\delta (\Sigma)$, then
\begin{align}\label{eq_int_id_B_m}
0 = \bigintsss_Q  \LC  \sum_{\sigma=1}^m\widetilde{B}_{\sigma(m-\sigma)} \sum_{\pi\in S_m}v_{\pi(1)}\cdots v_{\pi(\sigma)}\overline{v_{\pi(\sigma+1)}\cdots v_{\pi(m)}} \RC\overline{v_0} \rd x \rd t,
\end{align}
for any $v_\ell\in H^{2\kappa}(Q)$, $\ell=1,\cdots, m$, satisfying \eqref{eq_higher_linearization_IBVP_v} and $v_0\in H^{2\kappa}(Q)$ satisfying \eqref{eq_prop61_IBVP}.
\end{prop}

Now we are ready to recover $B_{\sigma\beta}$ by the following general density argument.

\begin{prop}\label{recoverB_m}
Let $m\ge2$.
If \eqref{eq_int_id_B_m} holds 
for any $v_\ell\in H^{2\kappa}(Q)$, $\ell=1,\,\cdots,\,m$, satisfying \eqref{eq_higher_linearization_IBVP_v} and $v_0\in H^{2\kappa}(Q)$ satisfying \eqref{eq_prop61_IBVP}, 
then for $1\le\sigma\le m$,
\[
\widetilde{B}_{\sigma(m-\sigma)}=0 \quad \hbox{in}\quad Q.
\]
\end{prop}

\begin{proof}
{For a fixed $x_0\in\Omega$,} we take $v_\ell$ and $v_0$ as in the proof of Proposition \ref{recoverB_2}. 
Following a similar argument, we obtain from the integral identity \eqref{eq_int_id_B_m} that 
\begin{equation}\label{eq_prop53_1}
O(\rho^{-1}) = \bigintsss_Q    \sum_{\sigma=1}^m\widetilde{B}_{\sigma(m-\sigma)} \sum_{\pi\in S_m}
\mathcal{E}_{\sigma,\pi}
V_0^{\pi(1)}\cdots V_0^{\pi(\sigma)}\overline{V_0^{\pi(\sigma+1)}\cdots V_0^{\pi(m)}V_0^{(0)}}  \rd x \rd t,
\end{equation}
where $\mathcal{E}_{\sigma,\pi} =e^{i(\Phi_{\pi(1)}+\cdots+\Phi_{\pi(\sigma)}-\Phi_{\pi(\sigma+1)}-\cdots-\Phi_{\pi(m)}-\Phi_{\pi(0)})}$, and $V_0^{(\ell)}$ and $\Phi_\ell$ are as in \eqref{eq_prop52_amplitude} and \eqref{eq_prop52_Phi} for $\ell = 0,\,\cdots,\, m$.

{\bf Step 1.} We first recover $\widetilde{B}_{1(m-1)}$. In doing so, we write the right-hand side of \eqref{eq_prop53_1} as the sum of three terms, 
\begin{align*}
\RNum{1} &= \bigintsss_Q    \widetilde{B}_{1(m-1)} \sum\limits_{\substack{\pi\in S_m\\ \pi(1)=1}}
\mathcal{E}^{(1)}_{1,\pi}
V_0^{(1)}\overline{V_0^{\pi(2)}\cdots V_0^{\pi(m)}V_0^{(0)}}  \rd x \rd t,
\\
\RNum{2} &= \bigintsss_Q    \widetilde{B}_{1(m-1)} \sum\limits_{\substack{\pi\in S_m\\ \pi(1)\neq 1}}
\mathcal{E}^{(2)}_{1,\pi}
V_0^{\pi(1)}\overline{V_0^{\pi(2)}\cdots V_0^{\pi(m)}V_0^{(0)}}  \rd x \rd t,
\\
\RNum{3} &= \bigintsss_Q \sum\limits_{\substack{1\le\sigma\le m\\
\sigma\neq 1}}  \widetilde{B}_{\sigma(m-\sigma)} \sum\limits_{\substack{\pi\in S_m}}
\mathcal{E}^{(3)}_{\sigma,\pi}
V_0^{\pi(1)}\cdots V_0^{\pi(\sigma)}\overline{V_0^{\pi(\sigma+1)}\cdots V_0^{\pi(m)}V_0^{(0)}}  \rd x \rd t.
\end{align*}

The exponents in $\mathcal{E}^{(1)}_{1,\pi}$ vanish by taking vectors $\omega_\ell$, $\ell=0,1,\cdots, m$, such that 
\begin{equation}\label{eq_prop53_2}
\begin{aligned}
\omega_1 &= \omega_2+\cdots+\omega_m+\omega_0,\\
|\omega_1|^2 &= |\omega_2|^2+\cdots+|\omega_m|^2+|\omega_0|^2, 
\end{aligned}
\end{equation}
which also guarantee all the exponents in $\mathcal{E}^{(2)}_{1,\pi}$, $\mathcal{E}^{(3)}_{\sigma,\pi}$  don't vanish. To see this, if some exponent in $\mathcal{E}^{(2)}_{1,\pi}$ vanishes, there exists $\pi\in S_m$, $\pi(1)\neq 1$ such that
\begin{equation}\label{eq_prop53_3}
\begin{aligned}
\omega_{\pi(1)} &= \omega_{\pi(2)}+\cdots+\omega_{\pi(m)}+\omega_0,\\
|\omega_{\pi(1)}|^2 &= |\omega_{\pi(2)}|^2+\cdots+|\omega_{\pi(m)}|^2+|\omega_0|^2.
\end{aligned}
\end{equation}
The second equation in \eqref{eq_prop53_2} and \eqref{eq_prop53_3} implies $|\omega_1|^2 = |\omega_{\pi(1)}|^2$, and therefore $w_\ell= 0$ for $0\le \ell\le m$, $\ell\neq 1$ and $\ell\neq \pi(1)$, which is a contradiction.

If some exponent in $\mathcal{E}^{(3)}_{\sigma,\pi}$ vanishes, there exists some $\sigma\neq 1$, $\pi\in S_m$ that 
\begin{equation}\label{eq_prop53_4}
\begin{aligned}
\omega_{\pi(1)}+\cdots+\omega_{\pi(\sigma)} &= \omega_{\pi(\sigma+1)}+\cdots+\omega_{\pi(m)}+\omega_0,\\
|\omega_{\pi(1)}|^2+\cdots+|\omega_{\pi(\sigma)}|^2 &= |\omega_{\pi(\sigma+1)}|^2+\cdots+|\omega_{\pi(m)}|^2+|\omega_0|^2.
\end{aligned}
\end{equation}
By taking the difference of \eqref{eq_prop53_2} and \eqref{eq_prop53_4}, we have either
\begin{enumerate}
\item[(i)]$|\omega_1|^2 = |\omega_{\pi(1)}|^2 + \cdots + |\omega_{\pi(\sigma)}|^2$ if $\pi(j)\neq 1$ for $1\le j\le \sigma$,
\end{enumerate}
or
\begin{enumerate}
\item[(ii)]$0 = |\omega_{\pi(1)}|^2 +\cdots + |\omega_{\pi(j_0-1)}|^2+|\omega_{\pi(j_0+1)}|^2 + \cdots + |\omega_{\pi(\sigma)}|^2$ if $\pi(j_0)=1 $ for some $1\le j_0\le \sigma$.
\end{enumerate}
Case (i) implies $\omega_0=0$, which is a contradiction.
Case (ii) implies $\omega_{\pi(1)}= \cdots = \omega_{\pi(j_0-1)}=\omega_{\pi(j_0+1)}= \cdots=\omega_{\pi(\sigma)}=0$, which is also a contradiction.

Following the same argument as in the proof of Proposition \ref{recoverB_2}, we can show all terms with nontrivial exponent are of order $O(\rho^{-1})$. Therefore, by sending $\rho\to \infty$, the only surviving term is $\RNum{1}$ and we arrive at 
\[
0 = \bigintsss_Q   (m-1)!  \widetilde{B}_{1(m-1)} 
V_0^{(1)}\overline{V_0^{(2)}\cdots V_0^{(m)}V_0^{(0)}}  \rd x \rd t.
\]
{For each $t_0\in (0,T)$, note that $\iota (t)\in C^\infty_0(0,T)$ in the definition of $V_0^{(\ell)}$ is arbitrary so that we can reduce the above integral to
\[
0 = \bigintsss_\Omega   (m-1)!  \widetilde{B}_{1(m-1)} (t_0,\cdot)
V_0^{(1)}\overline{V_0^{(2)}\cdots V_0^{(m)}V_0^{(0)}}(t_0,\cdot)  \rd x,
\]
with $V^{(\ell)}(t_0,x_0)\neq 0$. Since the integrand is localized in a $\delta$-neighborhood of $x_0$,} we get by multiplying the above equation by $\delta^{-n}$ and sending $\delta\to 0$ that 
\[
\widetilde{B}_{1(m-1)}(t_0,x_0) = 0.
\]

\textbf{Step 2.} Lastly, we recover $B_{b(m-b)}$, $2\le b\le m$. In doing so, we write the right-hand side of \eqref{eq_prop53_1} as the sum of three terms, 
\begin{align*}
\RNum{4}&=
\bigintsss_Q    \widetilde{B}_{b(m-b)} \sum\limits_{\substack{\pi\in S_m\\ \{\pi(1),\cdots,\pi(b)\}= \{1,\cdots,b\} }}
\mathcal{E}_{b,\pi}^{(4)}
V_0^{(1)}\cdots V_0^{(b)}\overline{V_0^{\pi(b+1)}\cdots V_0^{\pi(m)}V_0^{(0)}}  \rd x \rd t, \\
\RNum{5}&=
\bigintsss_Q    \widetilde{B}_{b(m-b)} \sum\limits_{\substack{\pi\in S_m\\ \{\pi(1),\cdots,\pi(b)\}\neq\{1,\cdots,b\} }}
\mathcal{E}^{(5)}_{b,\pi}
V_0^{\pi(1)}\cdots V_0^{\pi(b)}\overline{V_0^{\pi(b+1)}\cdots V_0^{\pi(m)}V_0^{(0)}}  \rd x \rd t,
\\
\RNum{6}&=
\bigintsss_Q \sum\limits_{\substack{1\le \sigma\le m\\\sigma\neq b}}   \widetilde{B}_{\sigma(m-\sigma)} \sum_{\pi\in S_m}
\mathcal{E}_{\sigma,\pi}^{(6)}
V_0^{\pi(1)}\cdots V_0^{\pi(\sigma)}\overline{V_0^{\pi(\sigma+1)}\cdots V_0^{\pi(m)}V_0^{(0)}}  \rd x \rd t.
\end{align*}

The exponents in $\mathcal{E}^{(4)}_{b,\pi}$ vanish by taking $\omega_\ell$, $\ell=0,\,1,\,\cdots,\, m$, such that 
\begin{equation}\label{eq_prop53_5}
\begin{aligned}
\omega_1+\cdots +\omega_b &= \omega_{b+1}+\cdots+\omega_m+\omega_0,\\
|\omega_1|^2+\cdots +|\omega_b|^2  &= |\omega_{b+1}|^2+\cdots+|\omega_m|^2+|\omega_0|^2.
\end{aligned}
\end{equation}

We shall need all the exponents in $\mathcal{E}^{(6)}_{\sigma,\pi}$ don't vanish and we prove by contradiction. 
If some exponent in $\mathcal{E}^{(6)}_{\sigma,\pi}$ vanishes, there exists some $\sigma\neq 1$, $\pi\in S_m$ that \eqref{eq_prop53_4} holds.
By taking defference of \eqref{eq_prop53_5} and \eqref{eq_prop53_4}, we have 
\begin{equation}\label{eq_prop53_6}
\sum_{i\in I} \omega_i = \sum_{j\in J} \omega_j
\end{equation}
for some $I$ being a subset of $\{1,\cdots, b\}$ and $J$ being a subset of $\{b+1,\cdots,m\}$. We claim either $I$ is a proper subset of $\{1,\cdots,b\}$ or $J$ is a proper subset of $\{b+1,\cdots,m\}$. If not, we have $\omega_0 = 0$, which is a contradiction. 

In order to conclude a contradiction from \eqref{eq_prop53_6}, we shall further impose the following conditions on $\omega_\ell$:
\begin{enumerate}
\item[(a)] The vectors $\omega_\ell$, $1\le \ell\le b$ are not in the span of $\{\omega_{b+1},\,\cdots,\,\omega_m\}$,
\item[(b)] The vectors $\omega_\ell$, $b+1\le \ell\le m$ are not in the span of $\{\omega_{1},\,\cdots,\,\omega_b\}$,
\item[(c)] Let $\widetilde{I}$ be any proper subset of $\{1,\,\cdots, \,b\}$, then $\sum_{i\in \widetilde{I}}\omega_{i}\neq0$,
\item[(d)] Let $\widetilde{J}$ be any proper subset of $\{{b+1},\,\cdots,\, m\}$, then $\sum_{j\in \widetilde{J}}\omega_{j}\neq0$.
\end{enumerate}

If $I$ is a proper subset of $\{1,\cdots, b\}$, \eqref{eq_prop53_6} 
contradicts with conditions (a) and (c). If $J$ is a proper subset of $\{b+1,\cdots,m\}$, \eqref{eq_prop53_6} contradicts with conditions (b) and (d).

Conditions (a)-(d) together with \eqref{eq_prop53_5} also guarantee all the exponents in $\mathcal{E}^{(5)}_{b,\pi}$ don't vanish. To see this, if some exponent in $\mathcal{E}^{(5)}_{b,\pi}$ vanishes, there exists $\pi\in S_m$, $\{\pi(1),\cdots, \pi(b)\}\notin S_b$ such that
\begin{equation}\label{eq_prop53_7}
\begin{aligned}
\omega_{\pi(1)}+\cdots+ \omega_{\pi(b)} &= \omega_{\pi(b+1)}+\cdots+\omega_{\pi(m)}+\omega_0,\\
|\omega_{\pi(1)}|^2+\cdots+ |\omega_{\pi(b)}|^2 &= |\omega_{\pi(b+1)}|^2+\cdots+|\omega_{\pi(m)}|^2+|\omega_0|^2.
\end{aligned}
\end{equation}
By taking difference of \eqref{eq_prop53_5} and \eqref{eq_prop53_7}, we again have \eqref{eq_prop53_6} for some $I$ being a subset of $\{1,\cdots, b\}$ and $J$ being a subset of $\{b+1,\cdots,m\}$. Then we use the same argument as above.

We provide a set of $\omega_\ell$, $0\le \ell\le m$, such that \eqref{eq_prop53_5} and condition (a)-(d) are satisfied: 
$$
\omega_1 = -(b-1)e_1,\quad \omega_2 = \cdots = \omega_b = e_1;
$$
$$
\omega_{b+1} = \cdots = \omega_{m}=\LC\frac{b(b-1)}{{m-b}+(m-b)^2}\RC^{1/2} e_2,\quad \omega_0= (b-m)\LC\frac{b(b-1)}{{m-b}+(m-b)^2}\RC^{1/2} e_2.
$$
For this construction, 
it's crucial that $b\ge 2$. This is the reason why we discuss the case $B_{1(m-1)}$ separetely.

Therefore, applying a similar argument as in Step 1 yields that 
$$
0= (m-b)!\bigintsss_Q    \widetilde{B}_{b(m-b)} 
V_0^{(1)}\cdots V_0^{(b)}\overline{V_0^{ (b+1)}\cdots V_0^{ (m)}V_0^{(0)}}  \rd x \rd t 
$$
since those with nonvanishing exponents go to zero as $\rho\rightarrow \infty$. 
This then leads to the uniqueness of $B_{b(m-b)}$, $2\leq b\leq m$, for the same reason as Step 1.

\end{proof}

\subsection{Proof of Theorem~\ref{thm_main}}
The linear coefficients $A$ and $q$ are uniquely recovered in Theorem~\ref{thm:recover A}. By applying the induction argument, Proposition~\ref{recoverB_2} and Proposition~\ref{recoverB_m} give the unique determination of the nonlinear coefficients $B_{\sigma\beta}$ for all $2\leq \sigma+\beta\leq M$ with $\sigma\geq 1$. Combining both completes the proof of the main theorem.

\appendix
\section{Inverse problems for the linear magnetic Schr\"odinger equation}\label{sec:appendix}
In this section, we consider the linear MSE $(i\p_t+c(t)\Delta_A+q) v =0$ with the aim of recovering both $c(t)$ and $(A,q)$ from the DN map.

\subsection{Forward problem for the linear MSE and an integral identity}
First, we give a well-posedness result for the IBVP problem.
\begin{prop}\label{prop_app_wellpose}
For any $f\in \mathcal{H}^{\frac{5}{2}}_0(\Sigma)$, $c(t)\in C^1([0,T];\R)$, $A\in C^2(\overline{Q};\R^n)$, and $q\in C^1(\overline{Q};\R)$. The IBVP \eqref{eq_app_IBVP} has a unique solution $v\in H^{1,2}(Q)$ satisfying 
\begin{align}\label{eq_app_wellpose}
\|v\|_{H^{1,2}(Q)}\lesssim \|f\|_{H^{\frac{5}{2}}(\Sigma)}.
\end{align}

\end{prop}
\begin{proof}
In light of \cite[Chapter 4, Theorem 2.3]{lions_non-homogeneous_1972}, there exists $U\in \mathcal{H}^{3}_0(Q)$ such that $U|_{\Sigma} = f$ with 
\begin{align}\label{eq_app_1_1}
\|U\|_{H^3(Q)}\lesssim \|f\|_{H^{\frac{5}{2}}(\Sigma)}.
\end{align}

To show the existence of $v$ to \eqref{eq_app_IBVP}, it suffices to find $\widetilde{v} = v- U\in H^{1,2}(Q)$ and $\widetilde{v}(0,\cdot)=0$ in $\Omega$ such that 
\[
L_{c,A,q}\widetilde{v} = -L_{c,A,q}U =: \widetilde{F}\in H^{1,0}(Q).
\]
By Theorem \ref{thm_linear_wellposedness} (b), since $\widetilde{F}(0,\cdot)=0$ in $\Omega$, there exists $\widetilde{v}\in H^{1,2}(Q)$ satisfying
\begin{align}\label{eq_app_1_2}
\|\widetilde{v}\|_{H^{1,2}(Q)}\lesssim\|\widetilde{F}\|_{H^{1,0}(Q)}.
\end{align}
Inequality \eqref{eq_app_wellpose} is implied by \eqref{eq_app_1_1} and \eqref{eq_app_1_2}.
\end{proof}

The corresponding GO solutions are as follows. 

\begin{prop}\label{prop_app_go}
Suppose {$c(t)\in C^6([0,T];\R)$}, $A\in C^6(\overline{Q};\R^n)$, $q\in C^5(\overline{Q};R)$.
Let $\omega\in\mathbb{S}^{n-1}$. There exists GO solutions to the linear magnetic Schr\"odinger equation $L_{c,A,q}v = 0$ of the form 
$$
v(t,x) = e^{i\Phi(t,x)}  \LC V_0(t,x) +\rho^{-1} V_{1}(t,x) \RC + R_{\rho} (t,x),
$$
{in $H^{1,2}(Q)$} satisfying the initial condition $v(0,\cdot)=0$ in $\Omega$.
Here $V_0,\, V_1\in H^{3}(Q)$ are given by \eqref{eq_amplitude_v_0_non} - \eqref{eq_amplitude_v_l_non} and satisfy 
\begin{equation}\label{eq_cor41_v0_v1}
\begin{aligned}
\|V_0\|_{H^{3}(Q)}\leq C \langle \tau,\xi\rangle^{4} h^{-3}, \quad  \|V_1\|_{H^{3}(Q)}\leq C \langle \tau,\xi\rangle^{6} h^{-4},
\end{aligned}
\end{equation}
for any $\tau\in\R$, $h >0$ small enough and $\xi \in \omega^\perp$, where the constant $C>0$ depends only on $\Omega$ and $T$. 
The remainder term $R_\rho$ satisfies
\begin{align}\label{eq_4_reminder}
\|R_\rho\|_{L^2(0,T;H^1(\Omega))} +\rho \|R_\rho\|_{L^2(Q)} \lesssim \langle \tau,\xi\rangle^{6} h^{-3}. 
\end{align}

In particular, as $\rho\to \infty$, we have 
\begin{align}\label{eq_app_est_V0}
\|V_0\|_{H^3(Q)} = O(1), \quad \|V_1\|_{H^3(Q)} = O(1)
\end{align}
and 
\begin{align}\label{eq_app_est_R}
\|R_\rho\|_{L^2(0,T;H^1(Q))}=O(1), \quad \|R_\rho\|_{L^2(Q)} = O(\rho^{-1}).
\end{align}

A similar argument holds for the adjoint equation $L_{c,A,q}v=0$ with $v(T,\cdot)=0$.
\end{prop}
\begin{proof}
The estimates \eqref{eq_cor41_v0_v1} follow directly from Proposition \ref{prop:GO_Nonconcentrated}. For the existence of the remainder term $R_\rho$, we apply Theorem \ref{thm_linear_wellposedness} and Remark \ref{remark_L2_est} to deduce that there exists a unique solution of $R_\rho\in H^{1,2}(Q)$ such that 
\begin{align}\label{eq_app_A2_proof_1}
\|R_\rho\|_{L^2(Q)}\lesssim \rho^{-1}\|e^{i\Phi}L_{c,A,q}V_1\|_{L^{2}(Q)}\lesssim \rho^{-1}\langle \tau,\xi\rangle^{5} h^{-2},
\end{align}
and
\[
\|R_\rho\|_{H^{1,2}(Q)}\lesssim \rho^{-1}\|e^{i\Phi}L_{c,A,q}V_1\|_{H^{1,0}(Q)}\lesssim \rho\langle \tau,\xi\rangle^{6}  h^{-4} ,
\]
where we have used \eqref{eq_cor41_v0_v1}.
Using interpolation, we get from the above two inequalities that 
\[
\|R_\rho\|_{L^2(0,T;H^1(\Omega))}\lesssim \|R_\rho\|_{L^2(Q)}^{\frac{1}{2}}\|R_\rho\|_{L^2(0,T;H^2(\Omega))}^{\frac{1}{2}} \lesssim \langle \tau,\xi\rangle^{6} h^{-3},
\]
which proves \eqref{eq_4_reminder}.
\end{proof}

The starting point to prove Theorem~\ref{thm_app_c} is the following integral identity.

\begin{prop}\label{prop_app_integral}
If $\Lambda_{c_1,A_1,q_1} =\Lambda_{c_2,A_2,q_2}$, then
\begin{equation}\label{eq_app_integral}
\begin{aligned}
\LA (c_1-c_2)\nabla v_1,\nabla v_0 \RA_{L^2(Q)} =\, & \LA i(c_1A_1-c_2A_2) \cdot\nabla v_1,v_0 \RA_{L^2(Q)} - \LA i(c_1A_1-c_2A_2) v_1, \nabla v_0 \RA_{L^2(Q)}
\\
&- \LA(c_1|A_1|^2-c_2|A_2|^2- q_1+q_2) v_1,v_0 \RA_{L^2(Q)},
\end{aligned}
\end{equation}
where $v_j\in H^{1,2}(Q)$ solves \eqref{eq_app_IBVP} with $L_{c,A,q} = L_{c_j,A_j,q_j}$, $j = 1,2$, and $v_0\in H^{1,2}(Q)$ solves 
\begin{align}\label{eq_app_ajoint}
\left\{\begin{array}{rrll}
L_{c_2,A_2,q_2} v_0  &=& 0  &\hbox{ in } Q,\\
v_0  &=& g  &\hbox{ on } \Sigma,\\
v_0 &=& 0  &\hbox{ on } \{t =T\}\times\Omega,
\end{array}  \right.
\end{align}
for some $g\in \mathcal{H}^\frac{5}{2}_0(\Sigma)$.
\end{prop}
\begin{proof}
Due to the hypothesis $\Lambda_{c_1,A_1,q_1} =\Lambda_{c_2,A_2,q_2}$, one has
\begin{align}\label{eq_app_A2_1}
c_1 \nu\cdot(\nabla+ iA_1)v_1|_{\Sigma} = c_2 \nu\cdot(\nabla+ iA_2)v_2|_{\Sigma}.
\end{align}
Let us denote $\widetilde{v} = v_1-v_2$, which then solves
\begin{align}\label{eq_app_A2_2}
\left\{\begin{array}{rrll}
L_{c_2,A_2,q_2} \widetilde{v}&=&( -L_{c_1,A_1,q_1}+L_{c_2,A_2,q_2}   )v_1 &\hbox{ in } Q,\\
\widetilde{v}  &=& 0  &\hbox{ on } \Sigma,\\
\widetilde{v} &=& 0  &\hbox{ on } \{t =0\}\times\Omega.
\end{array}  \right.
\end{align}
Multiplying \eqref{eq_app_A2_2} by $\overline{v_0}$ and integrating over $Q$, we apply the integration by parts so that the left-hand side becomes 
\begin{equation}\label{eq_app_A2_3}
\begin{aligned}
\LA L_{c_2,A_2,q_2}\widetilde{v}, v_0 \RA_{L^2(Q)} = \LA c_2\nu\cdot(\nabla+{iA_2})\widetilde{v},v_0 \RA_{L^2(\Sigma)}  
\end{aligned}
\end{equation}
since $v_0$ solves \eqref{eq_app_ajoint}, $\widetilde{v}=0$ on $\Sigma$ and $\widetilde{v}(0,\cdot)={v_0(T,\cdot)}=0$ on $\Omega$.
On the other hand, the right-hand side of \eqref{eq_app_A2_2} becomes
\begin{equation}\label{eq_app_A2_4}
\begin{aligned}
&\LA (-L_{c_1,A_1,q_1}+L_{c_2,A_2,q_2})v_1, v_0 \RA_{L^2(Q)}
\\
=&\LA (-c_1\Delta_{A_1}+c_2\Delta_{A_2})v_1,v_0 \RA_{L^2(Q)} + \LA(-q_1+q_2)v_1,v_0 \RA_{L^2(Q)}
\\
=&\LA (c_1-c_2)\nabla v_1, \nabla v_0\RA_{L^2(Q)}
+\LA i(-c_1A_1+c_2A_2)\cdot\nabla v_1, v_0 \RA_{L^2(Q)} + \LA i(c_1A_1-c_2A_2)v_1,\nabla v_0 \RA_{L^2(Q)}
\\
& + \LA (c_1|A_1|^2-c_2|A_2|^2 -q_1+q_2)v_1,v_0 \RA_{L^2(Q)}+ \LA (-c_1\nu\cdot(\nabla+iA_1)+c_2\nu\cdot(\nabla+iA_2))v_1,v_0 \RA_{L^2(\Sigma)}.
\end{aligned}
\end{equation}
Using \eqref{eq_app_A2_1}, the boundary terms on both sides are the same, that is,
that 
\[
\LC c_2\nu\cdot(\nabla+iA_2)\widetilde{v}\RC |_{\Sigma}=\LC (-c_1\nu\cdot(\nabla+iA_1)+c_2\nu\cdot(\nabla+iA_2))v_1\RC|_{\Sigma}. 
\]
Therefore we obtain \eqref{eq_app_integral} from \eqref{eq_app_A2_3} and \eqref{eq_app_A2_4}. 
\end{proof}

\subsection{Proof of Theorem~\ref{thm_app_c}}

We are now ready to prove Theorem~\ref{thm_app_c}.
\begin{proof}
\textbf{Step 1. Recovery of $c$. }
We shall first prove that $c_1(t) = c_2(t)$ for all $t\in (0,T)$ based on singular solutions inspired by \cite{feizmohammadi_inverse_2022}.  

To this end, by applying the contradiction argument, we assume that 
there exists a point $t^*\in (0,T)$ so that $c_1(t^*)> c_2(t^*)$, without loss of generality.
Due to the continuity of $c_1-c_2$, there exist sufficiently small $\varepsilon_0>0$ and $h>0$ such that $[t^*-2h, t^*+2h]\subset (0,T)$
and 
\begin{equation}
\label{eq_app_thm12_contra}  
c_1(t)-c_2(t) > \varepsilon_0 \quad\hbox{ for } t\in [t^*-2h, t^*+2h].
\end{equation}
Let ${\tilde\zeta}(t)\in C^\infty_0(\R)$ be a smooth function supported in $(t^*-2h, t^*+2h)$ such that $0\le \zeta\le 1$ and ${\tilde\zeta}(t) = 1$ on $[t^*-h, t^*+h]$.
Fix $0<r<1$ and $y\in \R^n\setminus\overline{\Omega}$ such that $\dist(y,\Omega) = r$.

We start with the case $n\ge 3$. Consider the function $\Phi_y\in C^\infty(\overline{\Omega})$ given by 
\[
\Phi_y(x) = \frac{1}{n(2-n)d_n}|x-y|^{2-n},
\]
where $d_n$ denotes the volume of the unit ball in $\R^n$. {By Proposition~\ref{prop:estimate Phi} below,} for any $0<\delta<\frac{1}{2}$, the following estimates hold:
\begin{align}\label{eq_app_phi_est}
\|\Phi_y\|_{L^2(\Omega)}\lesssim 
\begin{cases}
r^{2-\frac{n}{2}-\delta}, & n\ge 4\\
O(1), &n=3
\end{cases}
\quad\hbox{and}\quad r^{1-\frac{n}{2}}\lesssim \|\nabla\Phi_y\|_{L^2(\Omega)}\lesssim r^{1-\frac{n}{2}-\delta}.
\end{align}
We set 
\[
\Psi_y(t,x) = {\tilde\zeta}(t)(t)\Phi_y(x), \quad (t,x)\in \overline{Q}.
\]
Let $v_1$ solve \eqref{eq_app_IBVP} with $L_{c,A,q}= L_{c_1,A_1,q_1}$ and $f = \Psi_y$. To construct $v_1$ of the following form
\[
v_1 = \Psi_y + R_1,
\]
it suffices to show the reminder term $R_1$ satisfies 
\begin{align}
\left\{\begin{array}{rrll}\label{eq_app_R2}
L_{c_1,A_1,q_1} R_1  &=& -L_{c_1,A_1,q_1}\Psi_y  &\hbox{ in } Q,\\
R_1  &=& 0  &\hbox{ on } \Sigma,\\
R_1 &=& 0  &\hbox{ on } \{t =0\}\times\Omega.
\end{array}  \right.
\end{align}
Since $\Delta \Phi_y=0$ in $\Omega$ and
\[
L_{c_1,A_1,q_1}\Psi_y = i \tilde\zeta'(t) \Phi_y+2i{c_1} \tilde\zeta (t) A_1\cdot\nabla\Phi_y + (i c_1\nabla\cdot A_1 - c_1|A_1|^2+q_1) \tilde\zeta (t)\Phi_y,
\]
we get that 
\[
\|L_{c_1,A_1,q_1}\Psi_y\|_{L^\infty(0,T;H^{-1}(\Omega))} + \|L_{c_1,A_1,q_1}\Psi_y\|_{H^1(0,T;H^{-1}(\Omega))} \lesssim \|\Phi_y\|_{L^2(\Omega)}.
\]
The estimate above together with Remark \ref{remark_wellposed_c} implies there exists $R_1\in {H^{1,2}(Q)\cap}L^\infty(0,T;H^1(\Omega))$ 
such that $R_1$ solves \eqref{eq_app_R2} and satisfies
\begin{align}\label{eq_app_R1_est}
\|R_1\|_{L^\infty(0,T;H^1(\Omega))}\lesssim \|\Phi_y\|_{L^2(\Omega)}.
\end{align}

Similarly, we can construct $v_0$ solving \eqref{eq_app_ajoint} with $g = \Phi_y$ of the form $v_0= \Psi_y+R_0$, where $R_0$ satisfies 
\begin{align}\label{eq_app_R0_est}
\|R_0\|_{L^\infty(0,T;H^1(\Omega))}\lesssim 
\|\Phi_y\|_{L^2(\Omega)}.
\end{align}
Using \eqref{eq_app_R1_est} and \eqref{eq_app_R0_est}, we have for $j = 0,1$ that
\begin{align}\label{eq_app_vj_est}
\|v_j\|_{L^2(Q)}\lesssim \|\Psi_y\|_{L^2(Q)}+\|R_j\|_{L^2(Q)}\lesssim \|\Phi_y\|_{L^2(\Omega)},
\end{align}
and 
\begin{align}\label{eq_app_d_vj_est}
\|\nabla v_j\|_{L^2(Q)}\lesssim \|\nabla\Psi_y\|_{L^2(Q)}+\|\nabla R_j\|_{L^2(Q)}\lesssim \|\nabla\Phi_y\|_{L^2(\Omega)}+\|\Phi_y\|_{L^2(Q)}\lesssim\|\nabla\Phi_y\|_{L^2(Q)}.
\end{align}

Now we substitute $v_0$ and $v_1$ into the integral identity in Proposition \ref{prop_app_integral} and obtain 
\begin{equation}\label{eq_app_integral_plug}
\begin{aligned}
\LA (c_1-c_2)\nabla \Psi_y,\nabla \Psi_y \RA_{L^2(Q)} = &  \LA i(c_1A_1-c_2A_2) \cdot \nabla v_1,v_0 \RA_{L^2(Q)} - \LA i(c_1A_1-c_2A_2) v_1, \nabla v_0 \RA_{L^2(Q)}
\\
&- \LA (c_1|A_1|^2-c_2|A_2|^2- q_1+q_2) v_1,v_0 \RA_{L^2(Q)}
\\
&-\LA (c_1-c_2)\nabla R_1, \nabla \Psi_y  \RA_{L^2(Q)}  -\LA (c_1-c_2)\nabla \Psi_y, \nabla R_0 \RA_{L^2(Q)}
\\
&-\LA (c_1-c_2)\nabla R_1, \nabla R_0  \RA_{L^2(Q)}.
\end{aligned}
\end{equation}
Since ${\tilde\zeta}$ is supported in $(t^*-2h, t^*+2h)$ and ${\tilde\zeta}=1$ in $[t^*-h, t^*+h]$, with \eqref{eq_app_phi_est}, the left-hand side now is bounded from below by
\[
\int_{t^*-h}^{t^*+h}(c_1-c_2)\int_\Omega |\nabla\Psi_y|^2 \rd x\rd t \gtrsim h \inf_{t\in [t^*-h,t^*+h]}\{c_1(t)-c_2(t)\} r^{2-n}.
\]
To control the right-side of \eqref{eq_app_integral_plug}, we utilize \eqref{eq_app_phi_est}, \eqref{eq_app_R1_est} 
-\eqref{eq_app_d_vj_est} so that it is bounded from above by 
\begin{multline*}
\|\nabla v_1\|_{L^2(Q)} \|v_0\|_{L^2(Q)} + \| v_1\|_{L^2(Q)}\|v_0\|_{H^1(Q)}
+(\|\nabla R_1\|_{L^2(Q)} + \|\nabla R_0\|_{L^2(Q)})\|\nabla\Psi_y\|_{L^2(Q)}
\\
+\|\nabla R_1\|_{L^2(Q)} \|\nabla R_0\|_{L^2(Q)}
\lesssim
\|\Phi_y\|_{L^2(Q)}\|\nabla\Phi_y\|_{L^2(Q)}
\lesssim
\begin{cases}
r^{3-n-2\delta}, &n\ge 4,\\
r^{1-\frac{n}{2}-\delta}, &n=3.
\end{cases}
\end{multline*}
As a result, we derive that 
\[
0<h\inf_{t\in [t^*-h,t^*+h]}\{c_1(t)-c_2(t)\} \lesssim
\begin{cases}
r^{1-2\delta}, & n\ge 4,\\
r^{\frac{1}{2}-\delta}, &n=3,
\end{cases}
\]
which leads to $\inf_{t\in [t^*-h,t^*+h]}\{c_1(t)-c_2(t)\}=0$
by sending $r\to 0$ and noting $h>0$, contradicting to \eqref{eq_app_thm12_contra}. 
Hence, when $n=3$, we have $c_1 = c_2$ for all $t\in(0,T)$.

For the case $n=2$, we denote $x=(x_1, x_2)$, and $\p_2:=\p_{x_2}$. 
Now we take 
$$
\Phi_y(x) =  \frac{1}{2\pi}\ln |x-y|
$$ 
and still let $\Psi_y(t,x)= {\tilde\zeta}(t) \Phi_y(x)$. We consider the solution is of the form 
\[
v_1 = \p_2\Psi_y + R_1.
\]
Following a similar argument as the case $n\geq 3$ and applying estimates in Proposition~\ref{prop:estimate Phi}, we have 
\begin{align*}
0< 2h\inf_{t\in [t^*-h,t^*+h]}\{c_1(t)-c_2(t)\} r^{-2}
&\leq \int_{t^*-h}^{t^*+h}(c_1-c_2)\int_\Omega |\nabla {\p_2} \Psi_y|^2 \rd x\rd t  \\
&\leq \|{\p_2} \Phi_y\|_{L^2(Q)}\|\nabla {\p_2} \Phi_y\|_{L^2(Q)}\leq r^{{-1-2\delta}},\quad 0<\delta<{1\over 2}.
\end{align*}
Multiplying both sides by $r^2$ and letting $r\rightarrow 0$ also lead to a contradiction. This then implies $c_1=c_2$ in the two-dimensional case. 

From now on, we denote
\[
c\equiv c(t) := c_1(t) = c_2(t).
\]

\textbf{Step 2. Recovery of $A$ and $q$.} Now the integral identity \eqref{eq_app_integral} reduces to 
\begin{align*}
0= \, & \LA ic( A_1-A_2) \cdot\nabla v_1,v_0 \RA_{L^2(Q)} - \LA ic( A_1- A_2) v_1, \nabla v_0 \RA_{L^2(Q)}
\\
&- \LA ( c  |A_1|^2- c|A_2|^2 - q_1+q_2) v_1,v_0 \RA_{L^2(Q)}.
\end{align*}
From the hypothesis $\nabla\cdot A_1 =\nabla \cdot A_2$ in $Q$ and $A_1=A_2$ on $\Sigma$, integrating by parts yields that
\begin{equation}\label{eq_app_integral_A}
\begin{aligned}
0=2i \LA c(A_1-A_2)\cdot \nabla v_1,v_0 \RA_{L^2(Q)} 
- \LA (c|A_1|^2-c|A_2|^2- q_1+q_2) v_1,v_0 \RA_{L^2(Q)}.
\end{aligned}   
\end{equation}
Applying Proposition \ref{prop_app_go}, we take $v_1$ and $v_0$ to be the GO solutions in {$H^{1,2}(Q)$} below: 
\begin{align*}
v_1 =  e^{i\Phi(t,x)}\LC V_0^{(1)} +  \rho^{-1} V_{1}^{(1)}\RC+ R_{\rho}^{(1)},\quad 
v_0 = e^{i\Phi(t,x)}\LC V_0^{(0)} +  \rho^{-1} V_{1}^{(0)}\RC+ R_{\rho}^{(0)},
\end{align*}
with $v_1(0,\cdot)= v_0(T,\cdot) = 0$ in $\Omega$.
Their amplitudes $V_j^{(1)}$ and $V_j^{(0)}$, $j= 0,1$, are as in \eqref{eq_amplitude_v_0_non}, \eqref{eq_amplitude_v_l_non} corresponding to $A_1$ and $A_2$, respectively. In particular, the leading terms {in $C^5(\overline{Q})$} are
\begin{align*}
V_0^{(1)} (t,x) :={e^{i {c'(t)(x\cdot \omega)^2 \over 8c(t)^2}}} \zeta(t) \Theta (t,x) e^{i\int_0^\infty A_1(t,x+s\omega)\cdot\omega\rd s} 
\end{align*}
and
\begin{align*}
V_0^{(0)} (t,x) :={e^{i {c'(t)(x\cdot \omega)^2 \over 8c(t)^2}}}  \zeta(t) e^{i\int_0^\infty A_2(t,x+s\omega)\cdot\omega\rd s},
\end{align*}
where $\xi\in \omega^\perp$ and we choose $\theta=\widetilde{A}$ here so that
$$
\Theta(t,x) = \eta\cdot \nabla(e^{-i(t\tau+x\cdot\xi)}e^{-i\int_\R\omega\cdot \widetilde{A}(t,x+s\omega)\rd s}).
$$

We know that $\|V_j^{(k)}\|_{H^3(Q)}=O(1)$, {$\|R_\rho^{(k)}\|_{L^2(Q)} = O(\rho^{-1})$,}
$\|R_\rho^{(k)}\|_{L^2(0,T;H^1(\Omega))} = O(1)$, for
$j = 0,1$ and $k = 0,1$. By a similar argument as in the proof of Theorem \ref{thm:recover A}, we can deduce from \eqref{eq_app_integral_A} that
\[
\LA \omega\cdot 
 \sqrt{c} (A_1-A_2) V_0^{(1)}, V_0^{(0)}\RA_{L^2(Q)} = 0,
\]
which then yields $\sqrt{c} \zeta^2\widetilde A= 0$ due to the injectivity of the Fourier transform.
By taking $h$ in $\zeta(t)$ small enough and using the fact that $c >c_0>0$, it follows that $\widetilde{A}=0$. 

Finally, a similar argument also leads to the unique determination of $q$, which completes the proof of the theorem. 
\end{proof}

The estimates for the function $\Phi_y(x)$ are provided below.
\begin{prop}\label{prop:estimate Phi}
Let $0<\delta<{1\over 2}$.
\begin{itemize}
\item[(a)] For $n\geq 3$, 
$ 
\Phi_y(x) = \frac{1}{n(2-n)d_n}|x-y|^{2-n} 
$ 
satisfies 
\begin{align}\label{eq_app_phi_est_3}
\|\Phi_y\|_{L^2(\Omega)}\lesssim 
\begin{cases}
r^{2-\frac{n}{2}-\delta}, & n\ge 4\\
O(1), &n=3
\end{cases}
\quad \hbox{ and } \quad r^{1-\frac{n}{2}}\lesssim \|\nabla\Phi_y\|_{L^2(\Omega)}\lesssim r^{1-\frac{n}{2}-\delta}.
\end{align}
\item[(b)] For $n=2$, 
$ 
\Phi_y(x) = \frac{1}{2\pi}\ln |x-y| 
$  
satisfies 
\begin{align}\label{eq_app_phi_est dimension two}
\|\p_2\Phi_y\|_{L^2(\Omega)}\lesssim  r^{-{\delta}}
\quad \hbox{ and } \quad r^{-1}\lesssim \|\nabla \p_2\Phi_y\|_{L^2(\Omega)}\lesssim r^{{-1-\delta}}.
\end{align}
\end{itemize}
\end{prop}

\begin{proof}
For a fixed $y\in \R^n\setminus\overline{\Omega}$, we denote $\dist(y,\Omega)=r$, there exists $R>0$ so that the ball $\Omega\subset B(y,R)$ with center at $y$ and radius $R$.

(a)
For $n\ge 3$, $\nabla\Phi_y(x) =
{\frac{1}{nd_n}\frac{x-y}{|x-y|^n}}$.
For $0<\delta<\frac{1}{2}$, direct computations show
\[
\int_\Omega|\Phi_y(x)|^2\rd x\lesssim
\begin{cases}
r^{4-2\delta-n} \int_{B(y,R)}|x-y|^{2\delta-n}\rd x \lesssim r^{4-2\delta-n}, &n\ge 4,\\
\int_{B(y,R)}|x-y|^{-2}dx = O(1), &n=3,
\end{cases}
\]
and 
\[
\int_\Omega |\nabla\Phi_y(x)|^2\rd x \lesssim r^{2-2\delta-n}\int_{B(y,R)}|x-y|^{2\delta-n}\lesssim r^{2-2\delta-n}.
\]
To estimate $\|\nabla\Phi_y\|_{L^2(\Omega)}$ from below, we note that for sufficiently small $r$, there exists $x_0\in\Omega$ such that $|x_0-y|=3r$ and $B(x_0,r)\subset\Omega$. This leads to 
\[
\int_\Omega |\nabla\Phi_y(x)|^2\rd x 
\gtrsim 
\int_{B(x_0,r)\cap\Omega} |x-y|^{2-2n}\rd x 
\gtrsim 
\int_{B(x_0,r)}\LC |x-x_0|+|y-x_0| \RC^{2-2n}
\gtrsim r^{2-n}.
\]
(b)
For $n=2$, $\p_2 \Phi_y(x) = {1\over 2\pi} {x_2-y_2\over |x-y|^2}$. 
Direct computations show
\begin{align*}
\int_\Omega |\p_2 \Phi_y(x)|^2 \,dx \lesssim r^{- {2}\delta} \int_{B(y,R)}|x-y|^{-2+ {2}\delta}\,dx \lesssim r^{-{2}\delta},
\end{align*}
and 
\begin{align*}
\int_\Omega |\nabla \p_2 \Phi_y(x)|^2 \,dx \lesssim r^{-2- {2}\delta} \int_{B(y,R)}|x-y|^{-2+{2} \delta}\,dx \lesssim r^{-{2}- {2}\delta}.
\end{align*}
Using a similar argument as (a) to bound $\|\nabla\p_2\Phi_y\|_{L^2(\Omega)}$ from below, we obtain  
\begin{align*}
\int_\Omega |\nabla \p_2 \Phi_y(x)|^2 \,dx 
\gtrsim  \int_{B(x_0,r)}|x-y|^{-4}\,dx
\gtrsim \int_{B(x_0,r)} (|x-x_0|+|y-x_0|)^{-4}\,dx \gtrsim r^{-2}.
\end{align*} 
\end{proof}

\begin{center}
    Acknowledgements
\end{center}
    
Ru-Yu Lai is partially supported by the National Science Foundation (NSF) through grants DMS-2006731 and DMS-2306221. Gunther Uhlmann is partially supported by the NSF, and a Robert R. and Elaine F. Phelps Endowed Professorship at the University of Washington. 

\bibliography{NLSref}

\begin{thebibliography}{10}

\bibitem{AM}
I.~B. A\"icha and Y.~Mejri.
\newblock Simultaneous determination of the magnetic field and the electric
  potential in the {S}chr\"odinger equation by a finite number of boundary
  observations.
\newblock {\em Journal of Inverse and Ill-posed Problems}, 26:201--209, 2018.

\bibitem{Bella}
M.~Bellassoued.
\newblock Determination of coefficients in the dynamical {S}chr\"odinger
  equation in a magnetic field.
\newblock {\em Inverse Problems}, 33:055009, 2017.

\bibitem{BC}
M.~Bellassoued and M.~Choulli.
\newblock Stability estimate for an inverse problem for the magnetic
  {S}chr\"odinger equation from the dirichlet-to-neumann map.
\newblock {\em Journal of Functional Analysis}, 258:161--195, 2010.

\bibitem{BD}
M.~Bellassoued and D.~Dos Santos~Ferreira.
\newblock Stable determination of coefficients in the dynamical anisotropic
  {S}chr\"odinger equation from the {D}irichlet-to-{N}eumann map.
\newblock {\em Inverse Problems}, 26:125010, 2010.

\bibitem{Bella-Fraj2020}
M.~Bellassoued and O.~B. Fraj.
\newblock Stability estimates for time-dependent coefficients appearing in the
  magnetic {S}chr\"odinger equation from arbitrary boundary measurements.
\newblock {\em Inverse Problems and Imaging}, 14:841--865, 2020.

\bibitem{BKS}
M.~Bellassoued, Y.~Kian, and E.~Soccorsi.
\newblock An inverse problem for the magnetic {S}chr\"odinger equation in
  infinite cylindrical domains.
\newblock {\em Publications of the Research Institute for Mathematical
  Sciences}, 54:679--728, 2018.

\bibitem{Carlo}
C.~Cercignani.
\newblock {\em The Boltzmann Equation and Its Applications}.
\newblock Springer Applied Mathematical Sciences, volume 67, 1988.

\bibitem{CLOP}
X.~Chen, M.~Lassas, L.~Oksanen, and G.~Paternain.
\newblock Detection of {H}ermitian connections in wave equations with cubic
  non-linearity.
\newblock {\em Journal of European Mathematical Society}, 24(7):2191--2232,
  2022.

\bibitem{CKS}
M.~Choulli, Y.~Kian, and E.~Soccorsi.
\newblock Stable determination of time-dependent scalar potential from boundary
  measurements in a periodic quantum waveguide.
\newblock {\em SIAM Journal on Mathematical Analysis}, 47:4536--4558, 2015.

\bibitem{CS}
M.~Cristofol and E.~Soccorsi.
\newblock Stability estimate in an inverse problem for non-autonomous magnetic
  {S}chr\"odinger equations.
\newblock {\em Applicable Analysis}, 90:1499--1520, 2011.

\bibitem{eskin_inverse_2008}
G.~Eskin.
\newblock Inverse problems for the {Schrödinger} equations with time-dependent
  electromagnetic potentials and the {Aharonov}–{Bohm} effect.
\newblock {\em Journal of Mathematical Physics}, 49(2):022105, Feb. 2008.

\bibitem{feizmohammadi_inverse_2022}
A.~Feizmohammadi, Y.~Kian, and G.~Uhlmann.
\newblock An inverse problem for a quasilinear convection–diffusion equation.
\newblock {\em Nonlinear Analysis}, 222:112921, Sept. 2022.

\bibitem{FL2019}
A.~Feizmohammadi and L.~Oksanen.
\newblock An inverse problem for a semi-linear elliptic equation in
  {R}iemannian geometries.
\newblock {\em Journal of Differential Equations}, 269(6):4683--4719, 2020.

\bibitem{haberman2016unique}
B.~Haberman.
\newblock Unique determination of a magnetic {S}chr\"odinger operator with
  unbounded magnetic potential from boundary data.
\newblock {\em International Mathematics Research Notices}, 2018(4):1080--1128,
  2016.

\bibitem{Isakov93}
V.~Isakov.
\newblock On uniqueness in inverse problems for semilinear parabolic equations.
\newblock {\em Archive for Rational Mechanics and Analysis}, 124(1):1--12,
  1993.

\bibitem{john_finite_2016}
V.~John.
\newblock {\em Finite {Element} {Methods} for {Incompressible} {Flow}
  {Problems}}, volume~51 of {\em Springer {Series} in {Computational}
  {Mathematics}}.
\newblock Springer International Publishing, Cham, 2016.

\bibitem{Kang2002}
H.~Kang and G.~Nakamura.
\newblock Identification of nonlinearity in a conductivity equation via the
  {D}irichlet-to-{N}eumann map.
\newblock {\em Inverse Problems}, 18:1079--1088, 2002.

\bibitem{KianSoccorsi}
Y.~Kian and E.~Soccorsi.
\newblock H\"older stably determining the time-dependent electromagnetic
  potential of the {S}chr\"odinger equation.
\newblock {\em SIAM Journal on Mathematical Analysis}, 51:627--647, 2019.

\bibitem{KianTetlow}
Y.~Kian and A.~Tetlow.
\newblock H\"older-stable recovery of time-dependent electromagnetic potentials
  appearing in a dynamical anisotropic {S}chr\"odinger equation.
\newblock {\em Inverse Problems and Imaging}, 4:819--839, 2020.

\bibitem{krupchyk2014uniqueness}
K.~Krupchyk and G.~Uhlmann.
\newblock Uniqueness in an inverse boundary problem for a magnetic
  {S}chr\"odinger operator with a bounded magnetic potential.
\newblock {\em Communications in Mathematical Physics}, 327(3):993--1009, 2014.

\bibitem{KU201909}
K.~Krupchyk and G.~Uhlmann.
\newblock Partial data inverse problems for semilinear elliptic equations with
  gradient nonlinearities.
\newblock {\em Mathematical Research Letters}, 27(6):1801--1824, 2020.

\bibitem{KU2019}
K.~Krupchyk and G.~Uhlmann.
\newblock A remark on partial data inverse problems for semilinear elliptic
  equations.
\newblock {\em Proceedings of the AMS}, 148(2):681–685, 2020.

\bibitem{KLU18}
Y.~Kurylev, M.~Lassas, and G.~Uhlmann.
\newblock Inverse problems for {L}orentzian manifolds and non-linear hyperbolic
  equations.
\newblock {\em Inventiones Mathematicae}, 212(3):781--857, 2018.

\bibitem{lai_partial_2023}
R.-Y. Lai, X.~Lu, and T.~Zhou.
\newblock {P}artial data inverse problems for the nonlinear {S}chr{\"o}dinger
  equation.
\newblock {\em SIAM Journal on Mathematical Analysis}, 56(4), 2024.

\bibitem{LaiUhlmannYang}
R.-Y. Lai, G.~Uhlmann, and Y.~Yang.
\newblock Reconstruction of the collision kernel in the nonlinear {B}oltzmann
  equation.
\newblock {\em SIAM Journal on Mathematical Analysis}, 53(1):1049--1069, 2021.

\bibitem{LaiUhlmannZhou22}
R.-Y. Lai, G.~Uhlmann, and H.~Zhou.
\newblock Recovery of coefficients in semilinear transport equations.
\newblock {\em Archive for Rational Mechanics and Analysis}, 248, 2024.

\bibitem{LY2023}
R.-Y. Lai and L.~Yan.
\newblock Stable determination of time-dependent collision kernel in the
  nonlinear {B}oltzmann equation.
\newblock {\em SIAM Journal on Applied Mathematics}, 84(5), 2024.

\bibitem{LaiTingZhou20}
R.-Y. Lai and T.~Zhou.
\newblock Partial data inverse problems for nonlinear magnetic {S}chr\"odinger
  equations.
\newblock {\em Mathematical Research Letters}, 30(5):1535--1563, 2023.

\bibitem{LLLS201903}
M.~Lassas, T.~Liimatainen, Y.-H. Lin, and M.~Salo.
\newblock Inverse problems for elliptic equations with power type
  nonlinearities.
\newblock {\em Journal de Mathematiques Pures et Appliquees}, 145:44--82, 2021.

\bibitem{LLLS201905}
M.~Lassas, T.~Liimatainen, Y.-H. Lin, and M.~Salo.
\newblock Partial data inverse problems and simultaneous recovery of boundary
  and coefficients for semilinear elliptic equations.
\newblock {\em Revista Matem{\'a}tica Iberoamericana}, 37(4):1553--1580, 2021.

\bibitem{lassas_inverse_2024}
M.~Lassas, L.~Oksanen, S.~K. Sahoo, M.~Salo, and A.~Tetlow.
\newblock Coefficient determination for non-linear {S}chr\"odinger equations on
  manifolds, 2022.
\newblock arXiv:2201.03699.

\bibitem{LUW2018}
M.~Lassas, G.~Uhlmann, and Y.~Wang.
\newblock Inverse problems for semilinear wave equations on {L}orentzian
  manifolds.
\newblock {\em Communications in Mathematical Physics}, 360(2):555--609, 2018.

\bibitem{LiOuyang2022}
L.~Li and Z.~Ouyang.
\newblock Determining the collision kernel in the {B}oltzmann equation near the
  equilibrium.
\newblock {\em Proceedings of the American Mathematical Society},
  151:4855--4865, 2023.

\bibitem{lions_non-homogeneous_1972}
J.~L. Lions and E.~Magenes.
\newblock {\em Non-{Homogeneous} {Boundary} {Value} {Problems} and
  {Applications}}, volume~II.
\newblock Springer Berlin Heidelberg, Berlin, Heidelberg, 1972.

\bibitem{lions_non-homogeneous_1972-1}
J.~L. Lions and E.~Magenes.
\newblock {\em Non-{Homogeneous} {Boundary} {Value} {Problems} and
  {Applications}}, volume~I.
\newblock Springer Berlin Heidelberg, Berlin, Heidelberg, 1972.

\bibitem{malomed_nonlinear_2005}
B.~Malomed.
\newblock Nonlinear {Schrödinger} {Equations}.
\newblock In A.~Scott, editor, {\em Encyclopedia of Nonlinear Science}, pages
  639--643. Taylor \& Francis, New York, 2005.

\bibitem{NSU1995}
G.~Nakamura, Z.~Sun, and G.~Uhlmann.
\newblock Global identifiability for an inverse problem for the {S}chr\"odinger
  equation in a magnetic field.
\newblock {\em Mathematische Annalen}, 303:377--388, 1995.

\bibitem{panchenko2002inverse}
A.~Panchenko.
\newblock An inverse problem for the magnetic {S}chr\"odinger equation and
  quasi-exponential solutions of nonsmooth partial differential equations.
\newblock {\em Inverse Problems}, 18(5):1421--1434, 2002.

\bibitem{pitaevskii_bose-einstein_2003}
L.~Pitaevskii and S.~Stringari.
\newblock {\em Bose-Einstein {Condensation}}.
\newblock Clarendon Press. Clarendon, Oxford, U.K., 2003.

\bibitem{poschel_inverse_1987}
J.~Pöschel and E.~Trubowitz.
\newblock {\em Inverse spectral theory}.
\newblock Number 130 in Pure and applied mathematics. Academic Press, Boston,
  1987.

\bibitem{renardy_introduction_2004}
M.~Renardy and R.~C. Rogers.
\newblock {\em An introduction to partial differential equations}.
\newblock Number~13 in Texts in applied mathematics. Springer, New York, 2nd ed
  edition, 2004.

\bibitem{salo2004inverse}
M.~Salo.
\newblock Inverse problems for nonsmooth first order perturbations of the
  {L}aplacian.
\newblock {\em Annales Academiae Scientiarum Fennicae. Mathematica
  Dissertationes}, 139, 2004.

\bibitem{sun_inverse_1993}
Z.~Sun.
\newblock An inverse boundary value problem for {S}chr\"odinger operators with
  vector potentials.
\newblock {\em Transactions of the American Mathematical Society}, 338(2):953,
  Aug. 1993.

\bibitem{SunU97}
Z.~Sun and G.~Uhlmann.
\newblock Inverse problems in quasilinear anisotropic media.
\newblock {\em American Journal of Mathematics}, 119(4):771--797, 1997.

\bibitem{uhlmann2009calderon}
G.~Uhlmann.
\newblock {E}lectrical impedance tomography and {C}alder\'on's problem.
\newblock {\em Inverse Problems}, 25:123011, 2009.

\end{thebibliography}
\bibliographystyle{abbrv}
\end{document}